\documentclass[a4paper, 11pt, english]{article}

\usepackage[geometry]{ifsym}

\usepackage{pgf}
\usepackage{amsmath,amssymb}
\usepackage{mathrsfs}
\usepackage{bbm}
\usepackage{verbatim}
\usepackage{epsfig}
\usepackage{graphicx}
\usepackage[all]{xy}
\usepackage{color}
\usepackage{import}
\usepackage{bussproofs}
\usepackage{qtree}
\usepackage{tree-dvips}
\usepackage{lingmacros}
\usepackage{amsfonts}
\usepackage{mathdots}
\usepackage{amssymb}
\usepackage{pifont}
\usepackage[safe]{tipa}
\usepackage{newlfont}
\usepackage{latexsym}
\usepackage{multirow}
\usepackage{array,cmll}
\usepackage{textcomp}
\usepackage{accents}
\usepackage{rotating}
\usepackage{logreq}

\usepackage{hyperref}

  \usepackage{amsmath}
  \usepackage{amsthm}
  \usepackage{latexsym}
  \usepackage{amscd}
  \usepackage{amssymb}
  \usepackage{bussproofs}
  \usepackage{txfonts}

  \usepackage{tikz}
  \usepackage[position=top]{subfig}

  \usepackage{leqno}

\usepackage{extpfeil}

\input theorem.sty

\def\fCenter{{\mbox{$\ \vdash\ $}}}

\EnableBpAbbreviations

\newextarrow{\injectb}{5599}{>\relbar\relbar}
\newextarrow{\injectd}{5599}{\relbar\relbar >}

\newcommand{\red}[1]{\textcolor{red}{#1}}
\newcommand{\bfred}[1]{\textbf{\textcolor{red}{#1}}}

\def\aol{\rule[0.5865ex]{1.38ex}{0.1ex}}

\def\pdra{\mbox{$\,{\rotatebox[origin=c]{-90}{\raisebox{0.12ex}{$\pand$}}{\mkern-2mu\aol}}\,$}}
\def\pdla{\mbox{\rotatebox[origin=c]{180}{$\,{\rotatebox[origin=c]{-90}{\raisebox{0.12ex}{$\pand$}}{\mkern-2mu\aol}}\,$}}}

\def\mANDORatom#1{\hbox{\hbox to 0pt{$#1\TriangleUp$\hss}$#1\TriangleDown$}}

\newcommand{\mcAND}{%
\mathrel{\ooalign{\raisebox{-0.39ex}{$\mbox{\TriangleUp}$}\cr\kern4.2pt{\raisebox{-0.13ex}{$\cdot$}}}}}
\newcommand{\mcand}{%
\mathrel{\ooalign{$\vartriangle$\cr\kern1.99pt{\raisebox{-0.17ex}{$\cdot$}}}}}

\newcommand{\mAND}{\raisebox{-0.39ex}{\mbox{\,\TriangleUp\,}}}
\newcommand{\mand}{\vartriangle}

\newcommand{\nAND}{%
\mathrel{\ooalign{$\mbox{\TriangleUp}$\cr\kern0pt$\mbox{\rotatebox[origin=c]{180}{\TriangleUp}}$}}}

\newcommand{\nand}{%
\mathrel{\ooalign{$\vartriangle$\cr\kern0pt$\triangledown$}}}

\newcommand{\mcBAND}{%
\mathrel{\ooalign{\raisebox{-0.39ex}{$\mbox{\FilledTriangleUp}$}\cr\kern4.2pt{\raisebox{-0.13ex}{${\color{white}\cdot}$}}}}}
\def\mBAND{\raisebox{-0.39ex}{\mbox{\,\FilledTriangleUp\,}}}
\def\mband{\mbox{$\mkern+2mu\blacktriangle\mkern+2mu$}}

\newcommand{\mcband}{%
\mathrel{\ooalign{$\blacktriangle$\cr\kern1.99pt{\raisebox{-0.17ex}{${\color{white}\cdot}$}}}}}

\def\mbra{\mbox{$\,-{\mkern-3mu\blacktriangleright}\,$}}

\newcommand{\mcRA}{%
\mathrel{\ooalign{
                  \raisebox{-0.3ex}{$\rotatebox[origin=c]{-90}{$\mbox{{\TriangleUp}}$}$}
                                                                            \cr\kern2.7pt{\raisebox{0.2ex}{$\cdot\mkern1.3mu$}}}}}

\def\mRA{\mbox{\,\raisebox{-0.39ex}{\rotatebox[origin=c]{-90}{\TriangleUp}}\,}}
\newcommand{\mra}{\mbox{$\,-{\mkern-3mu\vartriangleright}\,$}}
\newcommand{\mla}{\mbox{$\,{\vartriangleleft\mkern-8mu-}\,$}}

\newcommand{\mcra}{%
\mathrel{\ooalign{$\,{\vartriangleright\,}$\cr\kern3pt{\raisebox{0ex}{$\cdot$}}}}}

\newcommand{\mcraline}{%
-{\mkern-6mu{\mathrel{\ooalign{$\,{\vartriangleright\,}$\cr\kern3pt{\raisebox{0ex}{$\cdot$}}}}}}}

\newcommand{\mdraline}{%
{\mathrel{\ooalign{$\,{\vartriangleright\,}$\cr\kern3pt{\raisebox{0ex}{$\cdot$}}}}}{\mkern-6mu}-}

\newcommand{\cra}{%
\mathrel{\ooalign{$\,-{\mkern-3mu\vartriangleright\,}$\cr\kern8pt{\raisebox{0ex}{$\cdot$}}}}}

\def\mSRA{\mbox{\,\raisebox{0.43ex}{$\thicksim\mkern-1.3mu$}\rotatebox[origin=c]{3.9999}{\TriangleRight}\,}}

\def\msra{\mbox{$\,\sim{\mkern-8mu\vartriangleright}\,$}}

\def\mBRA{\mbox{\,\raisebox{-0.39ex}{\rotatebox[origin=c]{-90}{\FilledTriangleUp}}\,}}
\newcommand{\mcBRA}{%
\mathrel{\ooalign{
                  \raisebox{-0.3ex}{$\rotatebox[origin=c]{-90}{$\mbox{\FilledTriangleUp}$}$}
                                                                            \cr\kern2.7pt{\raisebox{0.2ex}{${\color{white}\cdot}$}}}}}

\newcommand{\mcbra}{%
\mathrel{\ooalign{$\,-{\mkern-3mu\blacktriangleright\,}$\cr\kern8pt{\raisebox{0ex}{$\cdot$}}}}}

\def\mLA{\mbox{\,\raisebox{-0.39ex}{\rotatebox[origin=c]{90}{\TriangleUp}}\,}}
\newcommand{\mcLA}{%
\mathrel{\ooalign{
                  \raisebox{-0.3ex}{$\rotatebox[origin=c]{90}{$\mbox{\TriangleUp}$}$}
                                                                                     \cr\kern5.5pt{\raisebox{0.2ex}{$\cdot$}}
                                                                                                                              }}}

\def\la{\mbox{$\,\vartriangleleft{\mkern-8mu-}\,$}}

\newcommand{\mcla}{%
\mathrel{\ooalign{$\,{\vartriangleleft\,}$\cr\kern5pt{\raisebox{0ex}{$\cdot$}}}}}

\newcommand{\mclaline}{%
-{\mkern-6mu{\mathrel{\ooalign{$\,{\vartriangleleft\,}$\cr\kern5pt{\raisebox{0ex}{$\cdot$}}}}}}}

\def\nla{\mbox{$\,\vartriangleleft\,$}}

\def\mSLA{\mbox{\,\rotatebox[origin=c]{-3.9999}{\TriangleLeft}\raisebox{0.43ex}{$\mkern-1.3mu\thicksim$}\,}}

\def\msla{\mbox{$\,\vartriangleleft{\mkern-8mu\sim}\,$}}

\def\mBLA{\mbox{\,\raisebox{-0.39ex}{\rotatebox[origin=c]{90}{\FilledTriangleUp}}\,}}
\newcommand{\mcBLA}{%
\mathrel{\ooalign{
                  \raisebox{-0.3ex}{$\rotatebox[origin=c]{90}{$\mbox{\FilledTriangleUp}$}$}
                                                                                     \cr\kern5.5pt{\raisebox{0.2ex}{${\color{white}\cdot}$}}
                                                                                                                                            }}}

\def\mbla{\mbox{$\,\blacktriangleleft\,$}}
\def\bla{\mbox{$\,\blacktriangleleft{\mkern-8mu-}\,$}}

\def\mSBRA{\mbox{\,\raisebox{0.43ex}{$\thicksim\mkern-1.3mu$}\FilledTriangleRight}\,}

\def\msbra{\mbox{$\,\sim{\mkern-8mu\blacktriangleright}\,$}}

\def\mSBLA{\mbox{\,\FilledTriangleLeft\raisebox{0.43ex}{$\mkern-1.3mu\thicksim$}\,}}

\def\msbla{\mbox{$\,\blacktriangleleft{\mkern-8mu\sim}\,$}}

  \numberwithin{equation}{section}

\newcommand{\ls}{\lbrack}
\newcommand{\rs}{\rbrack}
\newcommand{\lc}{\langle}
\newcommand{\rc}{\rangle}

\newcommand{\pand}{\wedge}

\newcommand{\por}{\vee}

\newcommand{\pra}{\rightarrow}
\def\pdra{\mbox{$\,{\rotatebox[origin=c]{-90}{\raisebox{0.12ex}{$\pand$}}{\mkern-2mu\aol}}\,$}}

\def\pdla{\mbox{\rotatebox[origin=c]{180}{$\,{\rotatebox[origin=c]{-90}{\raisebox{0.12ex}{$\pand$}}{\mkern-2mu\aol}}\,$}}}
\newcommand{\gI}{%
\mathrel{\ooalign{$\mbox{T}$\cr\kern0pt$\mbox{\rotatebox[origin=c]{180}{T}}$}}}

\def\aga{\texttt{a}}

\def\agA{{\Large{\texttt{a}}}}
\def\agB{\Large{\texttt{b}}}

\def\aol{\rule[0.5865ex]{1.38ex}{0.1ex}}

\newcommand{\RESagaProxy}{\,\rotatebox[origin=c]{90}{$\{\rotatebox[origin=c]{-90}{$\aga$}\}$}\,}

\newcommand{\RESmbandthreeDia}{\,
$\raisebox{0.4ex}{$
\widehat{\phantom{\aga \!\! \mand_3 \!\! F}}$}
\!\!\!\!\!\!\!\!\!\!\!\!\!\!$
\aga  \blacktriangle_3  F$
\!\!\!\!\!\!\!\!\!\!\!\!\!\!\!\!\!
\raisebox{-0.7ex}{\rotatebox[origin=c]{-180}{$\widehat{\phantom{\aga \!\! \mand_3 \!\! F}}$}}$ }

\newcommand{\RESalphaProxy}{\,\rotatebox[origin=c]{90}{$\{\rotatebox[origin=c]{-90}{$\alpha$}\}$}\,}

\newcommand{\RESalphajProxy}{\,\rotatebox[origin=c]{90}{$\{\rotatebox[origin=c]{-90}{$\!\alpha_{\!j}$}\,\}$}\,}
\newcommand{\RESbetakProxy}{\,\rotatebox[origin=c]{90}{$\{\rotatebox[origin=c]{-90}{$\!\beta_{\!k}$}\,\}$}\,}

\newcommand{\RESbetaProxy}{\,\rotatebox[origin=c]{90}{$\{\rotatebox[origin=c]{-90}{$\beta$}\}$}\,}

\newcommand{\RESalphaBox}{\,\rotatebox[origin=c]{90}{$[\rotatebox[origin=c]{-90}{$\alpha$}]$}\,}

\newcommand{\RESalphaDia}{\,\rotatebox[origin=c]{90}{$\langle\rotatebox[origin=c]{-90}{$\alpha$}\rangle$}\,}

\newcommand{\Bigsemic}{\mbox{\Large {\bf ;}}}

\makeatletter
\newcommand{\WKnowProxy}[2]{%
  {\mathbin{\ooalign{$#1\circ#2 $\cr\hidewidth
   \raise.155ex\hbox{$#1{\scriptstyle{\ast}}#2$}\hidewidth\cr  }}}}
\makeatother

\makeatletter
\newcommand{\BKnowProxy}[2]{%
  {\mathbin{\ooalign{$#1\bullet#2 $\cr\hidewidth
   \raise.155ex\hbox{$#1{\scriptstyle{\color{white}{\ast}}}#2$}\hidewidth\cr  }}}}
\makeatother









\newcommand{\fns}{\footnotesize}
\newcommand{\mc}{\multicolumn}

\newtheorem{proposition}[theorem]{Proposition}

\theoremstyle{definition}
\newtheorem{defn}{Definition}
\newcommand{\commment}[1]{}
\def\mbra{\mbox{$\,-{\mkern-3mu\blacktriangleright}\,$}}
\newcommand{\RESagaDia}{\,\rotatebox[origin=c]{90}{$\langle\rotatebox[origin=c]{-90}{$\aga$}\rangle$}\,}
\newcommand{\RESagaBox}{\,\rotatebox[origin=c]{90}{$[\rotatebox[origin=c]{-90}{$\aga$}]$}\,}

\begin{document}

\title{Multi-type Display Calculus for Dynamic Epistemic Logic}
\author{Sabine Frittella\footnote{Laboratoire d'Informatique Fondamentale de Marseille (LIF) - Aix-Marseille Universit\'{e}.}\;\; Giuseppe Greco\footnote{Department of Values, Technology and Innovation - TU Delft.}\;\; Alexander Kurz\footnote{Department of Computer Science - University of Leicester.}\\ Alessandra Palmigiano\footnote{Department of Values, Technology and Innovation - TU Delft. The research of the second and fourth author has been made possible by the NWO Vidi grant 016.138.314, by the NWO Aspasia grant 015.008.054, and by a Delft Technology Fellowship awarded in 2013 to the fourth author.}\;\; Vlasta Sikimi\'c\footnote{The Institute of Philosophy of the Faculty of Philosophy - University of Belgrade.}}

\maketitle

\begin{abstract}

In the present paper, we introduce a  multi-type display calculus for dynamic epistemic logic, which we refer to as Dynamic Calculus. The display-approach is suitable to modularly chart the space of dynamic epistemic logics on weaker-than-classical propositional base. The presence of types endows the language of the Dynamic Calculus with additional expressivity,  allows for a smooth proof-theoretic treatment, and paves the way towards a general methodology for the design of proof systems for the generality of dynamic logics, and certainly beyond dynamic epistemic logic.  We prove that the Dynamic Calculus adequately captures Baltag-Moss-Solecki's  dynamic epistemic logic, and enjoys Belnap-style cut elimination.
{\em Math.\ Subject Class.\ 2010:} 	03B42,  03B20, 03B60, 03B45, 03F03, 03G10, 03A99. 

\end{abstract}

\tableofcontents

\section{Introduction}

\paragraph{Motivation.} The range of nonclassical logics has been rapidly expanding, driven by influences from other fields which have opened up new opportunities for applications. The logical formalisms which have been developed as a result of this interaction have attracted the interest of a research community wider than the logicians, and their theory has been intensively investigated, especially w.r.t.\ their semantics and computational complexity.

However, most of these logics lack a comparable proof-theoretic development. More often than not, the hurdles preventing a standard proof-theoretic development for these logics are due precisely to the very features which make them suitable for applications, such as e.g.\ their not being closed under uniform substitution, or the existence of certain interactions between logical connectives, which cannot be expressed within the language itself. 

 A case in point is  Baltag-Moss-Solecki's logic of epistemic actions and knowledge (EAK), which is the main focus of the present paper.  The Hilbert-style presentation of EAK prominently features non schematic axioms such as \[[\alpha] p\leftrightarrow (Pre(\alpha)\rightarrow p),\] where the variable $p$ ranges over  atomic propositions, and $Pre(\alpha)$ is a meta-linguistic abbreviation for an arbitrary formula, and axioms such as \[[\alpha][\aga] A\leftrightarrow (Pre(\alpha)\rightarrow \bigwedge\{[\aga][\beta] A\mid \alpha\aga \beta\}),\] in which the extra-linguistic label $\alpha\aga \beta$ expresses the fact that actions $\alpha$ and $\beta$ are indistinguishable for agent $\aga$.

Difficulties posed by features such as these caused the existing proposals of calculi in the literature to be often ad hoc, not easily generalizable e.g.\ to other logics, and more in general lacking a smooth proof-theoretic behaviour. In particular, the difficulty in smoothly transferring results from one logic to another is a problem in itself, since logics such as EAK typically come in large families. Hence, proof-theoretic approaches which uniformly apply to each logic in a given family are in high demand (for an expanded discussion of the existing proof systems for dynamic epistemic logics, see \cite[Section 3]{GAV}).

The problem of the transfer of results, tools and methodologies has been addressed in the proof-theoretic literature for the families of substructural and modal logics, and has given rise to the development of several generalizations of Gentzen sequent calculi (such as hyper-, higher level-, display- or labelled-sequent calculi).

\paragraph{Contribution.} The present paper focuses on the core technical aspects of a proof-theoretic methodology and set-up closely linked to Belnap's display calculi \cite{Belnap}.
Specifically, our main contribution is the introduction of a methodology for the design of display calculi based on   {\em multi-type languages}.
In  the case study provided by EAK,
we start by observing that having to resort to the label $\alpha\aga\beta$ is symptomatic of the fact that the language of EAK  lacks the necessary expressivity to autonomously capture the piece of information encoded in the label. 
%
%
%

In order to provide the desired additional expressivity, we introduce a language in which not only formulas are generated from formulas and actions (as it happens in the symbol $\langle\alpha \rangle A$) and formulas are generated from formulas and agents (as it happens in the symbol $\langle\aga \rangle A$), but also {\em actions} are generated from the interaction between {\em agents} and {\em actions}, which is precisely what the label $\alpha\aga\beta$ is about.

In the multi-type language for EAK introduced in the present paper, each generation step mentioned above is explicitly accounted for via special connectives taking arguments of {\em different types}. In principle, more than one alternative is possible in this respect; our choice for the present setting consists of the following types: $\mathsf{Ag}$ for agents, $\mathsf{Fnc}$ for functional actions, $\mathsf{Act}$ for actions, and $\mathsf{Fm}$ for formulas. Hence, the present setting introduces a separation between functional, i.e.\ deterministic actions, of type $\mathsf{Fnc}$, and  possibly nondeterministic actions, of type $\mathsf{Act}$  (see discussion at the end of section \ref{DC}).

The proposed calculus provides an interesting and in our opinion very promising {\em methodological} platform towards the uniform development of a general proof-theoretic account of all dynamic logics, 
and also, from a purely structurally proof-theoretic viewpoint, for clarifying and sharpening the formulation of  criteria leading to the statement and proof of meta-theoretic results such as  Belnap-style  cut-elimination  (see Section \ref{refinements}).

\paragraph{Structure of the paper.}  
In Section 2, we collect the relevant preliminaries on EAK, display calculi, and the (single-type) display calculus D'.EAK. In Section \ref{sec:multi}, we sketch the general features of the environment of multi-type display calculi,  extend Wansing's definition of properly displayable calculi to the multi-type setting, and  prove the corresponding extension of Belnap's cut elimination metatheorem. In Section \ref{DC}, we  propose a novel display calculus for EAK, which we refer to as {\em Dynamic Calculus}, and which concretely exemplifies the notion of multi-type display calculus. 
In Sections \ref{sec:soundness}-\ref{sec : safe virtual adjoints}, we prove that the Dynamic Calculus adequately captures EAK, and  enjoys  Belnap-style cut elimination. In Section \ref{Conclusion}, we collect some conclusions and indicate further directions.  The routine proofs and derivations are collected in  Section \ref{Appendix}, the appendix.

\section{Preliminaries}
    In the present section, we collect the needed preliminaries: in  \ref{sec:EAK}, we review the logic of epistemic actions and knowledge.
Our presentation  slightly departs from \cite{BMS}, and closely follows \cite{AMM, KP}.\footnote{The account of  EAK developed in \cite{AMM, KP} is specifically tailored to facilitate the dual characterization at the base of the definition of the intuitionistic counterparts of  EAK, which the calculus
introduced in Section \ref{DC} takes as basic.
So for the sake of a tighter presentation we include it here.}  In  \ref{ssec:IEAK}, we briefly review the intuitionistic version of EAK, the axiomatization of which is directly captured in the rules of the calculus introduced in Section \ref{DC}. In  \ref{ssec:DisplayLogic}, we sketch the main relevant features of display calculi. 
In  \ref{ssec:D'.EAK prelim}, we briefly report on the (single-type) display calculus for EAK introduced in \cite{GAV}.
\subsection{The logic of epistemic actions and knowledge}
\label{sec:EAK}

The logic of epistemic actions and knowledge (further on EAK) is a logical framework which combines a multi-modal classical logic with a dynamic-type propositional logic. Static modalities in EAK are parametrized with agents, and their intended interpretation is epistemic, that is, $\langle \aga\rangle A$ intuitively stands for `agent $\aga$ thinks that $A$ might be the case'. Dynamic modalities in EAK are parametrized with epistemic \emph{action-structures} (defined below) and their intended interpretation is analogous to that of dynamic modalities in e.g.\ Propositional Dynamic Logic. That is, $\langle \alpha\rangle A$ intuitively stands for `the action $\alpha$ is executable, and after its execution $A$ is the case'. Informally, action structures loosely resemble Kripke models, and encode information about epistemic actions such as e.g.\ public announcements, private announcements to a group of agents, with or without (actual or suspected) wiretapping, etc. Action structures consist of a finite nonempty domain of action-states, a designated state, binary relations on the domain for each agent, and a precondition map. Each state in the domain of an action structure $\alpha$ represents the possible appearance of the epistemic action encoded by $\alpha$. The designated state represents the action actually taking place. Each binary relation of an action structure represents the type, or degree, of uncertainty entertained by the agent associated with the given binary relation about the action taking place; for instance, the agents' knowledge, ignorance, suspicions. Finally, the precondition function maps each state in the domain to a formula, which is intended to describe the state of affairs under which it is possible to execute the (appearing) action encoded by the given state. This formula encodes the \emph{preconditions} of the action-state. The reader is referred to \cite{BMS} for further intuition and concrete examples.

Let \textsf{AtProp} be a countable set of atomic propositions, and $\mathsf{Ag}$ be a  nonempty set (of agents). The set $\mathcal{L}$ of  formulas $A$  of the logic of epistemic actions and knowledge (EAK), and the set $\mathsf{Act}(\mathcal{L})$ of the {\em action structures} $\alpha$ {\em over} $\mathcal{L}$ are defined simultaneously as follows:

\begin{center}
$A := p\in \mathsf{AtProp} \mid \neg A \mid A\vee A \mid \langle\aga\rangle A \mid \langle\alpha\rangle A\;\; \ \ (\alpha\in \mathsf{Act}(\mathcal{L}), \aga\in \mathsf{Ag}),$

\end{center}
where an {\em action structure over} $\mathcal{L}$ is a tuple  $\alpha = (K, k, (\alpha_\aga)_{\aga\in\mathsf{Ag}}, Pre_\alpha)$, such that $K$ is a finite nonempty set,  $k\in K$,  $\alpha_\aga\subseteq K\times K$  and $Pre_\alpha: K\to \mathcal{L}$.

The symbol $Pre(\alpha)$ stands for $Pre_\alpha(k)$.  For each action structure $\alpha$ and every $i\in K$, let $\alpha_i := (K, i, (\alpha_\aga)_{\aga\in\mathsf{Ag}}, Pre_\alpha)$. Intuitively, the family of action structures $\{\alpha_i\mid k\alpha_\aga i\}$ encodes the uncertainty of agent $\aga$ about the action $\alpha=\alpha_k$ that is actually taking place. Perhaps the best known epistemic actions are {\em public announcements}, formalized as action structures $\alpha$ such that $K = \{k\}$, and $\alpha_\aga = \{(k, k)\}$ for all $\aga\in\mathsf{Ag}$. The logic of public announcements (PAL) \cite{Plaza} can then be subsumed as the fragment of EAK restricted to action structures of the form described above.
The connectives $\top$, $\bot$, $\wedge$, $\rightarrow$ and $\leftrightarrow$ are defined as usual.

Standard models for EAK are relational structures $M = (W, (R_\aga)_{\aga\in \mathsf{Ag}}, V)$ such that $W$ is a nonempty set, $R_\aga\subseteq W\times W$ for each $\aga\in \mathsf{Ag}$, and $V:\mathsf{AtProp}\to \mathcal{P}(W)$. The interpretation of the static fragment of the language is standard. For every Kripke frame $\mathcal{F} = (W, (R_\aga)_{\aga\in\mathsf{Ag}})$ and each  action structure  $\alpha$, let the Kripke frame $\coprod_{\alpha}\mathcal{F} : = (\coprod_{K}W, ((R\times \alpha)_\aga)_{\aga\in \mathsf{Ag}})$ be defined
as follows: $\coprod_{K}W$ is the $|K|$-fold coproduct of $W$ (which is set-isomorphic to $W\times K$), and $(R\times \alpha)_\aga$ is a   binary relation on $\coprod_{K}W$ defined as $$(w, i)(R\times \alpha)_\aga (u, j)\quad \mbox{ iff }\quad w R_\aga u\ \mbox{ and }\ i\alpha_\aga j.$$

 For every model $M$ and each action structure $\alpha$, let $$\coprod_{\alpha}M := (\coprod_{\alpha}\mathcal{F}, \coprod_{K}V )$$ be such that $\coprod_{\alpha}\mathcal{F}$ is defined as above, and $(\coprod_{K}V)(p): = \coprod_{K}V(p)$ for every $p\in \mathsf{AtProp}$. Finally,  let the {\em update} of $M$ with the action structure $\alpha$ be the submodel $M^\alpha: = (W^\alpha, (R^\alpha_\aga)_{\aga\in \mathsf{Ag}}, V^\alpha)$ of $\coprod_{\alpha}M$ the domain of which is the subset $$W^\alpha: = \{(w, j)\in \coprod_{K}W\mid M, w\Vdash Pre_\alpha(j)\}.$$
 Given this preliminary definition, formulas of the form $\langle\alpha\rangle A$ are interpreted as follows:
$$M, w\Vdash \langle\alpha \rangle A\quad \mbox{ iff } \quad M, w\Vdash  Pre_\alpha(k)  \mbox{ and } M^\alpha, (w, k)\Vdash A.$$

The model $M^\alpha$ is intended to encode the (factual and epistemic) state of affairs after the execution of the action $\alpha$. Summing up, the construction of $M^\alpha$ is done in two stages: in the first stage, as many copies of the original model $M$ are taken as there are `epistemic potential appearances' of the given action (encoded by the action states in the domain of $\alpha$); in the second stage, states in the copies are removed if their associated original state does not satisfy the preconditions of their paired action-state.


A complete axiomatization of EAK consists of copies of the axioms and rules of the minimal normal modal logic K  for each modal operator, either epistemic or dynamic, plus the following (interaction) axioms:
\begin{eqnarray}
\langle\alpha\rangle p & \leftrightarrow & (Pre(\alpha)\wedge p); \label{eq:facts}\\
\langle \alpha\rangle \neg A & \leftrightarrow & (Pre(\alpha)\wedge \neg\langle \alpha\rangle A ); \label{eq:neg}\\
\langle \alpha\rangle (A\vee B) & \leftrightarrow & (\langle \alpha\rangle A\vee \langle \alpha\rangle B ); \label{eq:vee}\\
\langle \alpha\rangle \langle\aga\rangle A & \leftrightarrow & (Pre(\alpha)\wedge \bigvee\{\langle\aga\rangle\langle\alpha_i\rangle A\mid k\alpha_\aga i\}). \label{eq:interact axiom}
\end{eqnarray}

The interaction axioms above can be understood as attempts at defining the meaning of any given dynamic modality $\langle\alpha\rangle$ in terms of its interaction with the other connectives. In particular, while axioms \eqref{eq:neg} and \eqref{eq:vee} occur also in other dynamic logics such as PDL, axioms \eqref{eq:facts} and \eqref{eq:interact axiom} capture the specific behaviour of epistemic actions. Specifically, axiom \eqref{eq:facts} encodes the fact that epistemic actions do not change the factual state of affairs, and axiom \eqref{eq:interact axiom} plausibly rephrases the fact that `after the execution of $\alpha$, agent $\aga$ thinks that $A$ might be the case' in terms of `there being some epistemic appearance of $\alpha$ to $\aga$ such that $\aga$ thinks that, after its execution, $A$ is the case'.
An interesting aspect of these axioms is that they work as rewriting rules which can be iteratively used to transform any EAK-formula into an equivalent one free of dynamic modalities. Hence, the completeness of EAK follows from the completeness of its static fragment, and EAK is not more expressive than its static fragment. However, and interestingly, there is an exponential gap in succinctness between equivalent formulas in the two languages \cite{Luz}.

Action structures are one among many possible ways to represent actions. Following \cite{GKPLori},
we prefer to keep a black-box perspective on actions, and to identify agents $\aga$ with the indistinguishability relation they induce on actions; so, in the remainder of the article, the role of the action-structures $\alpha_i$ for $k\alpha i$ will be played by actions $\beta$ such that $\alpha\aga\beta$, allowing us to reformulate \eqref{eq:interact axiom} as
$$\langle \alpha\rangle \langle\aga\rangle A \  \leftrightarrow \  (Pre(\alpha)\wedge \bigvee\{\langle\aga\rangle\langle\beta\rangle A \mid \alpha\aga\beta\}).$$

\subsection{Intuitionistic  EAK}\label{ssec:IEAK}
The (single-agent version of the) intuitionistic logic of epistemic actions and knowledge (IEAK) has been introduced in \cite{KP}. In the present subsection we report on its multi-agent version. The reason for mentioning this logic in the preliminaries is that
the  calculus introduced in Section \ref{DC} takes the Hilbert-style axiomatization of IEAK---rather than that of its Boolean counterpart---as basic, and many of its rules are motivated by axioms which define the intuitionistic setting   (see Section \ref{DC} for further details on this topic).\footnote{The Boolean setting is captured by adding the so-called {\em Grishin  rules} (see page \pageref{page : Grishin rules}) to the basic framework.}.

Let \textsf{AtProp} be a countable set of atomic propositions, and let $\mathsf{Ag}$ be a nonempty set (of agents). The set $\mathcal{L}$(m-IK) of the formulas $A$ of the multi-modal version m-IK of Fischer Servi  intuitionistic modal logic IK (cf.\ \cite{FS84}) are inductively defined  as follows:
\[A := p\in \mathsf{AtProp} \mid \bot  \mid A\vee A \mid A\wedge A\mid A\rightarrow A \mid \langle\aga\rangle A \mid [\aga] A\  \ \  (\aga\in \mathsf{Ag})\]

Let $\neg A$ abbreviate as usual $A \rightarrow \bot$.
The Hilbert-style presentation of   m-IK is reported in Table~\ref{table:mIK}.

\begin{table}
\begin{center}
\begin{tabular}{rl}
\mc{2}{c}{\textbf{Axioms}}                                                         \\
    & $A\rightarrow (B\rightarrow A)$                                                   \\
     & $(A\rightarrow (B\rightarrow C))\rightarrow ((A\rightarrow B)\rightarrow (A\rightarrow C))$                   \\

     & $A\rightarrow (B\rightarrow A\pand B)$                                            \\
    & $A\pand B \rightarrow A$                                                   \\
     & $A\pand B \rightarrow B$                                                   \\

     & $A\rightarrow A \vee B$                                                    \\
    & $B\rightarrow A \vee B$                                                    \\
    & $(A\rightarrow C)\rightarrow ((B\rightarrow C)\rightarrow (A\vee B\rightarrow C))$                     \\

      & $\bot\rightarrow A$                                                         \\
      ~\\

    & $[\aga] (A\rightarrow B)\rightarrow ([\aga] A\rightarrow [\aga] B)$                 \\
    & $\langle\aga\rangle (A\vee B)\rightarrow \langle\aga\rangle A\vee \langle\aga\rangle B$                   \\
  & $\neg\langle\aga\rangle \bot$                                                  \\
FS1     & $\langle\aga\rangle (A\rightarrow B)\rightarrow ([\aga]  A\rightarrow \langle\aga\rangle B)$                 \\
FS2     & $(\langle\aga\rangle A\rightarrow [\aga] B)\rightarrow [\aga]  (A\rightarrow B)$                 \\

\mc{2}{c}{}                                                                        \\
\mc{2}{c}{\textbf{Inference Rules}}                                                \\
MP          & if $\vdash A\rightarrow B$ and $\vdash A$, then $\vdash B$                  \\
Nec        & if $\vdash A$, then $\vdash [\aga]  A$                              \\
\end{tabular}
\end{center}
\caption{Axioms and rules of the intuitionistic modal logic m-IK}
\label{table:mIK}
\end{table}

To define the language of IEAK, let \textsf{AtProp} be a countable set of atomic propositions, and let $\mathsf{Ag}$ be a nonempty set. The set $\mathcal{L}$(IEAK) of formulas $A$ of  the intuitionistic logic of epistemic actions and knowledge (IEAK), and the set $\mathsf{Act}(\mathcal{L})$ of the {\em action structures} $\alpha$ {\em over} $\mathcal{L}$ are defined simultaneously as follows:
\begin{center}
$A := p\in \mathsf{AtProp} \mid \bot\mid A\rightarrow A \mid A\vee A \mid A\wedge A\mid   \langle\aga\rangle A\mid [\aga] A\mid \langle\alpha\rangle A\mid [\alpha] A,$
\end{center}
where $\aga\in \mathsf{Ag}$, and an {\em action structure} $\alpha$ {\em over} $\mathcal{L}$(IEAK) is defined in a completely analogous way as action structures in the classical case, the only difference lying in the codomain of $Pre_\alpha$. Then, the logic IEAK is defined in a Hilbert-style presentation which includes the axioms and rules of m-IK plus the axioms  and rules  in Table~\ref{table:IEAK}.

\begin{table}
\begin{center}
\begin{tabular}{rl}
\mc{2}{c}{\textbf{Interaction Axioms}}                                                                       \\
 & $\lc\alpha\rc p \leftrightarrow Pre(\alpha) \pand p$                                                           \\
 & $\ls\alpha\rs p  \leftrightarrow Pre(\alpha) \rightarrow  p$                                                            \\
 & $\lc\alpha\rc \bot \leftrightarrow \bot$                                                                       \\
 & $\lc\alpha\rc \top \leftrightarrow Pre(\alpha)$                                                                \\
 & $\ls\alpha\rs \top \leftrightarrow  \top$                                                                       \\
 & $\ls\alpha\rs \bot \leftrightarrow \neg Pre(\alpha)$                                                           \\
 & $\ls\alpha\rs (A\pand B) \leftrightarrow \ls\alpha\rs A \pand \ls\alpha\rs B$                                  \\
 & $\lc\alpha\rc (A\pand B)\leftrightarrow  \lc\alpha\rc A \pand \lc\alpha\rc B$                                  \\
   & $\lc\alpha\rc (A\vee B) \leftrightarrow \lc\alpha\rc A \vee \lc\alpha\rc B$                                    \\
   & $\ls\alpha\rs (A\vee B) \leftrightarrow Pre(\alpha) \rightarrow  (\lc\alpha\rc A \vee \lc\alpha\rc B)$                 \\
  & $\lc\alpha\rc (A \rightarrow B) \leftrightarrow Pre(\alpha) \pand (\lc\alpha\rc A \rightarrow \lc\alpha\rc B)$                 \\
  & $\ls\alpha\rs (A \rightarrow B)\leftrightarrow \lc\alpha\rc A \rightarrow \lc\alpha\rc B$                                     \\
   & $\lc\alpha\rc \lc\aga\rc A \leftrightarrow Pre(\alpha)\pand \bigvee\{\lc\aga\rc\lc\beta\rc A \mid \alpha\aga\beta\} $  \\
  & $\ls\alpha\rs\lc\aga\rc A \leftrightarrow Pre(\alpha)\rightarrow \bigvee\{\lc\aga\rc\lc\beta\rc A \mid \alpha\aga\beta\} $   \\
  & $\ls\alpha\rs[\aga] A \leftrightarrow Pre(\alpha)\rightarrow  \bigwedge\{[\aga]\ls\beta\rs A \mid \alpha\aga\beta\}$  \\
   & $\lc\alpha\rc[\aga] A \leftrightarrow Pre(\alpha)\pand \bigwedge\{[\aga]\ls\beta\rs A \mid \alpha\aga\beta\}$ \\
\mc{2}{c}{}                                                                                                  \\
\mc{2}{c}{\textbf{Inference Rules}}                                                                          \\
vNec   & if $\vdash A$, then $\vdash \ls\alpha\rs A$                                                          \\
\end{tabular}
\end{center}
\caption{Axioms and rules of the intuitionistic epistemic logic IEAK}
\label{table:IEAK}
\end{table}

\subsection{Display calculi}\label{ssec:DisplayLogic} 

The first display calculus appears in Belnap's paper  \cite{Belnap}, as a sequent system augmenting and refining Gentzen's basic design of sequent calculi, which admit two types of rules: the structural, and the operational. Belnap's refinement is based on the introduction of a special syntax for the constituents of each sequent, which includes {\em structural connectives} along with {\em logical}, or {\em operational connectives}. For an expanded discussion of these ideas, the reader is referred to \cite{GAV, Wa98, Restall}. 

\paragraph{Structures and display property.} {\em Structures} are built up much in the same way as formulas, taking formulas as atomic components, and applying structural connectives (which are typically 0-ary, unary and binary) so that each structure can be uniquely associated with and identified by its generation tree. Every node of such a generation tree defines a {\em substructure} of the given structure.

\begin{defn} \label{def: display prop} (cf.\ \cite[Section 3.2]{Belnap}) 
A proof system enjoys the {\em full display property}  iff for every sequent $X \vdash Y$ and every substructure $Z$  of either  $X$ or  $Y$, the sequent  $X \vdash Y$ can be  transformed, using the rules of the system, into a logically equivalent sequent which is either of the form $Z \vdash W$ or of the form $W \vdash Z$, for some structure $W$. In the first case, $Z$ is \emph{displayed in precedent position}, and in the second case, $Z$ is \emph{displayed in succedent position}.
The rules enabling this equivalent rewriting  are called \emph{display postulates}.
\end{defn}
In what follows, we will sometimes write e.g.\ $(X\vdash Y)[Z]^{pre}$ (resp.\ $(X\vdash Y)[Z]^{suc}$) to indicate that $Z$ occurs as a substructure in precedent (resp.\ succedent) position within the sequent $X\vdash Y$.
Thanks to the fact that the display postulates are based on adjunction and residuation, it can be proved that exactly one of the two alternatives mentioned in the definition above occurs. In other words, in a system enjoying the display property, any substructure of any sequent  $X \vdash Y$ is always displayed either only in  precedent position or only in succedent position. This is why we can talk about occurrences of substructures in precedent or in succedent position, even if they are nested deep within a given sequent.

\paragraph{Uniform strategy for cut-elimination.} In \cite{Belnap}, a meta-theorem is proven, which gives sufficient conditions in order for a sequent calculus to enjoy cut elimination. This meta-theorem captures the essentials of the cut-elimination procedure Gentzen-style, and is the main technical motivation for the design of Display Logic.  Belnap's meta-theorem  gives a set of eight conditions on sequent calculi, most of which are verified by inspection on the shape of the rules. Together, these conditions guarantee that the cut rule is eliminable in the given sequent calculus, and that the system enjoys the subformula property.  When Belnap's meta-theorem can be applied, it provides a much smoother and more modular route to cut elimination than the
Gentzen-style proofs. Belnap's original meta-theorem has been generalized and refined by various authors (cf.\ \cite{Poggiolesi, Restall, Wa98}). Particularly relevant to us is the notion of {\em properly displayable calculus}, introduced in \cite[Section 4.1]{Wa98}, a generalization of which has been proposed in \cite{GAV}, which in its turn is further generalized in Section \ref{sec:quasi}.

\paragraph{Relativized display property.} The full display property is a key ingredient in the proof of the cut-elimination metatheorem. For instance, it enables a system enjoying it to meet Belnap's condition C$_8$ 
for the cut-elimination metatheorem.
However, it turns out that an analogously good behaviour can be guaranteed of any sequent calculus enjoying the following weaker property:

\begin{defn} \label{def: relativized display prop}  
A proof system enjoys the {\em relativized display property}  iff for every {\em derivable} sequent $X \vdash Y$ and every substructure $Z$  of either  $X$ or  $Y$, the sequent  $X \vdash Y$ can be  transformed, using the rules of the system, into a logically equivalent sequent which is either of the form $Z \vdash W$ or of the form $W \vdash Z$, for some structure $W$.
\end{defn}
The  calculus defined in Section \ref{DC} does not enjoy the full display property, but does enjoy the relativized display property above (more about this in Sections \ref{DC} and \ref{sec : safe virtual adjoints}), which enables it to verify the condition C'$_8$ (see Section \ref{sec:quasi}). More details about it are collected in Section \ref{ssec:cut elimination}.
Finally, notice that the definition of substructures in precedent or succedent position within each sequent can be given in a way which does not rely on the full display property. It is enough to rely on the polarity  of the coordinates of each structural connective: if these polarities are assigned, then for any sequent $X\vdash Y$, if $Z$ is a substructure of $X$, then $Z$ is in precedent (resp.\ succedent) position if, in the generation tree of $X$, the path from $Z$ to the root goes through  an even (resp.\ odd) number of coordinates with negative polarity. If $Z$ is a substructure of $Y$, then $Z$ is in succedent (resp.\ precedent) position if, in the generation tree of $Y$, the path from $Z$ to the root goes through  an even (resp.\ odd) number of coordinates with negative polarity.

\subsection{A single-type display calculus for EAK}
\label{ssec:D'.EAK prelim}

In \cite{GAV}, a display calculus is introduced for EAK, which is shown to be sound w.r.t.\ the final coalgebra semantics, syntactically  complete w.r.t.\ EAK and to enjoy cut-elimination Belnap-style. In the present subsection we briefly report on it, not only for the sake of providing a relevant example of display calculus, but above all because a translation can be established between the operational language of D'.EAK and  of the Dynamic Calculus (cf.\ Section \ref{DC}). This translation  is important for the further treatment of Sections \ref{sec:soundness} and \ref{sec : safe virtual adjoints}.

The structural and operational languages of D'.EAK are expansions of the standard structural and operational propositional  languages with the following (structural and operational) modal operators, indexed by  agents $\aga$ and actions $\alpha$, and (structural and operational) constant symbols:

\begin{center}
\begin{tabular}{|c|c|c|c|c|c|c|c|c|c|c|c|c| l}
\hline
\scriptsize{Structural symbols} &\mc{2}{c|}{$\{\aga\}$}       & \mc{2}{c|}{$\RESagaProxy$} &\mc{2}{c|}{$\{\alpha\}$}       & \mc{2}{c|}{$\RESalphaProxy$} & \mc{2}{c|}{$\Phi_\alpha$}   \\
\hline
\scriptsize{Operational symbols} &  $\lc\aga\rc$ & $\ls\mkern2mu\aga\mkern1mu\rs$& ($\RESagaDia$) & ($\RESagaBox$)&  $\lc\alpha\rc$ & $\ls\mkern2mu\alpha\mkern1mu\rs$  & $\RESalphaDia$ & $\RESalphaBox$  & $1_\alpha$ & $\phantom{1_\alpha}$\\
\hline
\end{tabular}
\end{center}
\noindent The structural connectives $\{\alpha\}$ and $\RESalphaProxy$ correspond to  diamond-type modalities when occurring in precedent position, and to  box-type modalities when occurring in succedent position. The structural and operational constants $\Phi_\alpha$ and $1_\alpha$ are used to capture the proof-theoretic behaviour of the metalinguistic abbreviation $Pre(\alpha)$ at the object-level.  In the rules below, the structural connective $\Phi_\alpha$ can occur only in precedent position. Hence, the structural constant symbol $\Phi_\alpha$ can never be interpreted as anything else than $1_\alpha$. However, a natural way to extend D'.EAK would be to introduce an operational constant symbol $0_\alpha$, intuitively standing for the postconditions of $\alpha$ for each action $\alpha$, and dualize the relevant rules so as to capture the behaviour of postconditions.

The  connectives $\RESalphaDia$ and $\RESalphaBox$ occur within brackets since they are not actually part of the logical language of D'.EAK, but point at the fact that  the structural connective $\RESalphaProxy$ is interpreted in the final coalgebra as the diamond (resp.\ box) associated with the {\em converse} of the relation associated with the epistemic action $\alpha$ (for an expanded discussion on this, the reader is referred to \cite[Section 5.2]{GAV}). The key aspect of the final coalgebra as a semantic environment for EAK is that  it makes it possible to see the dynamic connectives $[\alpha]$ and $\langle\alpha\rangle$ as parts of {\em adjoint pairs}, precisely involving the additional modalities $\RESalphaDia$ and $\RESalphaBox$. Specifically, we have the following (syntactic) adjunction relations $\langle\alpha\rangle \dashv \RESalphaBox $ and $ \RESalphaDia\dashv[\alpha]$: for all formulas $A, B$,
\begin{eqnarray}
\label{eqnarray : adjunction D'EAK}
\langle\alpha\rangle A\vdash B \ \mbox{ iff } \ A\vdash  \RESalphaBox B\quad\quad \RESalphaDia A\vdash B \ \mbox{ iff } \ A\vdash  [\alpha] B
\end{eqnarray}
The reader is referred to \cite[Section 5]{GAV} for a detailed discussion.
The two tables below introduce the  structural rules for the dynamic modalities which have the same shape as those for the agent-indexed modalities, here omitted. 

\begin{center}
{\scriptsize{
\begin{tabular}{rl}
\mc{2}{c}{\normalsize{\textbf{Structural Rules}}} \\
& \\

\AX$\textrm{I} \fCenter  X $
\LeftLabel{$nec_L^{dyn}$}
\UI$ \{\alpha \}\, \textrm{I} \fCenter  X $
\DisplayProof
&
\AX$ X \fCenter  \textrm{I} $
\RightLabel{$nec_R^{dyn}$}
\UI$ X \fCenter   \{\alpha\}\, \textrm{I} $
\DisplayProof  \\

 & \\

\AX$ \textrm{I}  \fCenter X$
\LeftLabel{$^{dyn}nec_L $}
\UI$\RESalphaProxy\, \textrm{I}  \fCenter  X$
\DisplayProof &
\AX$X \fCenter  \textrm{I}  $
\RightLabel{$^{dyn}nec_R$}
\UI$ X \fCenter \RESalphaProxy\, \textrm{I} $
\DisplayProof \\

 & \\

\AX$\{\alpha\} Y > \{\alpha\} Z \fCenter X$
\LeftLabel{$FS^{dyn}_L$}
\UI$\{\alpha\} (Y > Z) \fCenter X$
\DisplayProof &
\AX$Y \fCenter \{\alpha\} X > \{\alpha\} Z$
\RightLabel{$FS^{dyn}_R$}
\UI$Y \fCenter \{\alpha\} (X > Z)$
\DisplayProof \\

 & \\

\AX$\{\alpha\} X\,; \{\alpha\} Y \fCenter Z$
\LeftLabel{$mon^{dyn}_L$}
\UI$\{\alpha\} (X\,; Y) \fCenter Z$
\DisplayProof
&
\AX$Z \fCenter \{\alpha\} Y\,; \{\alpha\} X$
\RightLabel{$mon^{dyn}_R$}
\UI$Z \fCenter \{\alpha\} (Y\,; X)$
\DisplayProof \\

& \\
\AX$\RESalphaProxy Y > \RESalphaProxy X \fCenter Z$
\LeftLabel{$^{dyn}FS_L$}
\UI$\RESalphaProxy (Y > X) \fCenter Z$
\DisplayProof &
\AX$Y \fCenter \RESalphaProxy X > \RESalphaProxy Z$
\RightLabel{$^{dyn}FS_R$}
\UI$Y \fCenter \RESalphaProxy (X > Z)$
\DisplayProof \\

 & \\
\AX$\RESalphaProxy X\,; \RESalphaProxy Y \fCenter Z$
\LeftLabel{$^{dyn}mon_L$}
\UI$\RESalphaProxy (X\,; Y) \fCenter Z$
\DisplayProof &
\AX$Z \fCenter \RESalphaProxy Y\,; \RESalphaProxy X$
\RightLabel{$^{dyn}mon_R$}
\UI$Z \fCenter \RESalphaProxy (Y\,; X)$
\DisplayProof
\\
& \\
\AX$\{\aga\} (X\,; \RESagaProxy Y) \fCenter Z$
\LeftLabel{\scriptsize{$conj$}}
\UI$ \{\aga\} X\,;Y \fCenter Z$
\DisplayProof
&
\AX$X \fCenter \{\aga\}  (Y\,; \RESagaProxy Z) $
\RightLabel{\scriptsize{$conj$}}
\UI$X \fCenter \{\aga\} Y\,;Z $
\DisplayProof
\\
& \\
\AX$\RESagaProxy (X\,; \{\aga \} Y) \fCenter Z$
\LeftLabel{\scriptsize{$conj$}}
\UI$ \RESagaProxy X\,;Y \fCenter Z$
\DisplayProof
&
\AX$X \fCenter \RESagaProxy (Y\,; \{\aga \} Z) $
\RightLabel{\scriptsize{$conj$}}
\UI$X \fCenter  \RESagaProxy Y\,;Z $
\DisplayProof
\\
& \\
\end{tabular}
}}
\end{center}

\noindent
The $conj$-rules  and the $FS$-rules can be shown to be interderivable thanks to the following display postulates.

\begin{center}
{\scriptsize{
\begin{tabular}{rl}
\mc{2}{c}{\normalsize{\textbf{Display Postulates}}} \\

 & \\

\AX$\{\alpha\} X \fCenter Y$
\LeftLabel{$(\{\alpha\}, \RESalphaProxy)$}
\doubleLine
\UI$X \fCenter \RESalphaProxy Y$
\DisplayProof

 &

\AX$Y \fCenter \{\alpha\} X$
\RightLabel{$(\RESalphaProxy, \{\alpha\})$}
\doubleLine
\UI$\RESalphaProxy Y \fCenter X$
\DisplayProof \\

 & \\
\end{tabular}
}}
\end{center}
\smallskip

The display postulates above are  direct translations of the adjunction relations \eqref{eqnarray : adjunction D'EAK}.
Next, we report on the structural rules which are to capture the specific behaviour of epistemic actions: \label{pageref:atom}
\smallskip
\begin{center}
{\scriptsize{
\begin{tabular}{rl}

\mc{2}{c}{\normalsize{\textbf{Atom}}} \\

& \\

\mc{2}{c}{
\AXC{}
\RightLabel{{$atom$}}
\LeftLabel{{$\phantom{atom}$}}
\UI$\Gamma p \fCenter \Delta p$
%
\DisplayProof
} \\



 & \\
\end{tabular}
}}
\end{center}
\smallskip

\noindent where $\Gamma$ and $\Delta$ are arbitrary finite sequences of the form $(\alpha_1)\ldots (\alpha_n)$ (possibly of different length), such that each  $(\alpha_j)$ is of the form $\{\alpha_j\}$ or of the form $\RESalphajProxy$, for $1\leq j\leq n$. Intuitively, the {\em atom rules} capture the requirement that epistemic actions do not change the factual state of affairs (in the Hilbert-style presentation of EAK, this is encoded in the axiom \eqref{eq:facts} in Section \ref{sec:EAK}).

\smallskip
\begin{center}
{\scriptsize{
\begin{tabular}{@{}rl@{}}
\mc{2}{c}{\normalsize{\textbf{Structural Rules for Epistemic Actions}}} \\


 & \\

\mc{2}{c}{\ \ \ \
\AX$X \fCenter Y$
\RightLabel{\emph{balance}}
\UI$\{\alpha\} X\fCenter \{\alpha\} Y$
\DisplayProof}\\

 & \\

\AX$\{\alpha\} \RESalphaProxy X  \fCenter  Y$
\LeftLabel{$comp^{\alpha}_L$}
\UI$ \Phi_\alpha; X \fCenter  Y$
\DisplayProof
&
\AX$X \fCenter  \{\alpha\} \RESalphaProxy Y$
\RightLabel{$comp^{\alpha}_R$}
\UI$ X \fCenter \Phi_\alpha >Y$
\DisplayProof
\\

 & \\

\AX$ \Phi_\alpha;  \{\alpha\} X \fCenter  Y$
\LeftLabel{\emph{reduce}$_{L}$}
\UI$ \{\alpha\} X \fCenter  Y$
\DisplayProof
&
\AX$ Y \fCenter \Phi_\alpha > \{ \alpha \} X$
\RightLabel{\emph{reduce}$_{R}$}
\UI$ Y \fCenter \{ \alpha \} X$
\DisplayProof \\

& \\
\AX$ \{\alpha\} \{\aga\} X \fCenter Y$
\LeftLabel{\emph{swap-in}$_{L}$}
\UI$\Phi_\alpha; {\{\aga\}\{\beta\}_{\alpha\aga\beta}\, X} \fCenter Y$
\DisplayProof
&
\AX$Y \fCenter \{\alpha\} \{\aga\} X$
\RightLabel{\emph{swap-in}$_{R}$}
\UI$Y \fCenter \Phi_\alpha > {\{\aga\} \{\beta\}_{\alpha\aga\beta}\,X}$
\DisplayProof\\

& \\

\AX$\Big ( \{\aga\}\{\beta\} \,X \fCenter Y\mid \alpha\aga\beta\Big)$
\LeftLabel{\emph{swap-out}$_{L}$}
\UI$ \{\alpha\} \{\aga\} X \fCenter \Bigsemic\Big(Y\mid \alpha\aga\beta\Big)$
\DisplayProof
&
\AX$\Big(Y \fCenter \{\aga\}\{\beta\}\,X \mid \alpha\aga\beta\Big)$
\RightLabel{\emph{swap-out}$_{R}$}
\UI$\Bigsemic\Big(Y\mid \alpha\aga\beta\Big) \fCenter  \{\alpha\} \{\aga\} X$
\DisplayProof\\

&  \\

\end{tabular}
}}
\end{center}
\smallskip

\noindent The {\em swap-in} rules are unary and should be read as follows: if the premise holds, then the conclusion holds relative to any action $\beta$ such that $\alpha\aga\beta$.
The {\em swap-out} rules do not have a fixed arity; they have as many premises\footnote{The {\em swap-out} rule could indeed be infinitary if action structures were allowed to be infinite, which in the present setting, as in \cite{BMS}, is not the case.} as there are actions $\beta$ such that $\alpha\aga\beta$.
In their conclusion, the symbol $\Bigsemic \Big(Y \mid \alpha\aga\beta \Big)$ refers to a string $(\cdots(Y\,; Y) \,;\cdots \,; Y)$ with $n$ occurrences of $Y$, where $n = |\{\beta\mid \alpha\aga\beta\}|$. The {\em swap-in} and {\em swap-out} rules  encode the interaction between dynamic and epistemic modalities as it is captured by the interaction axioms in the Hilbert style presentation of EAK (cf.\ \eqref{eq:interact axiom} in Section \ref{sec:EAK} and similarly in Section \ref{ssec:IEAK}). The {\em reduce} rules encode well-known EAK validities such as $\langle\alpha\rangle A \rightarrow (Pre(\alpha)\wedge \langle\alpha\rangle A )$.


\noindent  Finally, the  operational rules for $\lc\alpha\rc$, $\ls\alpha\rs$, and $1_\alpha$ are reported  below:

{\scriptsize{
\begin{center}
\begin{tabular}{rl}
\mc{2}{c}{\normalsize{\textbf{Operational Rules}}} \\
 & \\
\AX$\{\alpha\} A \fCenter X$
\LeftLabel{\fns$\lc\alpha\rc_L$}
\UI$\lc\alpha\rc A \fCenter X$
\DisplayProof &
\AX$X \fCenter A$
\RightLabel{\fns$\lc\alpha\rc_R$}
\UI$\{\alpha\} X \fCenter \lc\alpha\rc A$
\DisplayProof \\
 & \\
\AX$A \fCenter X$
\LeftLabel{\fns$\ls\alpha\rs_L$}
\UI$\ls\alpha\rs A \fCenter \{\alpha\} X$
\DisplayProof &
\AX$X \fCenter \{\alpha\} A$
\RightLabel{\fns$\ls\alpha\rs_R$}
\UI$X \fCenter \ls\alpha\rs A$
\DisplayProof \\
 & \\
\AX$\Phi_\alpha  \fCenter X$
\LeftLabel{\scriptsize{${1_\alpha}_L$}}
\UI$ 1_\alpha \fCenter X$
\DisplayProof
 &
\AxiomC{\phantom{$\Phi_\alpha \vdash$}}
\RightLabel{\fns{${1_\alpha}_R$}}
\UnaryInfC{$\Phi_\alpha \vdash 1_\alpha $}
\DisplayProof \\
 & \\

\end{tabular}
\end{center}
}
}

\section{Multi-type calculi, and cut elimination metatheorem}
\label{sec:multi}

The present section is aimed at introducing the environment of multi-type display calculi. Our treatment  will be very general, and in particular, no signature will be specified. However, the calculus introduced in Section \ref{DC} is a concrete instantiation of this abstract description.

\subsection{Multi-type calculi}
\label{sec:multi}

Our starting point is a propositional language, the terms of which form $n$ pairwise disjoint types $\mathsf{T_1}\ldots \mathsf{T_n}$, each of which with its own signature.
We will use $a, b, c$ and $x, y, z$ to respectively denote operational and structural  terms of  unspecified (possibly different) type. Further, we assume that operational connectives and structural connectives are given {\em both} within each type {\em and} also between different types, so that the display property holds. 

%
%
In the applications we have in mind, the need will arise to support types that are semantically ordered by inclusion.
For example, in Section~\ref{DC} we will introduce, beside the type $\mathsf{Fm}$ of formulas, two types $\mathsf{Fnc}$ and $\mathsf{Act}$ of functional and general actions, respectively. The need for enforcing the distinction between functional and general actions   in the specific situation of  Section~\ref{DC} arises because of the presence of  the rule {\em balance} (see page \pageref{page:two types} for more details on this topic). The semantic point of view suggests to treat $\mathsf{Fnc}$ as a proper subset of $\mathsf{Act}$, but our syntactic stipulations, although will be sound w.r.t.\ this state of affairs, will be tuned for the more general situation in which the sets $\mathsf{Fnc}$ and $\mathsf{Act}$  are disjoint.
This is convenient as each term can be assigned a unique type {\em unambiguously}. This is a crucial requirement for the Belnap-style cut elimination theorem of the next section, and will be explicitly stated in condition C'$_2$ below.

\begin{definition}
\label{def:type-uniformity}
A sequent $x\vdash y$ is {\em type-uniform} if $x$ and $y$ are of the same type $\mathsf{T}$. In this case, we will say that $x\vdash y$ is of type $\mathsf{T}$.
\end{definition}
A fundamental and very natural  desideratum for rules in a multi-type display calculus is that they preserve type-uniformity, that is, each rule should be such that if all the premises are type uniform, then the conclusion is type uniform. As we will see, all rules in the multi-type calculus introduced in Section \ref{DC} preserve type uniformity.

Finally, in a display calculus, the cut rule is typically of the following form:
\begin{center}
\AX$X \fCenter A$
\AX$A \fCenter Y$
\RightLabel{$Cut$}
\BI$X \fCenter Y$
\DisplayProof
\end{center}
where $X, Y$ are structures and $A$ is a formula.
This translates straightforwardly to the multi-type environment, by the stipulation that  cut rules of the form
\begin{center}
\AX$x \fCenter a$
\AX$a \fCenter y$
\RightLabel{$Cut$}
\BI$x \fCenter y$
\DisplayProof
\end{center}
are allowed in the given multi-type system for each type. These cut rules will be asked to satisfy the following additional requirement:

\begin{definition}
\label{def:strong-type-uniformity}
A  rule is {\em strongly type-uniform} if its premises and conclusion are of the same type.
\end{definition}

\subsection{Quasi-properly displayable multi-type calculi}
\label{sec:quasi}
    In \cite{GAV}, to show that  Belnap-style cut elimination holds for the display calculus D'.EAK, the definition of quasi-properly displayable calculi is given (generalizing  Wansing's definition of properly displayable calculi \cite[Section 4.2]{Wa98}), and its corresponding Belnap style meta-theorem is discussed.  
We are working towards  the proof that the multi-type display calculus introduced in Section \ref{DC} enjoys cut elimination Belnap-style. The aim of the present subsection is then to extend the notion of quasi-properly displayable calculi to the multi-type environment.
\label{ssec:quasi-def}
 Let a {\em quasi-properly displayable multi-type calculus} be any display
calculus in a multi-type language satisfying the following list of conditions\footnote{See \cite{GAV} for a discussion on C'$_5$ and C''$_5$.}:

\paragraph{C$_1$: Preservation of operational terms.} Each operational term occurring in a premise of an inference rule {\em inf} is a subterm of some operational term in the conclusion of {\em inf}.
\paragraph{C$_2$: Shape-alikeness of parameters.} Congruent parameters\footnote{The congruence relation is an equivalence relation which is meant to identify the different occurrences of the same formula or substructure along the branches of a derivation \cite[section 4]{Belnap}, \cite[Definition 6.5]{Restall}. Condition C$_2$  can be understood as a condition on the {\em design} of the rules of the system if the congruence relation is understood as part of the specification of each given rule; that is, each rule of the system should come with an explicit specification of which elements are congruent to which (and then the congruence relation is defined as the reflexive and transitive closure of the resulting relation). In this respect, C$_2$ is nothing but a sanity check, requiring that the congruence is defined in such a way that indeed identifies the occurrences which are intuitively ``the same''.} are occurrences of the same structure.
\paragraph{C'$_2$: Type-alikeness of parameters.}  Congruent parameters have exactly the same type. This condition bans the possibility that a parameter changes type along its history.
\paragraph{C'$_3$: Restricted non-proliferation of parameters.} Each parameter in an inference rule {\em inf} is congruent to at most one constituent in the conclusion of {\em inf}. This restriction does not need to apply to parameters of any type $\mathsf{T}$ such that the only applications of cut with cut terms of type $\mathsf{T}$ are of the following shapes:

\begin{center}
\begin{tabular}{cc}
\AXC{$\vdots$}
\noLine
\UI$X\fCenter a$
\AX$a\fCenter a$
\BI$X\fCenter a$
\DisplayProof
&
\AX$a\fCenter a$
\AXC{$\vdots$}
\noLine
\UI$a\fCenter Y$
\BI$a\fCenter Y$
\DisplayProof
\\
\end{tabular}
\end{center}

\paragraph{C$_4$: Position-alikeness of parameters.} Congruent parameters are either all antecedent or all succedent parts of their respective sequents.
\paragraph{C'$_5$: Quasi-display of principal constituents.} If an operational term $a$ is principal in the conclusion sequent $s$ of a derivation $\pi$, then $a$ is in display, unless $\pi$ consists only of its conclusion sequent $s$ (i.e.\ $s$ is an axiom).
\paragraph{C''$_5$: Display-invariance of axioms.} If a display rule can be applied to an axiom $s$, the result of that rule application is again an axiom.
\paragraph{C'$_6$: Closure under substitution for succedent parts within each type.} Each rule is closed under simultaneous substitution of arbitrary structures for congruent operational terms occurring in succedent position, {\em within each type}.
\paragraph{C'$_7$: Closure under substitution for precedent parts within each type.} Each rule is closed under simultaneous substitution of arbitrary structures for congruent operational terms occurring in precedent position, {\em within each type}.

Condition C$_6$ (and  likewise C'$_7$) ensures, for instance, that if the following inference is an application of the rule $R$:

\begin{center}
\AX$(x \fCenter y) \big([a]^{suc}_{i} \,|\, i \in I\big)$
\RightLabel{$R$}
\UI$(x' \fCenter y') [a]^{suc}$
\DisplayProof
\end{center}

\noindent and $\big([a]^{suc}_{i} \,|\, i \in I\big)$ represents all and only  the occurrences of the operational term $a$ in the premiss which are congruent to the occurrence of $a$  in the conclusion\footnote{Clearly, if $I = \varnothing$, then the occurrence of $a$ in the conclusion is congruent to itself.}, 
then also the following inference is an application of the same rule $R$:

\begin{center}
\AX$(x \fCenter y) \big([z/a]^{suc}_{i} \,|\, i \in I\big)$
\RightLabel{$R$}
\UI$(x' \fCenter y') [z/a]^{suc}$
\DisplayProof
\end{center}

\noindent where the structure $z$ is substituted for $a$, and $z$ and $a$ have the same type.

\paragraph{C'$_8$: Eliminability of matching principal constituents.}

This condition  requests a standard Gentzen-style checking, which is now limited to the case in which  both cut formulas  are {\em principal}, and hence each of them has been introduced with the last rule application of each corresponding subdeduction. In this case, analogously to the proof  Gentzen-style, condition C'$_8$ requires being able to transform the given deduction into a deduction with the same conclusion in which either the cut is eliminated altogether, or is transformed in one or more applications of the cut rule, involving proper subterms of the original operational cut-term. In addition to this, specific to the multi-type setting is the requirement that the new application(s) of the cut rule be also {\em strongly type-uniform} (cf.\ condition C$_{10}$ below).

\paragraph{C''$_8$: Closure of axioms under cut.} If $x\vdash a$ and $a\vdash y$ are axioms, then $x\vdash y$ is again an axiom.

\paragraph{C$_9$: Type-uniformity of derivable sequents.} Each derivable sequent is type-uniform. 

\paragraph{C$_{10}$: Strong type-uniformity of cut rules.} All cut rules are strongly type-uniform (cf.\ Definition \ref{def:strong-type-uniformity}).
\bigskip

\subsection{Belnap-style metatheorem for multi-types}
In the present subsection, we state and prove the Belnap-style metatheorem which we will appeal to when establishing  the cut elimination Belnap-style  for the calculus we will introduce in the next section.

\begin{theorem}
\label{thm:meta multi}
Any multi-type display calculus satisfying C$_2$, C'$_2$, C'$_3$, C$_4$, C'$_5$, C''$_5$, C'$_6$, C'$_7$, C'$_8$, C''$_8$, C$_9$ and C$_{10}$ is cut-admissible. If also C$_1$ is satisfied, then the calculus enjoys the subformula property.
\end{theorem}

\begin{proof}
This is a generalization of the proof in \cite[Section 3.3, Appendix A]{Wan02}. For the sake of conciseness, we will expand only on the parts of the proof which depart from that treatment. As usual, the proof is done by induction on the ordered pair of parameters given by the complexity of the cut term and the height of the cut.
%
Our original derivation is
\begin{center}
\AXC{\ \ \ $\vdots$ \raisebox{1mm}{$\pi_1$}}
\noLine
\UI$x\fCenter a$
\AXC{\ \ \ $\vdots$ \raisebox{1mm}{$\pi_2$}}
\noLine
\UI$a \fCenter y$
\BI$x \fCenter y$
\DisplayProof
\end{center}

\newpage

\paragraph*{Principal stage: both cut formulas are principal.}\label{principal stage}
\noindent There are three subcases.

If the end sequent $x \fCenter y$ is identical to the conclusion of $\pi_1$ (resp.\ $\pi_2$), then we can eliminate the cut simply replacing the derivation above with $\pi_1$ (resp.\ $\pi_2$).

If the premises $x \vdash a$ and $a\vdash y$ are axioms, then, by C''$_8$, the conclusion $x\vdash y$ is an axiom, therefore the cut can be eliminated by simply replacing the original derivation with $x \fCenter y$.

If one of the two premises of the cut in the original derivation is not an axiom, then, by C'$_8$, there is a proof of $x \fCenter y$ which uses the same premise(s) of the original derivation and which involves only strongly uniform cuts on proper subterms of $a$.

\paragraph*{Parametric stage: at least one cut term is parametric.}
There are two subcases: either one cut term is principal or they are both parametric.

Consider the subcase in which one cut term is principal. W.l.o.g.\ we assume that the cut-term $a$ is principal in the  left-premise $x \vdash a$ of the cut in the original proof (the other case is symmetric). We can assume w.l.o.g.\ that the conclusion of the cut is different from either of its premises. Then, conditions C$_2$ and C'$_3$ make it possible to trace the history-tree of the occurrences of the cut-term $a$ in $\pi_2$ (cf.\ \cite[Remark 1]{GAV}), and by conditions C'$_2$ and $C_4$, any ancestor of $a$ is of the same type and in the same position (that is, is in precedent position).
 The situation can be pictured as follows:
\begin{center}
\AXC{\ \ \ $\vdots$ \raisebox{1mm}{$\pi_1$}}
\noLine
\def\fCenter{\vdash}
\UI$ x \fCenter a$
\AXC{\ \ \ \ \ $\vdots$ \raisebox{1mm}{$\pi_{2.i}$}}
\noLine
\UI$\underline{a}_i \fCenter y_i$
\noLine
\UIC{$\ \ \ddots$}
\noLine
\AXC{\ \ \ \ \ $\vdots$ \raisebox{1mm}{$\pi_{2.j}$}}
\noLine
\UI$(x_j \fCenter y_j)[\underline{a}_j]^{pre}$
\noLine
\UIC{\!\!$\vdots$}
\noLine
\AXC{\ \ \ \ \ $\vdots$ \raisebox{1mm}{$\pi_{2.k}$}}
\noLine
\UI$(x_k \fCenter y_k)[\overline{a}_k]^{pre}$
\noLine
\UIC{$\!\!\!\!\!\!\!\!\!\!\!\!\!\!\!\!\!\!\!\!\!\!\!\!\!\!\!\!\!\!\!\!\!\!\!\!\!\!\!\!\!\!\!\!\iddots$}
\noLine
\TIC{$\ \ \ \ \ \ \ \ \rule[-5.2mm]{0mm}{0mm}\ddots\vdots\iddots\rule{0mm}{10mm}$ \raisebox{1mm}{$\pi_2$}}
\noLine
\def\fCenter{\vdash}
\UIC{$\ \ \ a \fCenter y$}
\BI$ x \fCenter y$
\DisplayProof
\end{center}
\smallskip
where, for $i, j, k \in \{1, \ldots, n\}$, the nodes \[\underline{a}_i \fCenter y_i, \quad (x_j\vdash y_j)[a_j]^{pre}, \ \mbox { and }\  (x_k \fCenter y_k)[\overline{a}_k]^{pre}\] represent the three ways in which the leaves $a_i$, $a_j$ and $a_k$ in the history-tree of $a$ in $\pi_2$ can be introduced, and which will be discussed below. The notation $\underline{a}$  (resp.\ $\overline{a}$) indicates that the given occurrence is principal (resp.\ parametric). Notice that condition C$_4$ guarantees that all occurrences in the history of $a$ are in precedent position in the underlying derivation tree, and condition C'$_2$ guarantees that the type of $a$ never changes along its history.
Let $a_l$ be introduced as a parameter (as represented in the picture above in the conclusion of $\pi_{2.k}$ for $a_l = a_k$).
Assume that $(x_k \vdash y_k)[\overline{a}_k]^{pre}$ is the conclusion of an application \emph{inf} of the rule \emph{Ru} (for instance, in the calculus of Section \ref{DC}, this situation arises if $a_k$ is of type \textsf{Fm} and has been introduced with an application of Weakening, or if $a_k$ is of type \text{Fnc} and has been introduced with an application of Atom, or Balance).
Since $a_k$ is a leaf in the history-tree of $a$, we have that $a_k$ is congruent only to itself in $x_k \vdash y_k$.
Notice that the assumption that every derivable sequent is type-uniform (C$_9$), and the type-alikeness of parameters (C'$_2$) imply that the sequent $a_1$, $a_k$ and $x$ have the same type. Hence, C'$_7$ implies that it is possible to
substitute $x$ for $a_k$ by means of an application of the same rule \emph{Ru}. That is,
$(x_k \vdash y_k)[\overline{a}_k]$ can be replaced by $(x_k \vdash y_k)[\overline{x}/\overline{a}_k]$.

Let $a_l$ be introduced as a principal formula. The corresponding subcase in \cite{Wan02} splits into two subsubcases: either $a_l$ is introduced in display or it is not.

If $a_l $ is in display (as represented in the picture above in the conclusion of $\pi_{2.i}$ for $a_l = a_i$), then we form a subderivation using $\pi_1$  and $\pi_{2.i}$ and applying cut as the last rule. The assumptions that the original cut is strongly type-uniform (C$_{10}$), that every derivable sequent is type-uniform (C$_9$), and the type-alikeness of parameters (C'$_2$) imply that the sequent $\underline{a_i} \vdash y_i$ is of the same type as the sequents $x \vdash a$ and $a \vdash y$. Hence, the new cut is strongly type-uniform.

If $a_l $ is not in display (as represented in the picture above in the conclusion of $\pi_{2.j}$ for $a_l = a_j$), then condition C'$_5$ implies that $(x_j \vdash y_j)[\underline{a}_j]^{pre}$ is an axiom, and C''$_5$ implies that some axiom $\underline{a}_j \vdash y'_j$ exists, which is display-equivalent to the first axiom, and in which $a_j$ occurs in display.  Let $\pi'$ be the derivation which transforms $\underline{a}_j \vdash y'_i$ into $(x_j \vdash y_j)[\underline{a}_j]^{pre}$. We form a subderivation using $\pi_1$ and $\underline{a}_j \vdash y'_j$ and joining them with a cut application, then attaching $\pi'[x/a_j]^{pre}$ below the new cut.

The transformations just discussed explain how to transform the leaves of the history tree of $a$. Finally, since, as discussed above, $x$ has the same type of  $a$, condition C'$_7$ implies that substituting $x$ for each occurrence of $a$ in the history tree of the cut term $a$ in $\pi_2$ (and in each occurring $\pi'$ as above) gives rise to an admissible derivation $\pi_2[x/a]^{pre}$ (use C'$_6$ for the symmetric case).

\commment{
We illustrate this special situation in the diagrammatic proof below:

\begin{center}
\begin{tabular}{lcr}
\bottomAlignProof
\AXC{\ \ \ $\vdots$ \raisebox{1mm}{$\pi_1$}}
\noLine
\UI$X \fCenter A$
%
\AXC{$(X_i \fCenter Y_i) [\underline{A}_i]^{pre} [\underline{A}]^{suc}$}
\noLine
\UIC{\ \ \ $\vdots$ \raisebox{1mm}{$\pi_2$}}
\noLine
\UI$A\fCenter Y$
\BI$X \fCenter Y$
\DisplayProof

 & $\rightsquigarrow$ &

\bottomAlignProof
\AXC{\ \ \ $\vdots$ \raisebox{1mm}{$\pi_1$}}
\noLine
\UI$X \fCenter A$
%
\AXC{$\underline{A}_i \fCenter Y' [\underline{A}]^{suc}$}
\BI$X \fCenter Y' [A]^{suc}$
\noLine
\UIC{\ \ \ $\vdots$ \raisebox{1mm}{$\pi'$}}
\noLine
\UI$(X_i \fCenter Y_i) [X/A_i]^{pre} [A]^{suc}$
\noLine
\UIC{\ \ \ \ \ \ \ \ \ \ \ \ \ \ \ \ \ $\vdots$ \raisebox{1mm}{$\pi_2 [X/A]^{pre}$}}
\noLine
\UI$X \fCenter Y$
\DisplayProof
 \\
\end{tabular}
\end{center}
}
Summing up, this procedure generates the following proof tree:
\begin{center}
\AXC{\ \ \ $\vdots$ \raisebox{1mm}{$\pi_1$}}
\noLine
\def\fCenter{\vdash}
\UI$x \fCenter a$
\AXC{\ \ \ \ \ $\vdots$ \raisebox{1mm}{$\pi_{2.i}$}}
\noLine
\UI$\underline{a}_i \fCenter y_i$
\BI $ x \fCenter y_i$
\noLine
\UIC{\ \  $\ddots$}
\AXC{\ \ \ $\vdots$ \raisebox{1mm}{$\pi_1$}}
\noLine
\UI$x \fCenter a$
%
\AXC{$\underline{a}_j \fCenter y' $}
\BI$x \fCenter y' [a]^{suc}$
\noLine
\UIC{\ \ \ \ \ \ \ \ \ \ \ \ \ \ $\vdots$ \raisebox{1mm}{$\pi'[x/a]^{pre}$}}
\noLine
\UI$(x_j \fCenter y_j) [x/a_j]^{pre} $
\noLine
\UIC{\!\!$\vdots\ $}
\AXC{\ \ \ \ \ $\vdots$ \raisebox{1mm}{$\pi_{2.k}$}}
\noLine
\UI$(x_k \fCenter y_k)[\overline{x}/\overline{a}_k]^{pre}$
\noLine
\UIC{$\!\!\!\!\!\!\!\!\!\!\!\!\!\!\!\!\!\!\!\! \iddots$}
\noLine
\TIC{$\ \ \ \ \ \ \ \ \ \ \ \ \ \ \ \ \ \ \ \ \ \ \ \ \ \ \,\rule[-5.2mm]{0mm}{0mm}\ddots\vdots\iddots\rule{0mm}{10mm}$ \raisebox{1mm}{$\pi_2[x/a]^{pre}$}}
\noLine
\def\fCenter{\vdash}
\UIC{\ \ \ \ \ \ \ \ \ \ $x \fCenter y$}
\DisplayProof
\end{center}
We observe that in each newly introduced application of the cut rule, both cut terms are principal. Hence, we can apply the procedure described in the Principal stage and transform the original derivation in a derivation in which the cut terms of the newly introduced cuts have strictly lower complexity than the original cut terms. When the newly introduced applications of cut are of lower height than the original one, we do not need to resort to the Principal stage.\footnote{ This is for instance the case if, in the original derivation, the history-tree of the cut term $a$ (in the right-hand-side premise of the given cut application)  contains at most one leaf $a_l$ which is principal. 
However, the procedure described above in the Parametric stage does not always produce cuts of lower height.
For instance, in the calculus introduced in Section \ref{DC}, this situation may arise when two ancestors of a cut term of type \textsf{Fm} are introduced as principal along the same branch, and then are identified via an application of the rule Contraction.}

Finally, as to the subcase in which both cut terms are parametric, consider a proof with at least one cut. The procedure is analogous to the previous case. Namely, following the history of one of the cut terms up to the leaves, and applying the transformation steps described above, we arrive at a situation in which, whenever new applications of cuts are generated, in each such application at least one of the cut formulas is principal. To each such cut, we can apply (the symmetric version of) the Parametric stage described so far.

\end{proof}


\section{The Dynamic Calculus for EAK}\label{DC}
    
As mentioned in the introduction, the key idea is to introduce a language in which not only formulas are generated from formulas and actions (as it happens in the symbol $\langle\alpha \rangle A$) and formulas are generated from formulas and agents (as it happens in the symbol $\langle\aga \rangle A$), but also {\em actions} are generated from the interactions between {\em agents} and {\em actions}. 

\paragraph{An algebraically motivated introduction.} In the present section, we define a {\em multi-type language} into which the language of (I)EAK translates, and in which each generation step mentioned above is explicitly accounted for via special binary connectives taking arguments of {\em different types}. More than one alternative is possible in this respect; our choice for the present setting consists of the following types: $\mathsf{Ag}$ for agents, $\mathsf{Fnc}$ for functional actions, $\mathsf{Act}$ for actions, and $\mathsf{Fm}$ for formulas.
%
We also stipulate that $\mathsf{Ag}$, $\mathsf{Act}$,  $\mathsf{Fm}$ and $\mathsf{Fnc}$ are pairwise disjoint.
The new connectives, and their types, are:
\begin{eqnarray}
{\mand}_0, {\mband}_{\! 0} \ & :&  \mathsf{Fnc} \times \mathsf{Fm} \to \mathsf{Fm} \\
{\mand}_1, {\mband}_{\! 1} \ & :&  \mathsf{Act} \times \mathsf{Fm} \to \mathsf{Fm} \\
{\mand}_2, {\mband}_{\! 2} \  & :& \mathsf{Ag} \times \mathsf{Fm} \to \mathsf{Fm}\\
{\mand}_3, {\mband}_{\! 3} \ & :&  \mathsf{Ag} \times \mathsf{Fnc} \to \mathsf{Act}
\end{eqnarray}
We stipulate that the interpretations of the connectives are  maps preserving existing joins in each coordinate (see below) with  algebras as domains and codomains suitable to interpret (functional) actions, formulas, and agents respectively. For instance,  suitable choices for  domains of interpretation for formulas can be  complete atomic Boolean algebras or  perfect Heyting algebras (cf.\ \cite{KP}); in the setting of e.g.\ epistemic action logic (cf.\ \cite{Ditmarsch1}), following \cite{Alexandru}, the domain of interpretation for actions can be a quantale or a relation algebra (of which the functional actions can be a sub-monoid). 
In the setting of EAK, in which no algebraic structure is required of actions and agents, a suitable domain of interpretation can be a complete join-semilattice, which is completely join-generated by a given subset (interpreting the functional actions),   and  the domain of interpretation of agents can be a set.
\footnote{Notice also that for other dynamic logics the domain of interpretation of agents might be endowed with some algebraic structure; for instance, in the case of game logic (cf.\ \cite{ParikhPauli}), the set of agents consists of two elements, on which a negation-type operation can be assumed.}

In Section \ref{sec:soundness}, the final coalgebra $Z$ (cf.\ \cite[Section 5]{GAV}) is taken as semantic environment for the Dynamic Calculus. In this setting, the boolean algebra $\mathcal{P} Z$ is taken as the domain of interpretation for  $\mathsf{Fm}$-type terms,
$\mathsf{Fnc}$-type terms are interpreted as graphs of partial functions on $Z$, subject to certain restrictions, and the domain of interpretation of $\mathsf{Act}$-type terms is the complete $\bigcup$-semilattice generated by the domain of interpretation of $\mathsf{Fnc}$.

%
%

In all the domains of interpretation which are complete lattices (i.e.\ the algebras interpreting terms of type $\mathsf{Fm}$ and $\mathsf{Act}$), the fact that the interpretation of each connective $\mand$ and $\mband$  is completely join-preserving in its second coordinate
%
implies  that it
 has a right adjoint in its second coordinate. These right adjoints provide natural interpretation for   the following additional connectives:
 \begin{eqnarray}
{\mbra}_{\! 0} \ , {\mra}_{\! 0} \ & :&  \mathsf{Fnc} \times \mathsf{Fm} \to \mathsf{Fm} \\
{\mbra}_{\! 1} \ , {\mra}_{\! 1} \ & :&  \mathsf{Act} \times \mathsf{Fm} \to \mathsf{Fm} \\
{\mbra}_{\! 2} \ , {\mra}_{\! 2} \  & :& \mathsf{Ag} \times \mathsf{Fm} \to \mathsf{Fm}.
\end{eqnarray}
\noindent The assumptions above imply that ${\mand}_1$ and $\mband_1$ have right adjoints also   in their first coordinate. Hence, each of the following connectives can be naturally interpreted, in the setting above,  as the right adjoint of ${\mand}_1$ and $\mband_1$ respectively:

 \begin{eqnarray}
{\mbla}_{\! 1} \ , {\nla}_{\! 1} \ & :&  \mathsf{Fm}  \times \mathsf{Fm} \to \mathsf{Act}.
\end{eqnarray}
\noindent Intuitively, for all formulas $A, B$, the term $B {\mbla}_{\! 1}   A$ denotes the weakest epistemic action $\gamma$ such that, if $A$ was true before $\gamma$ was performed, then $B$ is true after any successful execution of $\gamma$. This is also related to to Vaughn Pratt's notion of weakest preserver (cf.\ \cite[Section 4.2]{Pra91})
\noindent However, we cannot assume that more adjoints exist, which would provide semantic interpretation for the following symbols:

 \begin{eqnarray*}
{\msbla}_{\! 0} \ , {\msla}_{\! 0} \ & :&  \mathsf{Fm}  \times \mathsf{Fm} \to \mathsf{Fnc}\\
{\msbla}_{\! 2} \ , {\msla}_{\! 2} \  & :& \mathsf{Fm}\times \mathsf{Fm} \to \mathsf{Ag}\\
{\msbla}_{\! 3} \ , {\msla}_{\! 3} \ & :&  \mathsf{Act} \times \mathsf{Fnc} \to \mathsf{Ag}\\
{\msbra}_{\! 3} \ , {\msra}_{\! 3} \ & :&  \mathsf{Ag} \times \mathsf{Act} \to \mathsf{Fnc}.
\end{eqnarray*}


\paragraph{Virtual adjoints.}\label{virtual adjoints} We adopt the following notational convention about the  three different shapes of arrows  introduced so far.
Arrows with  straight tails ($\mra$ and $\mbra$) stand for connectives which have a semantic counterpart and which are included in the language of the Dynamic Calculus (see the grammar of operational terms on page \pageref{page:def-formulas});
arrows with no tail  (e.g.\ $\mbla$ and $\nla$) do have a semantic interpretation but are {\em not} included in the language at the operational level, and
arrows with  squiggly tails ($\msra$, $\msla$, $\msbra$ and $\msbla$) stand for syntactic objects, called \textit{virtual adjoints}, which do not have a semantic interpretation, but  will play an important role, namely guaranteeing the Dynamic Calculus to enjoy the relativized display property (cf.\ Definition \ref{def: relativized display prop}). 
%
In what follows, virtual adjoints will be introduced {\em only} as structural connectives. That is,  they will not correspond to any operational connective, and they will not appear actively in any rule schema other than the display postulates (cf.\ Definition \ref{def: display prop}). 
As will be shown in Section \ref{sec : safe virtual adjoints}, these limitations  keep the calculus sound even if virtual adjoints do not have an independent semantic interpretation. 

\noindent The ${\mand}\dashv \mbra$ and $\mband\dashv \mra$ adjunction relations stipulated above translate into the following clauses  for every agent $\aga$, every functional action $\alpha$, every action $\gamma$, and every formula $A$:
\begin{eqnarray}
\alpha \mand_0 A\leq B\ \mbox{ iff }\ A\leq \alpha\mbra_{\!0} \ B  &\quad&  \alpha \mband_{\! 0} \  A\leq B\ \mbox{ iff }\ A\leq \alpha{\mra}_{\!0} \ B
\label{adjunction:heterogeneous0}
\\
\label{adjunction:heterogeneous}
\gamma \mand_1 A\leq B\ \mbox{ iff }\ A\leq \gamma\mbra_{\!1} \ B  &\quad&  \gamma \mband_{\! 1} \  A\leq B\ \mbox{ iff }\ A\leq \gamma{\mra}_{\!1} \ B
\\
\label{adjunction:heterogeneous1}
\aga \mand_2 A\leq B\ \mbox{ iff }\ A\leq \aga\mbra_{\!2} \ B  &\quad&  \aga \mband_{\! 2} \  A\leq B\ \mbox{ iff }\ A\leq \aga{\mra}_{\!2} \ B.
\end{eqnarray}
\noindent The adjunction relations ${\mand_1}\dashv \mbla_1$ and $\mband_{\! 1} \ \dashv \nla_1$ translate into the
following clauses for every action $\gamma$ and every formula $A$:
\begin{eqnarray}
\label{adjunction:heterogeneous:bis}
\gamma \mand_1 A\leq B\ \mbox{ iff }\
 \gamma \leq B \mbla_{\!1} A  &\quad&  \gamma \mband_{\! 1} \  A\leq B\ \mbox{ iff }\ \gamma \leq B \nla_{\!1} A.
\end{eqnarray}

As we will see, the display postulates corresponding to triangle- and arrow-shaped connectives are modelled over the conditions \eqref{adjunction:heterogeneous0}-\eqref{adjunction:heterogeneous:bis} above. Also the display postulates involving virtual adjoints are shaped in the same way, which explains their name.

\commment{
\noindent Here below, the conditions relative to the adjunction in the first coordinate for $i=0,2,3$: 
\begin{eqnarray}
\label{adjunction:heterogeneous-weird}
\alpha \mand_0 A\leq B\ \mbox{ iff }\ \alpha \leq B \msbla_{\!0} A  &\quad&  \alpha \mband_{\! 0} \  A\leq B\ \mbox{ iff }\ \alpha \leq B{\msla}_{\!0} A
\\
\label{adjunction:heterogeneous1-weird}
\aga \mand_2 A\leq B\ \mbox{ iff }\ \aga\leq B \msbla_{\!2} A  &\quad&  \aga \mband_{\! 2} \  A\leq B\ \mbox{ iff }\ \aga\leq B {\msla}_{\!2} A
\\
\label{adjunction:heterogeneous2-weird}
\aga \mand_3 \alpha\leq \gamma \ \mbox{ iff }\ \aga\leq \gamma \msbla_{\!3} \alpha  &\quad&  \aga \mband_3\alpha\leq \gamma \ \mbox{ iff }\ \aga\leq \gamma {\msla}_{\!3} \alpha\\
\label{adjunction:heterogeneous3-weird}
\aga \mand_3 \alpha\leq \beta\ \mbox{ iff }\ \alpha\leq \aga\msbra_{\!3} \beta  &\quad&  \aga \mband_3\alpha\leq \beta\ \mbox{ iff }\ \alpha\leq \aga{\msra}_{\!3} \beta.
\end{eqnarray}
}
\paragraph{Translating D'.EAK into the multi-type setting.} \label{translation DEAK to DC}
The intended link between the language of D'.EAK (cf.\ Section \ref{ssec:D'.EAK prelim})
and the language of the Dynamic Calculus is illustrated in the following table: 
\begin{center}
\begin{tabular}{c c c || c c c}
$\langle\alpha \rangle A$   & becomes & $\alpha\mand_0 A$ & $\RESalphaDia A$  & becomes & $\alpha\mband_{\! 0} \  A$ \\
 $\langle\aga \rangle A$ & becomes & $\aga\mand_2 A$ & $\RESagaDia A$  & becomes &$\aga\mband_{\! 2} \  A$\\
 $[\alpha] A$  & becomes & $\alpha\mra_{\! 0} \ A$ & $\RESalphaBox A$  & becomes &  $\alpha\mbra_{\! 0} \ A$ \\
 $[\aga] A$  & becomes &$\aga\mra_{\! 2} \ A$ & $\RESagaBox A$   & becomes &  $\aga\mbra_{\! 2} \ A$ \\
$1_\alpha$ & becomes & $\alpha \mand_0 \top$.
\end{tabular}
\end{center}
The table above can be extended to the definition of a formal translation between the operational language of D'.EAK and that of the Dynamic Calculus, simply by preserving the non modal propositional fragment.
We omit the details of this straightforward inductive definition.
In Section \ref{sec:soundness}, 
this translation will be elaborated on, and the interpretation of the language of the Dynamic Calculus in the final coalgebra will be defined so that the translation above preserves the validity of sequents.
In the light of this translation, the adjunction conditions in clauses (\ref{adjunction:heterogeneous0})  correspond to the   adjunction conditions \eqref{eqnarray : adjunction D'EAK} in D'.EAK, which, in their turn, motivate the display postulates reported on in Section \ref{ssec:D'.EAK prelim}:
\[
\langle\alpha\rangle \dashv \RESalphaBox \quad \quad \RESalphaDia\dashv[\alpha].
\]
The connectives $\mand_3$ and $\mband_3$ have no counterpart in the language of D'.EAK, but  the introduction of $\mband_3$ is exactly what brings the additional expressiveness we need in order to eliminate the label.   Indeed, we stipulate that for every $\aga$ and $\alpha$ as above,
\begin{equation}
\label{def: mband3}
\aga\mband_{\! 3} \   \alpha = \bigvee\{\beta\mid \alpha\aga\beta\}.
\end{equation}
A way to understand this stipulation is in the light of the discussion in \cite[Section 4.3]{GAV}  after clause (8).
There, in the context of a discussion about the proof system  in \cite{Alexandru}, the link between the semantic condition $f^M_A(m \star q) \leq f^M_A(m) \star f^Q_A(q)$ (cf.\ \cite[Definitions 2.2(2) and 2.3]{Alexandru}) and the axiom (\ref{eq:interact axiom})---which in \cite{Alexandru} was left implicit---is made more explicit, by understanding the action $f^Q_A(q)$ as the join, taken in $Q$, of all the actions $q'$ which are indistinguishable from $q$ for the agent $A$.  In the present setting, the stipulation (\ref{def: mband3}) says that  $\aga\mband_{\! 3} \   \alpha$ encodes exactly the same information encoded in $f^Q_A(q)$,  namely, the nondeterministic choice between all the actions that are indistinguishable from $\alpha$ for the agent $\aga$. 

\paragraph{Additional conditions.} As was the case in the setting of D'.EAK, in order to express in this new language that e.g.~$\langle\alpha\rangle$ and $[\alpha]$ are ``interpreted over the same relation'', Sahlqvist correspondence theory (cf.\ e.g.\ \cite{ALBA, ConPalSou, CGP} for a state-of-the art-treatment) provides us with two alternatives: one of them is that we impose the following Fischer Servi-type conditions  to hold for every $\aga\in \mathsf{Ag}$, $\alpha\in \mathsf{Fnc}$, $\gamma\in \mathsf{Act}$ and $A, B\in \mathsf{Fm}$:
\begin{eqnarray*}
( \alpha \mand_0 A)\rightarrow (\alpha \mra_{\! 0} \  B)  \leq \alpha \mra_{\! 0} \  (A\rightarrow B)& &( \alpha \mband_{\! 0} \  A)\rightarrow (\alpha \mbra_{\! 0} \  B)  \leq     \alpha \mbra_{\! 0} \  (A\rightarrow B)          \\
( \gamma \mand_1 A)\rightarrow (\gamma \mra_{\! 1} \  B)  \leq \gamma \mra_{\! 1} \  (A\rightarrow B)& &( \gamma \mband_{\! 1} \  A)\rightarrow (\gamma \mbra_{\! 1} \  B)  \leq     \gamma \mbra_{\! 1} \  (A\rightarrow B)          \\
( \aga \mand_2 A)\rightarrow (\aga \mra_{\! 2} \  B)  \leq \aga \mra_{\! 2} \  (A\rightarrow B) & &  ( \aga \mband_{\! 2} \  A)\rightarrow (\aga \mbra_{\! 2} \  B)    \leq   \aga \mbra_{\! 2} \  (A\rightarrow B).\\
~\\
\alpha \mand_{ 0}   (A \pdra B)   \leq
( \alpha \mra_{\! 0}\ A) \pdra (\alpha \mand_{ 0}   B)
& &
  \alpha \mband_{\! 0} \   (A \pdra B)
\leq  ( \alpha \mbra_{\! 0} \  A) \pdra (\alpha \mband_{\! 0} \  B)           \\
 \gamma \mand_{ 1}   (A \pdra B) \leq ( \gamma \mra_{\! 1} \  A) \pdra (\gamma \mand_{ 1}   B)
& &
 \gamma \mband_{\! 1} \  (A \pdra B)  \leq  ( \gamma \mbra_{\! 1} \  A) \pdra (\gamma \mband_{\! 1} \  B)            \\
 \aga \mand_{ 2}   (A \pdra B) \leq ( \aga \mra_{\! 2} \ A) \pdra (\aga \mand_{ 2}   B)
 & &  \aga \mband_{\! 2} \  (A \pdra B) \leq  ( \aga \mbra_{\! 2} \  A) \pdra (\aga \mband_{\! 2} \  B)    .
\end{eqnarray*}

%
To see that the conditions above correspond to the usual Fischer Servi axioms in standard modal languages, one can observe that
 the conditions in the first and third line above are  images, under the translation discussed above, of the Fischer Servi axioms reported on in  Section \ref{ssec:IEAK}).
The second alternative is to impose that, for every $0\leq i\leq 2$, the connectives $\mand_i$ and $\mband_i$ yield  {\em conjugated} diamonds (cf.\ discussion in \cite[Section 6.2]{GAV}); that is, the following inequalities hold for all $\aga\in \mathsf{Ag}$, $\alpha, \beta\in \mathsf{Fnc}$, and $A, B\in \mathsf{Fm}$:
\begin{eqnarray*}
(\alpha\mand_0 A)\wedge B\leq \alpha\mand_0 (A\wedge \alpha \mband_{\! 0} \  B) &\quad  & (\alpha\mband_{\! 0} \  A)\wedge B\leq \alpha\mband_{\! 0} \  (A\wedge \alpha \mand_0 B)\\
(\gamma\mand_1 A)\wedge B\leq \gamma\mand_1 (A\wedge \gamma \mband_{\! 1} \  B) & & (\gamma\mband_{\! 1} \  A)\wedge B\leq \gamma\mband_{\! 1} \  (A\wedge \gamma \mand_1 B)\\
(\aga\mand_2 A)\wedge B\leq \aga\mand_2 (A\wedge \aga \mband_{\! 2} \  B) & & (\aga\mband_{\! 2} \  A)\wedge B\leq \aga\mband_{\! 2} \  (A\wedge \aga \mand_2 B).\\
~\\
\alpha\mra_0 (A\vee \alpha \mbra_{\! 0} \  B) \leq (\alpha\mra_0 A)\vee B
  &  &
   \alpha\mbra_{\! 0} \  (A\vee \alpha \mra_0 B)
  \leq (\alpha\mbra_{\! 0} \  A)\vee B
  \\
\gamma\mra_1 (A\vee \gamma \mbra_{\! 1} \  B)
\leq (\gamma\mra_1 A)\vee B
&  &
\gamma\mbra_{\! 1} \  (A\vee \gamma \mra_1 B)
\leq (\gamma\mbra_{\! 1} \  A)\vee B
\\
\aga\mra_2 (A\vee \aga \mbra_{\! 2} \  B)
\leq (\aga\mra_2 A)\vee B
&  &
\aga\mbra_{\! 2} \  (A\vee \aga \mra_2 B)
\leq (\aga\mbra_{\! 2} \  A)\vee B.
\end{eqnarray*}
The conditions in the first and third line above are  images, under the translation discussed above, of the conjugation  conditions reported on in \cite[Section 6.2]{GAV}.

\paragraph{The operational language, formally.} 
Let us introduce the operational terms of the multi-type language  by the following simultaneous induction, based on sets $\mathsf{AtProp}$ of atomic propositions,  $\mathsf{Fnc}$ of functional actions, and  $\mathsf{Ag}$ of agents: \label{page:def-formulas}
%
\begin{align*}
 \mathsf{Fm} \ni A:: = \; & p \mid \bot\mid \top \mid A\wedge A\mid A\vee A\mid A \rightarrow A\mid A \pdra A\mid \\
 & \alpha\mand_0 A\mid \alpha\mra_{\! 0} \  A  \mid
 \gamma\mand_1 A\mid \gamma\mra_{\! 1} \  A  \mid
 \aga\mand_2 A\mid \aga\mra_{\! 2} \  A \mid
 \\
&
 \alpha\mband_{\! 0} \  A\mid \alpha\mbra_{\! 0} \  A \mid
\gamma\mband_{\! 1} \  A\mid \gamma\mbra_{\! 1} \  A \mid
\aga\mband_{\! 2} \  A\mid \aga\mbra_{\! 2} \  A
\\
 \mathsf{Fnc} \ni \alpha:: = \; & \alpha \\
\mathsf{Act}\ni \gamma:: =\; & 
\aga\mband_{\! 3} \  \alpha \mid \aga\mand_3 \alpha \\
 \mathsf{Ag} \ni \aga:: = \;& \aga
\end{align*}
%
The fundamental difference between the language above and the language of D'.EAK is that, in D'.EAK, agents and actions are {\em parametric indexes} in the construction of formulas, which are the only first-class citizens. In the present setting, however, each type lives on a par with any other. Because of the relative simplicity of the EAK setting, two of the four types are attributed  no algebraic structure at the operational level. However, it is not difficult to enrich the algebraic structure of those types with sensible and intuitive operations: for instance, the \textsf{skip} and \textsf{crash} actions are functional, and parallel and sequential composition and iteration on functional actions preserve functionality, hence can be added to the array of constructors for \textsf{Fnc}. As a consequence of the fact that each type is a first-class citizen, as we will see shortly, four types of structures will be defined, and the turnstile symbol in the sequents of this calculus will be interpreted in the appropriate domain. 

\paragraph{On the meta-linguistic labels $\alpha\aga\beta$.} Let us illustrate how the label $\alpha\aga\beta$ can be subsumed when translating D'.EAK-formulas in the multi-type language. Consider for example  (the intuitionistic counterparts of) the following axiom (cf.~(\ref{eq:interact axiom})): 
\[ (\ Pre(\alpha)\rightarrow  \bigwedge\{[\aga]\ls\beta\rs  A \mid \alpha\aga\beta\} \ )
\rightarrow \ls\alpha\rs[\aga] A
 \] 
%
By applying the translation above we get:
\[ (\ \alpha \mand_0 \top \rightarrow  \bigwedge\{\aga \mra_{\! 2} \ (\beta \mra_{\! 0} \  A)\mid \alpha\aga\beta\} \ ) \rightarrow  \alpha \mra_{\! 0} \ (\aga \mra_{\! 2} \ A) .\]
Since (the semantic interpretation of) $\mra_{\! 2} \ $ is completely meet-preserving in the second coordinate, the clause above is semantically equivalent to the following one:
\[ (\ \alpha \mand_0 \top \rightarrow  \big[ \aga \mra_{\! 2} \  \bigwedge\{\beta \mra_{\! 0} \  A\mid \alpha\aga\beta\} \big] \ ) \rightarrow   \alpha \mra_{\! 0} \ (\aga \mra_{\! 2} \ A).\]
The next step is the only place of the chapter in which we will need to assume that (the domains of interpretation of $\mathsf{Fcn}$ and $\mathsf{Act}$ are such that) $\mathsf{Fcn}\subseteq \mathsf{Act}$. Under this assumption, $\mra_{\! 0} \ $ can be taken as  the restriction of $\mra_{\! 1} \ $. By general order-theoretic facts (see e.g.\ \cite{DaveyPriestley2002}), the latter is  completely join-reversing in its first coordinate. Hence, we can equivalently rewrite the clause above as follows:
\[ (\ \alpha \mand_0 \top \rightarrow  \big[ \aga \mra_{\! 2} \  \big( \bigvee\{\beta \mid \alpha\aga\beta\} \mra_{\! 1} \  A \big) \big]   \ )  \rightarrow  \alpha \mra_{\! 0} \ (\aga \mra_{\! 2} \ A).\]
Now we apply the stipulation (\ref{def: mband3}) and get the following :
\begin{equation}\label{DC:axiom1} 
(\ \alpha \mand_0 \top \rightarrow   \big[ \aga \mra_{\! 2} \  \big((\aga \mband_{\! 3} \  \alpha) \mra_{\! 1} \  A \big) \big]    \ ) \rightarrow    \alpha \mra_{\! 0} \ (\aga \mra_{\! 2} \ A).\end{equation}
An analogous argument justifies that the following axiom:
\[\lc \alpha \rc \lc\aga\rc A \rightarrow (\ Pre(\alpha) \wedge \bigvee\{\lc\aga\rc \lc\beta \rc A\mid \alpha\aga\beta\} \ ) \]
corresponds to:
\begin{equation}\label{DC:axiom2}\alpha \mand_0( \aga\mand_2 A) \rightarrow ( \ \alpha \mand_0 \top \wedge   \aga\mand_2 [(\aga \mband_{\! 3} \  \alpha)\mand_1 A] \ ).\end{equation}
Without appealing to $\mathsf{Fcn}\subseteq \mathsf{Act}$, we  could take  the correspondences above as primitive stipulations.

\paragraph{Structural language, formally.} As discussed in the preliminaries, display calculi manipulate two closely related languages: the operational and the structural. Let us introduce the structural language of the Dynamic Calculus, which as usual matches the operational language, although in the present case not in the same way as in D'.EAK. We have {\em formula-type} structures, {\em functional action-type structures},  {\em action-type structures}, {\em agent-type structures}, defined by simultaneous recursion as follows: 
\begin{align*}
 \mathsf{FM} \ni X:: = \; & A \mid \textrm{I} \mid X~; X\mid X > X\mid \\
 & F \mAND_{\! 0} \ X\mid F \mRA_{\!\! 0} \  X \mid
 \Gamma\mAND_{\! 1} \  X\mid \Gamma\mRA_{\!\! 1} \  X \mid
 \agA\mAND_{\! 2} \  X\mid \agA\mRA_{\!\! 2} \  X \mid\\
 & F \mBAND_{\! 0} \  X\mid F \mBRA_{\!\! 0} \  X \mid
\Gamma\mBAND_{\! 1} \  X\mid \Gamma\mBRA_{\!\! 1} \  X \mid
 \agA\mBAND_{\! 2} \  X\mid \agA\mBRA_{\!\! 2} \  X
 \\
~\\
\mathsf{FNC} \ni F:: = \; & \alpha\mid X \mSLA_{\!\! 0} \   X \mid X \mSBLA_{\!\! 0} \  X\mid  \agA \mSRA_{\!\! 3} \  \Gamma\mid \agA \mSBRA_{\!\! 3} \  \Gamma\\
~\\
\mathsf{ACT} \ni \Gamma:: =\; & 
\agA\mBAND_{\! 3} \  F  \mid
\agA\mAND_{\! 3} \  F \mid  X \mLA_{\! 1} \   X \mid X \mBLA_{\! 1} \  X \\
~\\
\mathsf{AG} \ni \agA:: =  \; &\aga \mid X \mSLA_{\!\! 2} \   X \mid X \mSBLA_{\!\! 2} \  X \mid \Gamma \mSLA_{\!\! 3} \  F \mid \Gamma \mSBLA_{\!\! 3} \  F.
\end{align*}

\paragraph{The propositional base.} As is typical of display calculi, each operational connective corresponds to one structural connective. In particular, the propositional base connectives behave exactly as in D'.EAK, but for the sake of self-containment, we are going to report on these rules below: 

\begin{center}
\begin{tabular}{|r|c|c|c|c|c|c|c|c|c}
 \hline
 \scriptsize{Structural symbols} & \mc{2}{c|}{$<$}   & \mc{2}{c|}{$>$} & \mc{2}{c|}{$;$} & \mc{2}{c|}{I}    \\
 \hline
 \scriptsize{Operational symbols} & $\pdla$ & $\leftarrow$ & $\pdra$ & $\rightarrow$  & $\wedge$ & $\vee$  & $\top$ & $\bot$    \\
\hline
\end{tabular}

\end{center}
\vspace{10px}

{\scriptsize{
\begin{center}
\begin{tabular}{rl}

\mc{2}{c}{\textbf{Structural Rules}} \\

& \\

\AXC{\phantom{$p \fCenter p$}}
\LeftLabel{\fns$Id$}
\UI$p \fCenter p$
\DisplayProof &
\AX$X \fCenter A$
\AX$A \fCenter Y$
\RightLabel{\fns$Cut$}
\BI$X \fCenter Y$
\DisplayProof \\

 & \\

\AX$X \fCenter Y$
\doubleLine
\LeftLabel{\fns$\textrm{I}^1_{L}$}
\UI$\textrm{I} \fCenter Y < X$
\DisplayProof &
\AX$X \fCenter Y$
\doubleLine
\RightLabel{\fns$\textrm{I}^1_{R}$}
\UI$X < Y \fCenter \textrm{I}$
\DisplayProof \\

 & \\

\AX$X \fCenter Y$
\LeftLabel{\fns$\textrm{I}^2_{L}$}
\doubleLine
\UI$\textrm{I} \fCenter X > Y$
\DisplayProof &
\AX$X \fCenter Y$
\RightLabel{\fns$\textrm{I}^2_{R}$}
\doubleLine
\UI$Y > X \fCenter \textrm{I}$
\DisplayProof \\

 & \\

\AX$\textrm{I}  \fCenter X$
\LeftLabel{\fns$\textrm{I} W_{L}$}
\UI$Y \fCenter X $
\DisplayProof &
\AX$X \fCenter \textrm{I}$
\RightLabel{\fns$\textrm{I} W_{R}$}

\UI$X \fCenter Y$
\DisplayProof \\

 & \\

\AX$X \fCenter Z$
\LeftLabel{\fns$W^1_{L}$}
\UI$Y \fCenter Z < X$
\DisplayProof &
\AX$X \fCenter Z$
\RightLabel{\fns$W^1_{R}$}
\UI$X < Z \fCenter Y$
\DisplayProof \\

 & \\

\AX$X \fCenter Z$
\LeftLabel{\fns$W^2_{L}$}
\UI$Y \fCenter X > Z$
\DisplayProof &
\AX$X \fCenter Z$
\RightLabel{\fns$W^2_{R}$}
\UI$Z > X \fCenter Y$
\DisplayProof \\

 & \\

\AX$X \,; X \fCenter Y$
\LeftLabel{\fns$C_L$}
\UI$X \fCenter Y $
\DisplayProof &
\AX$Y \fCenter X \,; X$
\RightLabel{$C_R$}
\UI$Y \fCenter X$
\DisplayProof \\

 & \\

\AX$Y \,; X \fCenter Z$
\LeftLabel{\fns$E_L$}
\UI$X \,; Y \fCenter Z $
\DisplayProof &
\AX$Z \fCenter X \,; Y$
\RightLabel{\fns$E_R$}
\UI$Z \fCenter Y \,; X$
\DisplayProof \\

 & \\

\AX$X \,; (Y \,; Z) \fCenter W$
\LeftLabel{\fns$A_{L}$}
\UI$(X \,; Y) \,; Z \fCenter W $
\DisplayProof &
\AX$W \fCenter (Z \,; Y) \,; X$
\RightLabel{\fns$A_{R}$}
\UI$W \fCenter Z \,; (Y \,; X)$
\DisplayProof \\

 & \\
\end{tabular}
\end{center}
}
}

{\scriptsize{
\begin{center}
\begin{tabular}{rl}

\mc{2}{c}{\textbf{Display Postulates}} \\

 & \\

\AX$X\, ; Y \fCenter Z$
\LeftLabel{\fns$(\;, <)$}
\doubleLine
\UI$X \fCenter Z < Y$
\DisplayProof &
\AX$Z \fCenter X\ ; Y$
\RightLabel{\fns$(<, ;)$}
\doubleLine
\UI$Z < Y \fCenter X$
\DisplayProof \\

 & \\

\AX$X\, ; Y \fCenter Z$
\LeftLabel{\fns$(\;, >)$}
\doubleLine
\UI$Y \fCenter X > Z$
\DisplayProof &
\AX$Z \fCenter X\, ; Y$
\RightLabel{\fns$(>, ;)$}
\doubleLine
\UI$X > Z \fCenter Y$
\DisplayProof \\

 & \\
\end{tabular}
\end{center}
}
}

\noindent The classical base is obtained by adding the so-called \emph{Grishin} rules (following e.g.~\cite{Gore}), which encode  classical, but not intuitionistic validities:
\label{page : Grishin rules}
{\scriptsize{
\begin{center}
\begin{tabular}{rl}


 & \\

\AX$X\, > ( Y\ ;Z) \fCenter W$
\LeftLabel{\fns$Gri_L$}
\doubleLine
\UI$(X>Y)\ ; Z \fCenter W$
\DisplayProof &
\AX$W \fCenter X >(Y\ ; Z)$
\RightLabel{\fns$Gri_R$}
\doubleLine
\UI$W \fCenter (X> Y)\ ;Z$
\DisplayProof \\

 & \\
\end{tabular}
\end{center}
}
}

{\scriptsize{
\begin{center}
\begin{tabular}{rl}
\mc{2}{c}{\textbf{Operational Rules}} \\
 & \\

\AXC{\phantom{$\bot \fCenter \textrm{I}$}}
\LeftLabel{\fns$\bot_L$}
\UI$\bot \fCenter \textrm{I}$
\DisplayProof &
\AX$X \fCenter \textrm{I}$
\RightLabel{\fns$\bot_R$}
\UI$X \fCenter \bot$
\DisplayProof \\

 & \\

\AX$\textrm{I} \fCenter X$
\LeftLabel{\fns$\top_L$}
\UI$\top \fCenter X$
\DisplayProof
 &
\AXC{\phantom{$\textrm{I} \fCenter \top$}}
\RightLabel{\fns$\top_R$}
\UI$\textrm{I} \fCenter \top$
\DisplayProof\\

 & \\

\AX$A \,; B \fCenter Z$
\LeftLabel{\fns$\pand_L$}
\UI$A \pand B \fCenter Z$
\DisplayProof &
\AX$X \fCenter A$
\AX$Y \fCenter B$
\RightLabel{\fns$\pand_R$}
\BI$X \,; Y \fCenter A \pand B$
\DisplayProof
\ \, \\

 & \\
\AX$A \fCenter X$
\AX$B \fCenter Y$
\LeftLabel{\fns$\vee_L$}
\BI$A \vee B \fCenter X \,; Y$
\DisplayProof &
\AX$Z \fCenter A \,; B$
\RightLabel{\fns$\vee_R$}
\UI$Z \fCenter A \vee B $
\DisplayProof \\

 & \\
\ \,

\AX$B \fCenter Y  $
\AX$X \fCenter A  $
\LeftLabel{\fns$\leftarrow_L$}
\BI$B \leftarrow A \fCenter Y<X$
\DisplayProof &
\AX$Z \fCenter B<A$
\RightLabel{\fns$\leftarrow_R$}
\UI$Z \fCenter B \leftarrow A$

\DisplayProof \\

 & \\
\ \,
\AX$B <A \fCenter Z$
\LeftLabel{\fns$\pdla_L$}
\UI$B \pdla A \fCenter Z$
\DisplayProof &
\AX$Y \fCenter B$
\AX$ A \fCenter X$
\RightLabel{\fns$\pdla_R$}
\BI$ Y <  X  \fCenter B \pdla A $
\DisplayProof \\

 & \\
\ \,
\AX$X \fCenter A$
\AX$B \fCenter Y$
\LeftLabel{\fns$\rightarrow_L$}
\BI$A \rightarrow B \fCenter X > Y$
\DisplayProof &
\AX$Z \fCenter A > B$
\RightLabel{\fns$\rightarrow_R$}
\UI$Z \fCenter A \rightarrow B $
\DisplayProof \\

 & \\
\ \,
\AX$A> B\fCenter Z$
\LeftLabel{\fns$\pdra_L$}
\UI$A \pdra B \fCenter Z$
\DisplayProof &
\AX$A \fCenter X$
\AX$ Y  \fCenter B$
\RightLabel{\fns$\pdra_R$}
\BI$X>Y \fCenter A \pdra B $
\DisplayProof \\

 & \\

\end{tabular}
\end{center}
}
}

\paragraph{Rules for heterogeneous connectives.} Unlike what was the case in the setting of D'.EAK, in  the present setting, each heterogeneous structural connective is associated with at most one operational connective, as illustrated in the following table: for $0\leq i\leq 3$ and $j\in \{0, 2\}$,

\begin{center}
\begin{tabular}{| c | c|c|c | c | c|c|c|c|c|c|c|c|c|c|c|}
\hline
\scriptsize{Structural symbols}  &
\mc{2}{c|}{$\mAND_{\! i}$}&\mc{2}{c|}{ $\mBAND_{\! i}$}&\mc{2}{c|}{$\mRA_{\!\! j}$}&\mc{2}{c|}{ $\mBRA_{\!\! j} $} \\
\hline
\scriptsize{Operational symbols} &
$\mand_{ i}$&\phantom{$\mand_{ i}$}& $\mband_{\! i}$&\phantom{$\mand_{\! i}$}&\phantom{$\mra_{\!\! j}$}&$\mra_{\!\! j}$&\phantom{$\mra_{\!\! j}$}& $\mbra_{\!\! j}$ \\
\hline
\end{tabular}
\end{center}
That is,  structural connectives are to be interpreted as usual in a context-sensitive way, but the present language  lacks the operational connectives which would correspond to them on one of the two sides. This is of course because in the present setting we do not need them.
However, in a setting in which they would turn out to be needed, it would not be difficult to introduce the missing operational connectives.
We can now introduce the operational rules for  heterogeneous connectives. Let $x, y$ stand for structures of an undefined type, and let $a, b$ denote operational terms of the appropriate type.  Then, for $0\leq i\leq 3,$

\begin{center}
\begin{tabular}{cc}
\\
\AX$a \mAND_{\! i} \ b \fCenter z$
\LeftLabel{\fns${\mand_i}_L$}
\UI$a \mand_i b \fCenter z$
\DisplayProof &
\AX$x \fCenter a$
\AX$y \fCenter b$
\RightLabel{\fns${\mand_i}_R$}
\BI$x \mAND_{\! i} \  y \fCenter a \mand_i b$
\DisplayProof
\\

\\
\AX$a \mBAND_{\! i} \  b \fCenter z$
\LeftLabel{\fns${\mband_{\! i}}_L$}
\UI$a \mband_{\! i} \  b \fCenter z$
\DisplayProof &
\AX$x \fCenter a$
\AX$y \fCenter b$
\RightLabel{\fns${\mband_{\! i}}_R$}
\BI$x \mBAND_{\! i} \  y \fCenter a \mband_{\! i} \  b$
\DisplayProof
\\
\end{tabular}
\end{center}

and for $0\leq i\leq 2,$

\begin{center}
\begin{tabular}{cc}
\\
\AX$x \fCenter a$
\AX$B \fCenter Y$
\LeftLabel{\fns${\mra_{\!\!i}}_L$}
\BI$a {\mra_{\!\!i}} B \fCenter x {\mRA_{\!\!i}} Y$
\DisplayProof &
\AX$Z \fCenter a \mRA_{\!\!i} B$
\RightLabel{\fns${\mra_{\!\!i}}_R$}
\UI$Z \fCenter a \mra_{\!\!i} B $
\DisplayProof \\

\\

\AX$x \fCenter a$
\AX$B \fCenter Y$
\LeftLabel{\fns${\mbra_{\!\!i}}_L$}
\BI$a {\mbra_{\!\!i}} B \fCenter x {\mBRA_{\!\!i}} Y$
\DisplayProof &
\AX$Z \fCenter a \mBRA_{\!\!i} B$
\RightLabel{\fns${\mbra_{\!\!i}}_R$}
\UI$Z \fCenter a \mbra_{\!\!i} B $
\DisplayProof \\
\end{tabular}
\end{center}
where $B, Y, Z$ are formula-type operational and structural terms. Clearly, the rules in the two tables above for $i = 0, 2$ yield the operational rules for the dynamic and epistemic modal operators under the translation given early on. Notice that each sequent is always interpreted in one domain. However, since heterogeneous connectives take arguments of different types (which justifies their name), premises of binary rules are of course interpreted in different domains.

\noindent Axioms  will be given in three types\footnote{Indeed, there is no axiom schema for atomic terms of type $\mathsf{Act}$, because the language does not admit them.}, as follows:

\begin{center}
\begin{tabular}{c c c c c c c}
$\aga \vdash \aga$ & $\alpha \vdash \alpha$  & $p\vdash p$  & $\bot \vdash \textrm{I} $& $ \textrm{I} \vdash \top$\\

\\

\end{tabular}
\end{center}
\noindent where  the first and second axioms from the left are of type $\mathsf{Ag}$ and $\mathsf{Fnc}$ respectively, and the remaining ones are of type $\mathsf{Fm}$. A generalization of $p\vdash p$ will be added below to the system (see {\em atom} axiom on page \pageref{page:atom}). 

\noindent Further, we allow the following {\em strongly type-uniform} (cf.\ Definition \ref{def:strong-type-uniformity}) cut rules on operational terms:
\[
\AX$\agA \fCenter \aga$
\AX$\aga \fCenter \agB$
\BI$\agA \fCenter \agB$
\DisplayProof
\quad
\quad
\AX$F \fCenter \alpha$
\AX$\alpha \fCenter G$
\BI$F \fCenter G$
\DisplayProof
\quad
\quad
\AX$\Gamma \fCenter \gamma$
\AX$\gamma \fCenter \Delta$
\BI$\Gamma \fCenter \Delta$
\DisplayProof
\quad
\quad
\AX$X\fCenter A$
\AX$A \fCenter Y$
\BI$X \fCenter Y$
\DisplayProof
\]

\noindent Next,  we give the display postulates for heterogeneous connectives. In what follows, let $x, y, z$ stand for structures of an undefined type. Then, for $0\leq i\leq 2,$

\begin{center}
\begin{tabular}{c c}
\AX$x \mAND_{\!i} \  y\fCenter z$
\LeftLabel{\scriptsize{$(\mand_i , \mbra_{\!\!i})$}}
\doubleLine
\UI$y \fCenter x {\mBRA_{\!\!i}} \ z$
\DisplayProof &
\AX$x \mBAND_{\!i} \   y\fCenter z$

\RightLabel{\scriptsize{$(\mband_i , \mra_{\!\!i})$}}
\doubleLine
\UI$y \fCenter x {\mRA_{\!\!i}} \ z$

\DisplayProof \\
\end{tabular}
\end{center}

\noindent For  $i = 1$, we also have

\begin{center}
\begin{tabular}{c c}
\AX$x \mAND_{\! 1} \  y\fCenter z$
\LeftLabel{\scriptsize{$(\mand_1 , \bla_{\!\!1})$}}
\doubleLine
\UI$x \fCenter z {\mBLA_{\!\!1}} \ y$
\DisplayProof &
\AX$x \mBAND_{\! 1} \  y\fCenter z$

\RightLabel{\scriptsize{$(\mband_{\! 1} \  , \mla_{\!\!1})$}}
\doubleLine
\UI$x \fCenter z {\mLA_{\!\!1}} \ y$

\DisplayProof \\
\end{tabular}
\end{center}

\noindent The display postulates above involve structural connectives each of which has a semantic interpretation. In the following display postulates, the squiggly arrows are not semantically justified: they are the \textit{virtual adjoints}, informally introduced at the beginning of the present Section \ref{virtual adjoints}, which  will be discussed in detail in Section \ref{sec : safe virtual adjoints}.
For each $i = 0,2,3$, we have:

\begin{center}
\begin{tabular}{c c}
\AX$x \mAND_{\!i} \  y \fCenter z$
\LeftLabel{\scriptsize{$(\mand_i , \msbla_{\!\!i})$}}
\doubleLine
\UI$x \fCenter z {\mSBLA_{\!\!i}} \ y$
\DisplayProof &
\AX$x \mBAND_{\!i} \   y\fCenter z$

\RightLabel{\scriptsize{$(\mband_i , \msla_{\!\!i})$}}
\doubleLine
\UI$x \fCenter z {\mSLA_{\!\!i}} \ y$

\DisplayProof \\
\end{tabular}
\end{center}
and for $i = 3$,

\begin{center}
\begin{tabular}{c c}
\AX$x \mAND_{\! 3} \  y\fCenter z$
\LeftLabel{\scriptsize{$(\mand_3 , \msbra_{\!\!3})$}}
\doubleLine
\UI$y \fCenter x {\mSBRA_{\!\!3}} \ z$
\DisplayProof &
\AX$x \mBAND_{\! 3} \  y\fCenter z$

\RightLabel{\scriptsize{$(\mband_{\! 3} \  , \msra_{\!\!3})$}}
\doubleLine
\UI$y \fCenter x {\mSRA_{\!\!3}} \ z$

\DisplayProof \\
\end{tabular}
\end{center}

Notice that sequents occurring in each display postulate involving heterogeneous connectives are {\em not} of the same type. However, it is easy to see that the display postulates preserve the type-uniformity (cf.\ Definition \ref{def:type-uniformity}); that is, if the premise of any instance of a display postulate is a type-uniform sequent, then so is its conclusion.
Next, the {\em necessitation}, {\em conjugation}, {\em Fischer Servi}, and {\em monotonicity} rules: for $0\leq i \leq 2$, 

\begin{center}
\begin{tabular}{c c}

\AX$\textrm{I}\fCenter W$
\LeftLabel{\scriptsize{$({nec_i} \mand)$}}

\UI$ x \mAND_{\!i} \   \textrm{I} \fCenter W$
\DisplayProof &
\AX$W \fCenter \textrm{I} $
\RightLabel{\scriptsize{$({nec_i} \mra)$}}

\UI$ W \fCenter x \mRA_{\!\!i} \ \textrm{I}$
\DisplayProof \\

\\
\AX$\textrm{I}\fCenter W$
\LeftLabel{\scriptsize{$({nec_i} \mband)$}}

\UI$x \mBAND_{\!i} \   \textrm{I} \fCenter W$
\DisplayProof &
\AX$W \fCenter \textrm{I} $
\RightLabel{\scriptsize{$({nec_i} \mbra)$}}

\UI$W \fCenter x \mBRA_{\!\!i} \ \textrm{I}$
\DisplayProof \\

 \\

\AX$x \mAND_{\!\!i\,} ((x \mBAND_{\!\!i\,} Y) \,; Z)\fCenter W$
\LeftLabel{\scriptsize{$({conj_i} \mand)$}}
\UI$Y\ ; (x\mAND_{\!\!i\,} Z) \fCenter W$
\DisplayProof
&
\AX$W \fCenter x \mRA_{\!\!i\,} ((x \mBRA_{\!\!i\,} Y) \,; Z)$
\RightLabel{\scriptsize{$({conj_i} \mra)$}}
\UI$W \fCenter  Y\ ; (x\mRA_{\!\!i\,} Z) $
\DisplayProof
\\

\\

\AX$x \mBAND_{\!\!i\,} ((x \mAND_{\!\!i\,} Y) \,; Z) \fCenter W$
\LeftLabel{\scriptsize{$({conj_i} \mband)$}}
\UI$Y\ ; (x\mBAND_{\!\!i\,} Z) \fCenter W$
\DisplayProof
 &
\AX$W \fCenter  x \mBRA_{\!\!i\,} ((x \mRA_{\!\!i\,} Y) \,; Z) $
\RightLabel{\scriptsize{$({conj_i} \mbra)$}}
\UI$W \fCenter  Y\ ; (x\mBRA_{\!\!i\,} Z) $
\DisplayProof \\

\\

\AX$ ( x \mRA_{\!\! i} \ Y)> (x \mAND_{\! i} \  Z)  \fCenter W$
\LeftLabel{\scriptsize{$({FS_{\! i}}\mand)$}}
\UI$ x \mAND_{\! i} \  (Y > Z)  \fCenter W$
\DisplayProof
&

\AX$W \fCenter ( x \mAND_{\! i} \  Y)> (x \mRA_{\!\! i} \  Z)  $
\RightLabel{\scriptsize{$({FS_{\! i}}\mra)$}}

\UI$W \fCenter x \mRA_{\!\! i} \  (Y > Z)$
\DisplayProof
\\

\\
\AX$( x \mBRA_{\!\! i} \  Y)> (x \mBAND_{\! i} \  Z)  \fCenter W $
\LeftLabel{\scriptsize{$({FS_{\! i}}\mband)$}}

\UI$ x \mBAND_{\! i} \  (Y > Z)  \fCenter W$
\DisplayProof &
\AX$W\fCenter ( x \mBAND_{\! i} \  Y)> (x \mBRA_{\!\! i} \  Z)  $
\RightLabel{\scriptsize{$({FS_{\! i}}\mbra)$}}
\UI$W \fCenter x \mBRA_{\!\! i} \  (Y > Z)$
\DisplayProof
\\

\\
\AX$( x \mAND_{\! i} \  Y) \ ; (x \mAND_{\! i} \  Z)  \fCenter W $
\LeftLabel{\scriptsize{$({mon_i}\mand)$}}

\UI$ x \mAND_{\! i} \  (Y \  ; Z)  \fCenter W$
\DisplayProof &
\AX$W\fCenter ( x \mRA_{\!\! i} \  Y)  \  ; (x \mRA_{\!\! i} \  Z)  $
\RightLabel{\scriptsize{$({mon_i}\mra)$}}
\UI$W \fCenter x \mRA_{\!\! i} \  (Y \  ; Z)$
\DisplayProof
\\

\\
\AX$ ( x \mBAND_{\! i} \  Y) \  ; (x \mBAND_{\! i} \  Z)  \fCenter W $
\LeftLabel{\scriptsize{$({mon_i}\mband)$}}

\UI$ x \mBAND_{\! i} \  (Y \  ; Z)  \fCenter W$
\DisplayProof &
\AX$ W \fCenter   ( x \mBRA_{\!\! i} \  Y) \  ; (x \mBRA_{\!\! i} \  Z)  $
\RightLabel{\scriptsize{$({mon_i}\mbra)$}}
\UI$W \fCenter  x \mBRA_{\!\! i} \  (Y \  ; Z)  $
\DisplayProof
\end{tabular}
\end{center}


\noindent Next, we introduce the rules translating the interaction axioms between dynamic and epistemic modalities. In what follows we omit the subscripts, since the reading is unambiguous. 

\begin{center}
\begin{tabular}{cc}

\AX$ (\agA \mBAND F) \mBAND (\agA \mBAND X) \fCenter Y$
\LeftLabel{\scriptsize{swap-out$_L$}}
\UI$ \agA \mBAND(F \mBAND X ) \fCenter Y$
\DisplayProof
&

\AX$X\fCenter (\agA \mBAND F) \mBRA (\agA \mBRA Y)$
\RightLabel{\scriptsize{swap-out$_R$}}
\UI$X \fCenter \agA \mBRA (F \mBRA Y)$
\DisplayProof

\\

\\
\AX$ \agA \mBAND(F \mBAND X ) \fCenter Y$
\LeftLabel{\scriptsize{swap-in$_L$}}
\UI$  (\agA \mBAND F) \mBAND (\agA \mBAND ((F \mAND \textrm{I} )\ ;   X)) \fCenter Y$

\DisplayProof

&

\AX$X \fCenter \agA \mBRA (F \mBRA Y)$
\RightLabel{\scriptsize{swap-in$_R$}}
\UI$X\fCenter (\agA \mBAND F) \mBRA (\agA \mBRA ((F \mAND \textrm{I} ) > Y))$

\DisplayProof
\\

\end{tabular}
\end{center}

\noindent The structure  $(\agA \mBAND F)$ in the \emph{swap}-rules above has absorbed the labels $\alpha\aga\beta$ in the corresponding \emph{swap}-rules of D'.EAK. Moreover,  new {\em swap-out} rules  are unary, whereas the corresponding ones in D'.EAK are of a non-fixed arity.

The following \emph{atom} axiom translates the \emph{atom} axiom of D'.EAK:
\[
F_1 \circ (F_2 \circ \cdots (F_n \circ p) \cdots)\vdash G_1 {\rhd} (G_2 {\rhd} \cdots (G_m {\rhd} p)\cdots)
\]\label{page:atom}
where  $F_1,\ldots, F_n, G_1,\ldots, G_m\in \mathsf{FNC}$, $\circ \in \{\mAND_{\! 0} \ , \mBAND_{\! 0} \ \}$, ${\rhd} \in \{\mRA_{\!\! 0} \ , \mBRA_{\!\! 0} \ \}$ and $n, m\in \mathbb{N}$. In what follows, we sometimes indicate the {\em atom} axiom with the shorter symbol $\Phi p\vdash \Psi p$.
Notice the following difference between  the present {\em atom} axiom and the one of D'.EAK (cf.\ \cite{GAV}): the structural variables $F$s and $G$s (which are typically instantiated as operational variables $\alpha$ and $\beta$ of type $\mathsf{Fnc}$)  translate what in the {\em atom} axiom of D'.EAK were indexes for logical connectives, whereas in the Dynamic Calculus, the operational variables contained in any instantiation of the $F$s and $G$s are first-class citizens, on the same ground as the operational variable $p$ of type $\mathsf{Fm}$. Hence we need to stipulate whether the introduction of each of these variables is parametric or not. As is customary in the literature on display calculi (cf.\ \cite[Definition 4.1]{Belnap}), we stipulate that the only principal variables in {\em atom} are the $p$s, and all the other variable occurrences  are parametric. 

\noindent Finally, the following {\em balance} rule:\label{page:balance}
\[
\AX$ X\fCenter Y$
\UI$ F \mAND_{\! 0} \  X\fCenter F \mRA_{\!\! 0} \  Y$
\DisplayProof
\]

\noindent is sound only for  $F\in \mathsf{FNC}$, and cannot be extended to an arbitrary actions.\footnote{To see this, notice that  this rule instantiates to \[
\AX$ p\fCenter p$
\UI$ F \mAND_{\! 0} \  p\fCenter F \mRA_{\!\! 0} \  p$
\DisplayProof
\] If the rule \emph{balance} is to be sound, the validity of the premise implies the validity of the conclusion, which is the translation of the sequent  $\langle F\rangle p\vdash [F]p$, which is equivalent to the axiom $\langle F\rangle p\rightarrow [F]p$. It is a well known fact from Sahlqvist theory that the latter axiom corresponds to the condition that the binary relation associated with $\langle F\rangle$ and $[F]$ is the graph of a partial function.} In this rule, every variable occurrence is parametric, and each occurrence of $F$ is only congruent to itself. 

\paragraph{Justifying the two types of actions.}\label{page:two types}
As discussed in the introduction, one of the initial aims of the present paper was introducing a formal framework expressive enough so as to capture at the object-level the information encoded in the meta-linguistic label $\alpha\aga\beta$.
From the order-theoretic analysis at the beginning of the present section, it emerged that the additional expressivity encoded in the connective $\mband_{\! 3} $ and its interpretation \eqref{def: mband3} requires a semantic environment which cannot  be restricted to functional actions. The introduction of the general type $\mathsf{Act}$ serves this purpose.  However, the fact that the rule \emph{balance} is only sound for functional actions is the reason why both types $\mathsf{Fnc}$ and $\mathsf{Act}$ are needed in order for the Dynamic Calculus to satisfy conditions C'$_6$ and C'$_7$ of Section \ref{sec:quasi}. Indeed, the distinct type $\mathsf{Fnc}$ allows for the  rule \emph{balance} to be formulated so that all parametric variables occur unrestricted within each type.



\section{Soundness}
\label{sec:soundness}
    
\newcommand{\lsem}{\mathopen{[\![}}
\newcommand{\rsem}{\mathclose{]\!]}}
\newcommand{\sem}[1]{\lsem #1 \rsem}
\newcommand{\RESbetaBox}{\,\rotatebox[origin=c]{90}{$[\rotatebox[origin=c]{-90}{$\beta$}]$}\,}
\newcommand{\RESbetaDia}{\,\rotatebox[origin=c]{90}{$\langle\rotatebox[origin=c]{-90}{$\beta$}\rangle$}\,}


In the present section, we discuss the soundness of the rules of the Dynamic Calculus and  prove that those which do not involve virtual adjoints (cf.\ Section \ref{virtual adjoints})
are sound with respect to the final coalgebra semantics.
In \cite[Section 5]{GAV}, basic facts about the final coalgebra have been collected and it is explained in detail how the rules of display calculi are to be interpreted in the final coalgebra.
Here we will briefly recall some basics, and refer the reader to \cite[Section 5]{GAV} for a complete discussion.


Structures will be translated into operational terms of the appropriate type, and operational terms  will be interpreted
according to their type. Specifically,  each atomic proposition $p$ is assigned to a subset $\sem{p}$ of the final coalgebra $Z$, each agent  $\aga$  a binary relation $\aga_Z = \sem{\aga}$ on $Z$
representing as usual $\aga$'s uncertainty about the world, and
each functional actions $\alpha$ is assigned  a functional (i.e.\ deterministic) relation $\alpha_Z = \sem{\alpha}\subseteq Z\times Z$ subject to the restriction defining the specific feature of epistemic actions, namely, that for all $z, z'\in Z$, if $z \alpha_Z z'$, then $z\in \sem{p}$ iff $z'\in \sem{p}$ for every atomic proposition $p$.

Further, each agent $\aga$ is associated with an auxiliary binary relation $\aga_\mathsf{Fnc}$ on the domain of interpretation of $\mathsf{Fnc}$, which is the collection of graphs of partial functions having subsets of $Z$ as domain and range.  For each agent $\aga$, the relation $\aga_\mathsf{Fnc}$ represents $\aga$'s uncertainty about which action takes place).

In order to translate structures as operational terms, structural connectives need to be translated as logical connectives. To this effect,
non-modal structural connectives are associated with pairs of logical connectives, and any given occurrence of a structural connective is translated as one or the other, according to
its (antecedent or succedent) position.
The following table illustrates how to translate each propositional structural connective of type $\mathsf{FM}$, in the upper row, into one or the other of the logical connectives corresponding to it on the lower row: the one on the left-hand (resp.\ right-hand) side, if the structural connective occurs in precedent (resp.\ succedent) position.

%
\begin{center}
\begin{tabular}{|r|c|c|c|c|c|c|c|c|c}
 \hline
 \scriptsize{Structural symbols} & \mc{2}{c|}{$<$}   & \mc{2}{c|}{$>$} & \mc{2}{c|}{$;$} & \mc{2}{c|}{I}    \\
 \hline
 \scriptsize{Operational symbols} & $\pdla$ & $\leftarrow$ & $\pdra$ & $\rightarrow$  & $\wedge$ & $\vee$  & $\top$ & $\bot$    \\
\hline
\end{tabular}
\end{center}
Recall that, in the Boolean setting treated here, the  connectives $\pdla$ and $\pdra$ are interpreted as $A\pdla B : = A\wedge \neg B$ and $A\pdra B : = \neg A\wedge B$.

\vspace{5px}
The soundness of structural and operational rules which only involve active components of type $\mathsf{FM}$
has been discussed in \cite{GAV} and is here therefore omitted.

As to the heterogeneous connectives, their translation into the corresponding operational connectives is indicated in the table below,
to be understood similarly to the one above,
where the index $i$ ranges over $\{0,1,2,3\}$ for the triangles and over $\{0,1,2\}$ for the arrows.
%
\begin{center}
\begin{tabular}{| c|c|c | c | c|c|c|c|c|}
\hline
\footnotesize{Structural symbols}
&\mc{2}{c|}{$\mAND_{\! i}$}&\mc{2}{c|}{ $\mBAND_{\! i}$}
&\mc{2}{c|}{$\mRA_{\!\! i}$}&\mc{2}{c|}{ $\mBRA_{\!\! i} $}\\
\hline
\footnotesize{Operational symbols} &$\mand_i$&\phantom{$\mand_i$}
& $\mband_{\! i}$&\phantom{$\mand_{\! i}$}
&\phantom{$\mra_{\!\! i}$}&$\mra_{\!\! i}$
&\phantom{$\mra_{\!\! i}$}& $\mbra_{\!\! i} $\\
\hline
\end{tabular}
\end{center}
%


The interpretation of the heterogeneous connectives involving formulas and agents corresponds to that of the well-known forward and backward modalities discussed in Section \ref{ssec:D'.EAK prelim} (below on the right-hand side we recall the notation of D'.EAK):
\begin{align*}
\sem{\aga \mand_2 A}&=\{z\in Z\mid \exists z' \,.\, z\,\aga_Z\, z' \ \& \ z'\in\sem{A}\}
& \langle\aga\rangle A\\
\sem{\aga \mband_{\! 2} \ A}&=\{z\in Z\mid \exists z \,.\, z'\aga_Z\, z \ \&\  z'\in\sem{A}\}
& \RESagaDia A\\
\sem{\aga \mra_{\!\! 2} \ A}&=\{z\in Z\mid \forall z' \,.\, z\,\aga_Z\, z' \Rightarrow z'\in\sem{A}\}
& [\aga] A \\
\sem{\aga \mbra_{\!\! 2} \ A}&=\{z\in Z\mid \forall z \,.\, z'\aga_Z\, z \Rightarrow z'\in\sem{A}\}
& \RESagaBox A
\end{align*}
The connectives $\mand_0,\mra_{\!\! 0},\mband_{\! 0},\mbra_{\!\! 0}$, involving formulas and functional actions, are interpreted in the same way, replacing the relation $\aga_Z$ with the  deterministic relations $\alpha_Z$. 
From the definitions above,  it immediately follows for any $\alpha\in \mathsf{Fnc}$, we have
 $\sem{\alpha\mand_0 \top} = dom(\alpha_Z)$, where the set
 $$dom(\alpha_Z)  :=\{ z\in Z \mid \exists z'  (z'\in Z \; \& \; z \alpha_Z z') \}$$
 is the \textit{domain} of  $\alpha_Z $.
 
It can also be readily verified that, after having fixed the relations interpreting all $\alpha$s and $\aga$s, the translation of Section \ref{translation DEAK to DC} preserves the semantic interpretation, that is, $\sem{A} = \sem{A'}$ for any D'.EAK formula $A$, where $A'$ denotes the translation of $A$ in the language of the Dynamic Calculus.

The auxiliary relations $\aga_\mathsf{Fnc} = \sem{\aga}^{\mathsf{Fnc}}$ are used to define the interpretations of $\mand_{ 3}$- and $\mband_{\! 3}$-operational terms. Following \ref{def: mband3}, we let
 \begin{eqnarray*}
\sem{\aga \ \mband_{\! 3} \ \alpha} &= &\bigcup\{G \mid \alpha_Z \,\aga_\mathsf{Fnc}\,G \},\\
\sem{\aga \mand_3 \alpha } &=& \bigcup\{G \mid G \,\aga_\mathsf{Fnc}\,\alpha_Z \}.
\end{eqnarray*}
The connectives $\mand_{ 1},\mra_{\!\! 1},\mband_{\! 1},\mbra_{\!\! 1}$, involving  $\mathsf{Act}$-type operational terms $\gamma$, are interpreted in the same way as the $0$- and $2$- indexed connectives, replacing the relation $\aga_Z$ with the interpretation of the appropriate operational term $\gamma$ of type $\mathsf{Act}$. \\
%


The soundness of  all operational rules for  heterogeneous connectives
immediately follows from the fact that their semantic
counterparts as defined above are monotone or antitone
in each coordinate.

The soundness of the rule \emph{balance}  immediately follows from the fact that the functional actions are interpreted as deterministic relations (for more details cf.\ \cite[Section 6.2]{GAV}).

The soundness of the cut-rules follows from the transitivity of the inclusion relation in the domain of interpretation of each type.

The soundness of  the \textit{Atom} axioms is argued similarly to that of the \textit{Atom} axioms of the system D'.EAK, crucially using the fact that epistemic actions do not change the factual states of affairs (cf.\ \cite[Section 6.2]{GAV}).

The display rules $(\mand_i , \mbra_{\!\!i})$ and $(\mband_i , \mra_{\!\!i})$ for $0\leq i\leq 2,$ and
$(\mand_1 , \bla_{\!\!1})$ and $(\mband_{\! 1} , \la_{\!\!1})$
are sound as the semantics of the triangle and arrow connectives form adjoint pairs.

On the other hand, in the display rules $(\mand_3, \msbra_{\!\!3})$,  $(\mband_{\! 3}, \msra_{\!\!3})$,   $(\mand_i , \msbla_{\!\!i})$ and $(\mband_i , \msla_{\!\!i})$ for $i = 0,2,3$, the arrow-connectives are what we  call \emph{virtual adjoints} (cf.\ Section \ref{DC}), that is, they do not have a semantic interpretation.
In the next section, we will account for the fact that their presence in the calculus is safe.

Soundness of  necessitation, conjugation, Fischer Servi, and monotonicity rules is straightforward and proved as in \cite{GAV}.
In the remainder of the section, we discuss the soundness of
the new rules swap-in and swap-out recalled below.

\begin{fact}
\label{fact : mband 1 join preserving first coordinate}
The following defining clause for the interpretation of $\mband_{\! 1}$-operational terms
$$\sem{\gamma \ \mband_{\! 1} \ A}=\{z\in Z\mid \exists z \,.\, z' \gamma_Z \, z \ \&\  z'\in\sem{A}\}
$$
immediately implies that the semantic interpretation of $\mband_{\! 1}$ is completely $\bigcup$-preserving in its first coordinate.
\end{fact}
\begin{proof}
If $\gamma_Z =\bigcup_{i\in I} \beta_i $, then clearly $z'\gamma_Z z$ iff $z' \beta_i z'$ for some $i\in I$.
\end{proof}
As to the soundness of \emph{swap-out}$_L$,
assume that the structures $\agA$, $F$, $X$ and $Y$ have been given the following interpretations, according to their type, as discussed above:
 $\mathbf{a}_Z \subseteq Z\times Z$,  $\mathbf{a}_{\mathsf{Fnc}}  $ is a binary relation on graphs of partial functions on $Z$,  $F_Z$ is a functional relation on $Z$, and  $X_Z, Y_Z \subseteq Z$.
Let
$$ \mathbf{a}_Z \ \mband_{\! 3} \ F_Z \ := \ \bigcup \{ \beta \mid F_Z \mathbf{a}_{\mathsf{Fnc}}  \beta \}.$$

\newcommand{\RESmbandthreeAzFzDia}{\,
$\raisebox{0.8ex}{$
\widehat{\;\;\;\;\;\;\;\;\;\;\;\;}$}
\!\!\!\!\!\!\!\!\!\!\!\!\!\!\!\!\!\!\!\!\!\!\!
\raisebox{-1.3ex}{\rotatebox[origin=c]{-180}{$
\widehat{\;\;\;\;\;\;\;\;\;\;\;\;}$}}
\!\!\!\!\!\!\!\!\!\!\!\!\!\!\!\!\!\!\!\!\!\!
$\mathbf{a}_Z  \blacktriangle_3  F_Z}

\newcommand{\RESmbandthreeAzFzDiatwo}{\,
$\raisebox{0.8ex}{$
\widehat{\;\;\;\;\;\;\;\;\;\;\;\;}$}
\!\!\!\!\!\!\!\!\!\!\!\!\!\!\!\!\!\!\!\!\!\!\!\!
\raisebox{-1.3ex}{\rotatebox[origin=c]{-180}{$
\widehat{\;\;\;\;\;\;\;\;\;\;\;\;}$}}
\!\!\!\!\!\!\!\!\!\!\!\!\!\!\!\!\!\!\!\!\!\!
$\mathbf{a}_Z  \blacktriangle_3  F_Z}

\newcommand{\RESagAzDia}{\,\rotatebox[origin=c]{90}{$\langle \, \rotatebox[origin=c]{-90}{$\mathbf{a}_Z$}\,\rangle$}\,}

\newcommand{\RESFzDia}{\,\rotatebox[origin=c]{90}{$\langle \, \rotatebox[origin=c]{-90}{$F_Z$}\,\rangle$}\,}

\newcommand{\RESGDia}{\,\rotatebox[origin=c]{90}{$\langle \, \rotatebox[origin=c]{-90}{$G$}\,\rangle$}\,}




Assume that the premise of \emph{swap-out}$_L$ is satisfied. That is:
\begin{center}
$\RESmbandthreeAzFzDia \; \RESagAzDia X_Z \subseteq Y_Z ,$
\end{center}
where the symbols $\RESmbandthreeAzFzDiatwo$ and $\RESagAzDia$ denote the semantic diamond operations associated with the converses of the relations
$\mathbf{a}_Z \mband_3 F_Z$ and $\mathbf{a}_Z$ respectively.
Then, the following chain of equivalences holds:



\begin{center}
\begin{tabular}{rclr}
$\RESmbandthreeAzFzDia \; \RESagAzDia X_Z \subseteq Y_Z $
& iff
& $\bigcup\{\RESGDia \RESagAzDia X_Z \mid F_Z \mathbf{a}_\mathsf{Fnc}\, G \} \subseteq Y_Z $
\hfill (fact \ref{fact : mband 1 join preserving first coordinate})
\\
& iff
& $\RESGDia \RESagAzDia X_Z \subseteq Y_Z$
for every $G$ s.t.\ $F_Z \mathbf{a}_\mathsf{Fnc}\,G$
&\\
& iff
& $X_Z \subseteq [ \mathbf{a}_Z ] [ G ]  Y_Z$
for every $G$ s.t.\ $F_Z \mathbf{a}_\mathsf{Fnc}\,G$
&\\
& iff
& $X_Z \subseteq \bigcap \{  [ \mathbf{a}_Z ] [ G ]  Y_Z \mid
F_Z \mathbf{a}_\mathsf{Fnc}\,G \}$
&\\
& hence
&$X_Z \subseteq
(dom(F_Z))^c \; \cup \;
 \bigcap \{  [ \mathbf{a}_Z ] [ G ]  Y_Z \mid
F_Z \mathbf{a}_\mathsf{Fnc}\,G \}$.
\end{tabular}
\end{center}
Consider the new variables $p,q$, $\aga$, $\alpha$, and $\beta_i$
for each $G_i$ such that $ F_Z \mathbf{a}_{\mathsf{Fnc}} G_i$.
Let us stipulate  that
$\sem{p} := X_Z$, $\sem{q}:= Y_Z$,
$\sem{\aga}:= \mathbf{a}_Z$, $\sem{\alpha}:= F_Z$, and
$\sem{\beta_i} := G_i$.
Hence $\sem{Pre(\alpha)} = \sem{\alpha \mand_0 \top}  = dom(F_Z)$. Therefore, the computation above can continue as follows:
\begin{center}
\begin{tabular}{rclr}
\phantom{$\RESmbandthreeDia \; \RESagaDia X_Z \subseteq Y_Z $}
&
& $X_Z \subseteq
(dom(F_Z))^c \; \cup \;
 \bigcap \{  [ \mathbf{a}_Z ] [ G ]  Y_Z \mid
F_Z \mathbf{a}_\mathsf{Fnc}\,G \}$.
&\\
& iff
& $\sem{p} \subseteq \sem{Pre(\alpha) \rightarrow
\bigwedge\{[\aga][\beta] q \mid \alpha\aga\,\beta\}}$
&\\
& iff
& $ \sem{p} \subseteq \sem{[\alpha][\aga] q }$
&\\
& iff
& $X_Z \subseteq [F_Z][\mathbf{a}_Z] Y_Z $
&\\
& iff
& $\RESagAzDia \RESFzDia X_Z \subseteq Y_Z$
\end{tabular}
\end{center}
which completes the proof of the soundness of \emph{swap-out}$_L$.
The proof of the soundness of the remaining swap-rules is similar.



\section{Completeness and cut elimination}
    In \ref{ssec: completeness}, we discuss the  completeness of the Dynamic Calculus w.r.t.\ the final coalgebra semantics. We show that the translation (cf.\ Section \ref{translation DEAK to DC}) of each of the EAK axioms is derivable in the Dynamic Calculus. Our proof is indirect, and relies on the fact that EAK is complete   w.r.t.\ the final coalgebra semantics, and that the translation preserves the semantic interpretation on the final coalgebra (as discussed in Section \ref{sec:soundness}).  In \ref{ssec:belnap-cut}, we show that the Dynamic Calculus is quasi-properly displayable (cf.\ Section \ref{ssec:quasi-def}). By Theorem \ref{thm:meta multi}, this is enough to establish that the calculus enjoys  cut elimination and the subformula property.

\subsection{Derivable rules and  completeness}
\label{ssec: completeness}

In what follows, $\aga$ and $\alpha$ are atomic variables (and also the generic operational terms) of type $\mathsf{Ag}$ and $\mathsf{Fnc}$ respectively,  and $A, B$ are generic operational terms of type $\mathsf{Fm}$.
Since the reading is unambiguous, in the remainder of the present paper the indexes of the heterogeneous connectives are dropped.


\noindent \label{page : rule reduce} Under the stipulations above, the translations of the rules {\em reduce} from D'.EAK (cf.\ Section \ref{ssec:D'.EAK prelim}) can be derived in the Dynamic Calculus as follows.

{\scriptsize{
\[
\AX$\alpha \mAND \textrm{I}\ ; \alpha \mAND A \fCenter X$
\LeftLabel{Dis$_0 \!\mand$}
\UI$\alpha \mAND (\textrm{I}\ ; A)\fCenter X $
\UI$ \textrm{I}\ ; A \fCenter \alpha \mBRA X$
\UI$ \textrm{I}\ \fCenter \alpha \mBRA X < A$
\UI$A \fCenter \alpha \mBRA X $
\UI$ \alpha \mAND A \fCenter X $
\DisplayProof
\]
}
}

\noindent Also  the translations of the {\em comp} rules are derivable in the Dynamic Calculus as follows.

{\scriptsize{
\[
\AX$\alpha \mAND (  \alpha \mBAND X) \fCenter Y$
\UI$ \alpha \mBAND X \fCenter \alpha \mBRA Y$
\UI$ \textrm{I} \fCenter \alpha \mBAND X >  \alpha \mBRA Y  $
\UI$   \alpha \mBAND X\ ; \textrm{I}  \fCenter \alpha \mBRA Y$
\UI$ \alpha \mAND (  \alpha \mBAND X\ ; \textrm{I}) \fCenter  Y$
\LeftLabel{conj$_0 \!\mand$}
\UI$X\ ; \alpha \mAND \textrm{I} \fCenter Y $
\UI$\alpha \mAND \textrm{I}\ ; X \fCenter Y $
\DisplayProof
\]
}
}



\noindent Let us derive the axiom (\ref{DC:axiom1}):

{\scriptsize{
\begin{prooftree}

\AxiomC{$\aga \vdash \aga$}

\AxiomC{$\aga \vdash \aga$}
\AxiomC{$\alpha \vdash \alpha$}
\BinaryInfC{$\aga \mBAND \alpha \vdash \aga \mband \alpha$}

\AxiomC{$A \vdash A$}
\BinaryInfC{$(\aga \mband \alpha) \mra A \vdash (\aga \mBAND \alpha) \mRA A$}

\BinaryInfC{$ \aga \mra ((\aga \mband \alpha) \mra A )\vdash  \aga \mRA ((\aga \mBAND \alpha)  \mRA A)$}
\UnaryInfC{$  \aga \mBAND (\aga \mra ((\aga \mband \alpha) \mra A ))\vdash  ((\aga \mBAND \alpha)  \mRA A)$}
\UnaryInfC{$ ((\aga \mBAND \alpha)  \mBAND ( \aga \mBAND (\aga \mra ((\aga \mband \alpha) \mra A )))\vdash A$}
\LeftLabel{swap-out$_L$}
\UnaryInfC{$\aga \mBAND (\alpha \mBAND (\aga \mra ((\aga \mband \alpha) \mra A ))) \vdash A$}
\UnaryInfC{$ \alpha \mBAND (\aga \mra ((\aga \mband \alpha) \mra A )) \vdash \aga \mRA A$}
\UnaryInfC{$ \alpha \mBAND (\aga \mra ((\aga \mband \alpha) \mra A )) \vdash \aga \mra A$}
\UnaryInfC{$\aga \mra ((\aga \mband \alpha) \mra A ) \vdash  \alpha \mRA (\aga \mra A)$}
\UnaryInfC{$\aga \mra ((\aga \mband \alpha) \mra A ) \vdash  \alpha \mra (\aga \mra A)$}

\end{prooftree}
}
}

\noindent Let us derive the axiom (\ref{DC:axiom2}):
{\scriptsize{
\[
\AxiomC{$\aga \vdash \aga$}
\AxiomC{$A \vdash A$}
\AxiomC{$\aga \vdash \aga$}
\AxiomC{$\alpha \vdash \alpha$}
\BinaryInfC{$\aga \mBAND \alpha \vdash \aga \mband \alpha$}
\BinaryInfC{$(\aga \mBAND \alpha) \mAND A \vdash   (\aga \mband \alpha) \mand A$}
\BinaryInfC{$\aga \mAND ((\aga \mBAND \alpha) \mAND A )\vdash   \aga \mand ((\aga \mband \alpha) \mand A)$}
\UnaryInfC{$(\aga \mBAND \alpha) \mAND A \vdash  (\aga \mBRA ( \aga \mand ((\aga \mband \alpha) \mand A)))$}
\UnaryInfC{$A \vdash (\aga \mBAND \alpha) \mBRA (\aga \mBRA ( \aga \mand ((\aga \mband \alpha) \mand A)))$}
\RightLabel{swap-out$_R$}
\UnaryInfC{$ A \vdash \aga \mBRA  (\alpha \mBRA ( \aga \mand ((\aga \mband \alpha) \mand A))$}
\UnaryInfC{$\aga \mAND A \vdash  \alpha \mBRA ( \aga \mand ((\aga \mband \alpha) \mand A)$}
\UnaryInfC{$\aga \mand A \vdash  \alpha \mBRA ( \aga \mand ((\aga \mband \alpha) \mand A)$}
\UnaryInfC{$ \alpha \mAND (\aga \mand A) \vdash \aga \mand ((\aga \mband \alpha) \mand A)$}
\UnaryInfC{$ \alpha \mand (\aga \mand A) \vdash \aga \mand ((\aga \mband \alpha) \mand A)$}
\DisplayProof
\]
}
}

\noindent A slight difference between the setting of \cite{Dyckhoff} and the present setting is that in that paper only the dynamic boxes are allowed in the object language, even if their propositional base is taken as non classical. In the present setting however, both the dynamic  boxes and diamonds are taken as primitive connectives. When moving to a propositional base which is weaker than the Boolean one,  also the diamond/box  interaction  axioms such as the following one become primitive: $$  [\alpha] \lc \aga \rc A \leftrightarrow 1_\alpha \rightarrow \bigvee\{\lc\aga\rc \lc \beta \rc A \mid \alpha\aga\beta \}.$$

The axiom above translates as: $$  \alpha \mra  (\aga \mand A)  \leftrightarrow  \alpha \mand \top  \rightarrow \aga \mand ((\aga \mband \alpha) \mand A).$$

\begin{tabular}{c}
\footnotesize{
\AX$\alpha \fCenter \alpha$
\AX$\aga \fCenter \aga$
\AX$\aga \fCenter \aga$
\AX$\alpha \fCenter \alpha$
\BI$\aga \mBAND \alpha \fCenter \aga \mband \alpha$
\AX$ A \fCenter A$
\BI$(\aga \mBAND \alpha) \mAND A \fCenter (\aga \mband \alpha) \mand A$
\BI$\aga \mAND ((\aga \mBAND \alpha) \mAND A) \fCenter \aga \mand ((\aga \mband \alpha) \mand A) $
\UI$(\aga \mBAND \alpha) \mAND A \fCenter \aga \mBRA( \aga \mand ((\aga \mband \alpha) \mand A)) $
\UI$ A \fCenter (\aga \mBAND \alpha) \mBRA( \aga \mBRA( \aga \mand ((\aga \mband \alpha) \mand A)) )$
\RightLabel{\scriptsize{swap-out}$_R$}
\UI$A \fCenter \aga  \mBRA ( \alpha \mBRA( \aga \mand ((\aga \mband \alpha) \mand A)) )$
\UI$\aga \mAND A \fCenter  \alpha \mBRA( \aga \mand ((\aga \mband \alpha) \mand A)) $
\UI$\aga \mand A \fCenter  \alpha \mBRA( \aga \mand ((\aga \mband \alpha) \mand A)) $
\BI$\alpha \mra  (\aga \mand A) \fCenter \alpha \mRA (\alpha \mBRA( \aga \mand ((\aga \mband \alpha) \mand A)))$
\UI$\alpha \mBAND ( \alpha \mra  (\aga \mand A)) \fCenter \alpha \mBRA( \aga \mand ((\aga \mband \alpha) \mand A))$
\UI$\textrm{I} \fCenter (\alpha \mBAND ( \alpha \mra  (\aga \mand A))) > (\alpha \mBRA( \aga \mand ((\aga \mband \alpha) \mand A)))$
\UI$ \alpha \mBAND ( \alpha \mra  (\aga \mand A))~; \textrm{I} \fCenter \alpha \mBRA( \aga \mand ((\aga \mband \alpha) \mand A))$
\UI$\alpha \mAND(\alpha \mBAND ( \alpha \mra  (\aga \mand A) )~; \textrm{I}) \fCenter  \aga \mand ((\aga \mband \alpha) \mand A)$
\LeftLabel{\scriptsize{conj}$_{0}\! \mand$}
\UI$ ( \alpha \mra  (\aga \mand A) ) ~; ( \alpha \mAND \textrm{I} ) \fCenter  \aga \mand ((\aga \mband \alpha) \mand A)$
\UI$  \alpha \mAND \textrm{I}  \fCenter ( \alpha \mra  (\aga \mand A) ) >( \aga \mand ((\aga \mband \alpha) \mand A))$
\UI$   \textrm{I}  \fCenter \alpha \mBRA ( \alpha \mra  (\aga \mand A) ) >( \aga \mand ((\aga \mband \alpha) \mand A))$
\UI$   \top  \fCenter \alpha \mBRA ( \alpha \mra  (\aga \mand A) ) >( \aga \mand ((\aga \mband \alpha) \mand A))$
\UI$  \alpha \mAND \top  \fCenter ( \alpha \mra  (\aga \mand A) ) >( \aga \mand ((\aga \mband \alpha) \mand A))$
\UI$  \alpha \mand \top  \fCenter ( \alpha \mra  (\aga \mand A) ) >( \aga \mand ((\aga \mband \alpha) \mand A))$
\UI$  ( \alpha \mra  (\aga \mand A) )~;  \alpha \mand \top  \fCenter  \aga \mand ((\aga \mband \alpha) \mand A)$
\UI$  \alpha \mand \top ~ ; ( \alpha \mra  (\aga \mand A) ) \fCenter  \aga \mand ((\aga \mband \alpha) \mand A)$
\UI$  \alpha \mra  (\aga \mand A)  \fCenter  \alpha \mand \top  >  \aga \mand ((\aga \mband \alpha) \mand A)$
\UI$  \alpha \mra  (\aga \mand A)  \fCenter  \alpha \mand \top  \rightarrow \aga \mand ((\aga \mband \alpha) \mand A)$
\DisplayProof
}
\end{tabular}

\vspace{5px}

For the other direction,
recall that the counterpart of the rule \textit{reduce} is derivable in the Dynamic Calculus \pageref{page : rule reduce}:
\begin{center}
\begin{tabular}{c}
\footnotesize{
\AX$\alpha \fCenter \alpha$
\AX$ \textrm{I} \fCenter \top$
\BI$\alpha \mAND \textrm{I} \fCenter \alpha \mand  \top$

\AX$\aga \fCenter \aga$
\AX$A \fCenter A$
\BI$\aga \mAND A \fCenter \aga \mand A$
\LeftLabel{\scriptsize{balance}}
\UI$\alpha \mAND (\aga \mAND A) \fCenter \alpha \mRA (\aga \mand A) $
\UI$ \aga \mAND A \fCenter \alpha \mBRA (\alpha \mRA (\aga \mand A)) $
\UI$  A \fCenter \aga \mBRA( \alpha \mBRA (\alpha \mRA (\aga \mand A)) )$
\RightLabel{\scriptsize{swap-in$_R$}}
\UI$A \fCenter (\aga \mBAND \alpha) \mBRA (\aga \mBRA (  (\alpha \mAND \textrm{I} ) >  (\alpha \mRA (\aga \mand A))  ) )$

\UI$(\aga \mBAND \alpha) \mAND A  \fCenter \aga \mBRA (  (\alpha \mAND \textrm{I} ) >  (\alpha \mRA (\aga \mand A))  ) $

\UI$\aga \mBAND \alpha \fCenter ( \aga \mBRA (  (\alpha \mAND \textrm{I} ) >  (\alpha \mRA (\aga \mand A))  )) \mBLA  A$
\UI$\aga \mband \alpha \fCenter  ( \aga \mBRA (  (\alpha \mAND \textrm{I} ) >  (\alpha \mRA (\aga \mand A))  )) \mBLA  A$

\UI$(\aga \mband \alpha) \mAND A \fCenter  \aga \mBRA (  (\alpha \mAND \textrm{I} ) >  (\alpha \mRA (\aga \mand A))  ) $

\UI$(\aga \mband \alpha) \mand A \fCenter   \aga \mBRA (  (\alpha \mAND \textrm{I} ) >  (\alpha \mRA (\aga \mand A))  ) $

\UI$  \aga \mAND ((\aga \mband \alpha) \mand A) \fCenter   (\alpha \mAND \textrm{I} ) >  (\alpha \mRA (\aga \mand A))  $

\RightLabel{$reduce_R$}
\UI$ \aga \mAND ((\aga \mband \alpha) \mand A) \fCenter  \alpha \mRA (\aga \mand A)  $

\UI$  \aga \mand ((\aga \mband \alpha) \mand A) \fCenter     \alpha \mRA (\aga \mand A)  $

\BI$ \alpha \mand \top  \rightarrow \aga \mand ((\aga \mband \alpha) \mand A) \fCenter   \alpha \mAND \textrm{I} >  \alpha \mRA (\aga \mand A) $
\RightLabel{$reduce_R$}
\UI$ \alpha \mand \top  \rightarrow \aga \mand ((\aga \mband \alpha) \mand A) \fCenter \alpha \mRA (\aga \mand A )$
\UI$ \alpha \mand \top  \rightarrow \aga \mand ((\aga \mband \alpha) \mand A) \fCenter \alpha \mra (\aga \mand A )$

\DisplayProof
}
\end{tabular}
\end{center}

\vspace{5px}
\noindent The derivations (of the translations) of the remaining axioms 
have been relegated to the appendix.  

\subsection{Belnap-style cut-elimination, and subformula property}\label{ssec:belnap-cut}
In the present subsection, we prove that the Dynamic Calculus for EAK is a quasi-properly displayable calculus (cf.\ Section \ref{ssec:quasi-def}). By Theorem \ref{thm:meta multi}, this is enough to establish that the calculus enjoys the cut elimination and the subformula property.
Conditions C$_1$, C$_2$, C$_4$, C'$_5$, C'$_6$, C'$_7$ and C$_{10}$ are straightforwardly verified by inspecting the rules and are left to the reader.

Condition C''$_5$ can be straightforwardly argued by observing that the only axioms to which a display postulate can be applied are of the {\em atom} form: $\Phi p\vdash \Psi p$. In this case, the only applicable display postulates are those rewriting $\mAND$- or $\mBAND$-structures into $\mBRA$- and $\mRA$-structures and vice versa, which indeed preserve the {\em atom} shape.  Condition C''$_8$ is  straightforwardly verified by inspection on the axioms.
Condition  C'$_2$ can be straightforwardly verified by inspection on the rules, for instance by observing that the domains and codomains of adjoints are rigidly determined.

The following proposition shows that condition C$_9$ is met:
\begin{proposition}
\label{prop:type regular}
Any derivable sequent in the Dynamic Calculus for EAK is type-uniform.
\end{proposition}

\begin{proof}
We prove the  proposition by induction on the height of the derivation.
The base case is verified by inspection; indeed,  the  following axioms are type-uniform by definition of their constituents: 

\[\aga \vdash \aga \quad \alpha \vdash \alpha \quad \Phi p\vdash \Psi p  \quad \bot \vdash \textrm{I}  \quad \textrm{I} \vdash \top \quad \   \]

\noindent As to the inductive step, one can verify by inspection that all the rules of the Dynamic  Calculus preserve type-uniformity, 
and that the Cut rules are strongly type-uniform.
\end{proof}

As to condition C'$_3$, all parameters in all but the {\em swap-in} rules satisfy the condition of non-proliferation. In each {\em swap-in} rule, the parameters of type $\mathsf{Ag}$  and $\mathsf{Fnc}$ in the premise are congruent to {\em two} parameters in the conclusion. However, it is not difficult to see that in each derivation, each application of any cut rule
\begin{center}

\AXC{$\vdots$}
\noLine
\UI$x\fCenter a$
\AXC{$\vdots$}
\noLine
\UI$a\fCenter y$
\BI$x\fCenter y$
\DisplayProof
\end{center}
of type $\mathsf{Ag}$  or $\mathsf{Fnc}$ must be such that the structure $x$ reduces to the atomic term $a$. Indeed, because the sequent $x\vdash a$ is derivable, by Proposition \ref{prop:type regular}
it must be type uniform, that is, the structure $x$ needs to be of type $\mathsf{AG}$ if $a$ is, or of type $\mathsf{FNC}$ if $a$ is. If $x$ was not atomic, then its main structural connective would be a squiggly arrow $\mSLA$ or $\mSBLA$. Because these connectives do not have any operational counterpart, such a structure cannot have been introduced by an application of an operational rule. Hence, the only remaining possibility is that it has been introduced via a display postulate. But also this case is impossible, since in display postulates introduce these connectives only in the succedent, and $x$ is in precedent position. This finishes the verification of condition C'$_3$.

\noindent Finally, the verification steps for C'$_8$ are collected in Section \ref{ssec:cut elimination}.

\section{Conservativity}
\label{sec : safe virtual adjoints}

In the definition of the language of the Dynamic Calculus, we have adopted a rather inclusive policy. That is, the operational language includes almost all the logical symbols which could be assigned a natural interpretation purely on the basis of reasonable assumptions on the order-theoretic properties of the domains of interpretation of the various types of terms, the only exception being the connectives  $\mbla_{\! 1}$ and $\nla_{\! 1}$, which are excluded from the language although they are semantically justified.  A very useful and powerful consequence of the fact that the Dynamic Calculus enjoys cut elimination Belnap-style is that this cut elimination is then inherited by  the subcalculi  corresponding to each fragment of the operational language of the Dynamic Calculus which  verify as they stand the assumptions of Theorem \ref{thm:meta multi}. However, the question is still open about whether these  subcalculi interact with each other in unwanted ways when their proof power is concerned: for any two such fragments $\mathcal{L}_1\subseteq \mathcal{L}_2$, does the subcalculus corresponding to $\mathcal{L}_2$ conservatively extend the one corresponding to $\mathcal{L}_1$? Typically, the absence of unwanted interactions among subcalculi is deduced from having cut elimination, and soundness and completeness w.r.t.\ a given semantics. This way, in \cite{GAV} it is also shown  that the system D'.EAK conservatively extends EAK.

However, this strategy is not immediately applicable to the setting of the Dynamic Calculus, due to the structural symbols referred to as {\em virtual adjoints}, which are easily recognizable, since they are shaped like arrows with a squiggly tail: $\mSBLA$, $\mSRA$ etc. Virtual adjoints   have no semantic justification, and hence, the rules in which they specifically occur (that is, the display postulates relative to  them) cannot be justified on semantic grounds. 
 The reason for including virtual adjoints in the language of the Dynamic Calculus is for it to enjoy the relativized display property, discussed in Section \ref{ssec:DisplayLogic}, 
 which is key to guarantee the crucial condition C'$_8$, requiring the existence of a way to solve the principal stage of the cut elimination theorem (cf.\ Section \ref{principal stage}, see also Section \ref{ssec:cut elimination}).

When discussing virtual adjoints in Section \ref{DC}, we claimed that, since they are only introduced in a derivation by way of display postulates and  do not specifically intervene in any other structural rule, their presence in the calculus does not add unwanted proof power compared to D'.EAK (and hence to EAK). 
This is the sense in which the introduction of the virtual adjoints can be regarded as syntactically sound. 
The aim of the present section is to prove this claim. 

\medskip

A general and very powerful method for proving the conservativity of display calculi has been introduced  in \cite{CDGT13,CDGT13ext} for the full intuitionistic linear logic.
This method involves no less than two translations, one from the given display calculus into  an intermediate shallow inference nested sequent calculus, and another one from the intermediate calculus into a deep inference nested sequent calculus. This method is very intricate, requiring the verification of hundreds of cases which account for every possible interaction between the shallow and the deep calculus. The intricacy of this proof was such that the correctness of the results in \cite{CDGT13,CDGT13ext}
 has been established by  formalizing them  in the proof assistant Isabelle/HOL, as reported in
 \cite{DCGT14}.

However, in the present section, a much smoother proof of conservativity is given for the Dynamic Calculus for EAK, which does not rely on any nested sequent calculus. Rather, the proof below relies on very specific and uncommon features of the design of the Dynamic Calculus for EAK. In a sense, the very fact that such a smooth proof is possible   witnesses  how uncommonly well behaved EAK is.   

\begin{definition}
A sequent $x\vdash y$ is \emph{severe} if in the generation trees of either $x$ or $y$ there are occurrences of structural connectives to which no display postulates can be applied. Such occurrences will be referred to as \emph{severe}.
\end{definition}
Clearly, the definition above makes sense only in the context of calculi which, as is the case of the Dynamic Calculus, do not enjoy the full  display property (cf.\ Definition \ref{def: display prop}).
It can be easily verified that, in the specific case of the Dynamic Calculus, there are only two types of severe occurrences:   triangle-type connectives rooting a structure in succedent position, and arrow-type connectives rooting a structure in precedent position.
Examples of  severe sequents then are $(\aga \mRA \alpha)\mAND A\vdash B$ and  $A\vdash \aga\mAND B$. 

\begin{lemma}
\label{lemma: severe preservation}
Any rule in the Dynamic Calculus preserves the severity of sequents. That is, if a rule is applied to a severe sequent, the conclusion of that rule application is also  severe.
\end{lemma}
\begin{proof}
By inspection on the rules.
\end{proof}
\begin{fact}
\label{fact: shape sequents ag fnc}
Let $x\vdash y$ be a sequent  of type $\mathsf{AG}$ or $\mathsf{FNC}$, in the full language of Dynamic Calculus, which is derivable by means of a derivation   $\pi$ in which  no application of Weakening, Necessitation, Balance or Atom introduce occurrences of virtual adjoints. Then $x =  a$ for some operational term $a$ of the appropriate type.
\end{fact}
\begin{proof}
If $x$ is not an atomic structure, then the grammar of $\mathsf{AG}$ and $\mathsf{FNC}$ prescribes that   $x$ has a virtual adjoint as a main connective. However, the assumptions imply that such structures can be introduced only by way  of  applications of  display postulates, which introduce them in succedent position. Hence, given the assumptions,   there is no way in which such a connective can be introduced in precedent position.
\end{proof}

\begin{lemma}
\label{lemma:virtual free}
 Let $X\vdash Y$ be a sequent   of type $\mathsf{FM}$
which is derivable in the Dynamic Calculus by means of a  derivation   $\pi$ in which  no application of Weakening, Necessitation, Balance  or Atom introduce occurrences of virtual adjoints. Then a derivation $\pi'$ of $X\vdash Y$ exists every node of which (hence the conclusion in particular) is free of virtual adjoints.
\end{lemma}
\begin{proof}
Let $s$ be some node/sequent in $\pi$ where the given virtual adjoint has been introduced. Since virtual adjoints in the Dynamic Calculus are all virtual ``right adjoints'', and since, by assumption, they are introduced only by way  of  applications of  display postulates, the given virtual adjoint  is the main connective in the succedent of the sequent $s$. Moreover, virtual adjoints are main connectives of structures of type $\mathsf{AG}$ and $\mathsf{FNC}$. By type uniformity, this implies that the sequent $s$ is either of type $\mathsf{AG}$ or $\mathsf{FNC}$, and therefore $s$ cannot be the conclusion of $\pi$. Some  rule $R$ must exist which takes $s$ as a premise. It can be easily verified by inspection that $R$ cannot coincide with any structural rule  in the Dynamic Calculus which is neither a cut of the appropriate type nor  a display postulate, since all structural rules  different from Cut and  display postulates have premises of type $\mathsf{FM}$. We can also assume w.l.o.g.\ that $R$ is not an application of Cut.
Indeed, by Fact \ref{fact: shape sequents ag fnc},  $s$  is of the form $a\vdash x$, with $a$ being an operational term, and  $x$ being a non-atomic structure by assumption. Hence, if $R$ was a Cut-application, the inference must be of the form
\begin{center}
\AxiomC{$\vdots$}
\noLine
\UnaryInfC{$y\vdash a$}
\AxiomC{$\vdots$}
\noLine
\UnaryInfC{$a\vdash x$}
\BinaryInfC{$y\vdash x$}
\DisplayProof
\end{center}
Because Cut rules in the Dynamic Calculus are strongly type-regular, also $y\vdash a$ would be a (derivable) sequent of type $\mathsf{AG}$ or $\mathsf{FNC}$, hence Fact \ref{fact: shape sequents ag fnc} applies to $y\vdash a$. That is, $y$ must be atomic, and because $y\vdash a$ is derivable,  $y$ must coincide with $a$. Hence, the conclusion of that Cut application is again $s$. This shows that if $R$ was Cut, w.l.o.g.\ we would be able remove that application from the proof tree.
The remaining options are that $R$ coincides with an introduction rule of some heterogeneous logical connective. Recall that, by Fact \ref{fact: shape sequents ag fnc},  $s$  is of the form $a\vdash x$, where $a$ is an operational term. Then, it can be verified by inspection that no heterogeneous rule is applicable if $x$ is not an atomic structure, which is not the case of the sequent $s$, as discussed above. Finally, since the left-hand side of $s$ is atomic, no other display postulates are applicable to $s$ but the converse direction of the same display postulate which had introduced the virtual adjoint and which makes it disappear. Therefore, the refinement $\pi'$ of $\pi$ consists in removing these double and redundant applications of  display postulates.
\end{proof}

\begin{lemma}
\label{lemma: intro severe}
If \textit{inf} is an application of Balance, Atom, Necessitation or Weakening in which some  occurrence of a virtual adjoint is introduced, then the conclusion of \textit{inf} is a severe sequent.
\end{lemma}
\begin{proof}
As to Balance, Atom and Necessitation$_i$ with $i = 0, 2$, notice that each of these rules introduces a structure $x$ of type $\mathsf{FNC}$ or $\mathsf{AG}$  in precedent position. If a virtual adjoint is introduced as a substructure of $x$, then $x$ is non-atomic, and it can be immediately verified by inspecting the syntax of $\mathsf{FNC}$ and $\mathsf{AG}$ that  the main connective of $x$ is an  arrow-type connective, which would then be in precedent position. Hence, the resulting sequent is severe.

As to Weakening and Necessitation$_1$, notice that these rules  introduce structures $x$ of type $\mathsf{FM}$ and $\mathsf{ACT}$ respectively. Recall that virtual adjoints root structures of type $\mathsf{FNC}$ or $\mathsf{AG}$.  Hence, if some virtual adjoint occurs in the generation tree of  $x$,  it cannot occur at the root of $x$. Hence, the virtual adjoint must occur in the scope of some other structural connective. Notice that the heterogeneous connectives are the only ones which can take as argument a structure rooted in a virtual adjoint. We claim that either the virtual adjoint occurs in precedent position (which would be enough to conclude that the conclusion of \textit{inf} is severe), or under the scope of some structural connective to which
no display postulate can be applied.
 Assume that the virtual adjoint occurs in succedent position. If  its immediate ancestor in the generation tree of $X$ is a triangle-type connective, then these connectives are in succedent position too, and hence no display postulate can be applied to them, which makes the conclusion of  \textit{inf} severe, as required.
Similarly, if the immediate ancestor of the virtual adjoint is an arrow-type connective,  then it can be easily checked by inspection that these connectives take structures of type $\mathsf{FNC}$ or $\mathsf{AG}$ exclusively in their antitone coordinate (that is, on the flat side of the arrow). Hence, the arrow-type connective is in precedent position, and hence no display postulate can be applied to it, as required.
\end{proof}

\begin{cor}
\label{cor: interesting direction}
Let $A'\vdash B'$ be a sequent of type $\mathsf{Fm}$ in the language of the Dynamic Calculus, such that $A'$ and $B'$ are, respectively, images of some D'.EAK-formulas $A$ and $B$ under the translation of Section \ref{translation DEAK to DC}. If $A'\vdash B'$ is derivable in the Dynamic Calculus, then $A\vdash B$ is  derivable in D'.EAK.
\end{cor}
\begin{proof}
 Let $\pi$ be a derivation of $A'\vdash B'$ in the Dynamic Calculus. By assumption, $A'\vdash B'$ is not severe. Hence, no rule application in $\pi$ can introduce severe sequents, since these, by Lemma \ref{lemma: severe preservation}, would then propagate till the conclusion.  Hence in particular,  by Lemma \ref{lemma: intro severe}, in $\pi$ there cannot be any applications of Balance, Atom, Necessitation or Weakening in which some  occurrence of a virtual adjoint is introduced. Therefore, by Lemma \ref{lemma:virtual free}, a derivation $\pi'$ of $A'\vdash B'$ exists in which no virtual adjoints occur. By the results collected in Section \ref{sec:soundness}, the derivation $\pi'$ is sound w.r.t.\ the final coalgebra semantics. Hence  $A'\vdash B'$ is satisfied on the final coalgebra semantics. Since, as discussed in Section \ref{sec:soundness}, $\sem{A} = \sem{A'}$ and $\sem{B} = \sem{B'}$, this implies that $A\vdash B$ is satisfied. Since D'.EAK is complete w.r.t.\ the final coalgebra semantics, a D'.EAK-derivation of  $A\vdash B$ exists.
\end{proof}

\commment{

\begin{fact}
Every rule of the Dynamic Calculus preserves the non-innocence of  sequents in derivations the conclusion of which is of type $\mathsf{FM}$.
\end{fact}
\begin{proof}
If a non-innocent sequent of type $\mathsf{FM}$ is derived by means of a derivation  $\pi$, then, by Lemma \ref{lemma:virtual free}, in $\pi$ there must be some application of Balance, Atom, Necessitation or Weakening. By Fact \ref{lemma: intro severe}, this implies that not only the sequent  is non-innocent but it is severe. Then by lemma \ref{lemma: severe preservation}, the severity is preserved and hence the non-innocence  is simply passed on along the derivation.
\end{proof}
}


\section{Conclusions and further directions}\label{Conclusion}



The present paper is part of a line of research  aimed at developing adequate proof-calculi for dynamic logics. These  logics have proven to be very difficult to treat with standard proof-theoretic tools, due to the very features which characterize them and make them applicable to diverse fields of science, spanning from artificial intelligence to social science and economics. A central desideratum in this line of research  is the development of methods which apply uniformly to different logics, and which allow a smooth transfer of results from one logic to another. The framework of display calculi has successfully met this desideratum for wide classes of logics in the family of  modal and substructural logics. In particular, in the framework of display calculi it is possible to state and prove metatheorems which guarantee any given proof system to enjoy the all-important  cut elimination property, provided it meets  certain conditions on its design.

The main contribution of the present paper is the definition of a display calculus which  smoothly encompasses the most proof-theoretically impervious  features of  Baltag Moss and Solecki's logic of epistemic actions and knowledge. Besides being well performing  (it adequately captures EAK and enjoys Belnap-style cut elimination), this calculus provides an interesting and in our opinion very promising {\em methodological} platform towards the uniform development of a general proof-theoretic account of all dynamic logics, 
and also, from a purely structurally proof-theoretic viewpoint, for clarifying and sharpening the formulation of  criteria leading to the statement and proof of meta-theoretic results such as  Belnap-style  cut-elimination, or conservativity issues.

\paragraph{Seminal approaches.}
\label{divide et impera}
The starting point of this methodology is to introduce enough syntactic devices, both at the operational and at the structural level, so that the parameters indexing logical connectives can be accounted for in the system  as {\em terms} in the language of choice. This gives rise to the definition of multi-type languages, endowed with connectives which manage the interaction of the different types. This  approach appears seminally  in both \cite{Alexandru} and \cite{Dyckhoff}; however, in neither  paper it is fully explored: in \cite{Alexandru} there is no theory of contexts governing the interaction of different types, and in \cite{Dyckhoff}, this interaction is clarified, but only at the metalinguistic level.

\paragraph{Multi-type calculus for PDL.} The multi-type approach has been applied to PDL in \cite{PDL}. In that setting, introducing two separate types for transitive actions and for general actions makes it possible to overcome the big hurdle given by the fact  that the induction axiom features occurrences of the same formula in precedent and in succedent position, which makes it is severely non amenable to the treatment in standard  display calculi. Another interesting case study is given by Parikh's game logic, which is ongoing work \cite{GL}.

\paragraph{Refinements of Belnap's conditions, and type-uniformity.}\label{refinements} The multi-type setting will hopefully prove to be conceptually advantageous to achieve a better grasp and a simpler statement of Wansing's and Belnap's  {\em regularity} requirements (cf.\ conditions $C_6/C_7$ in \cite{Wa98},
\cite[Section 2]{Be2}) for the Belnap-style cut elimination, via the notion of type-uniformity (Definition \ref{def:type-uniformity}). In  \cite{Belnap}, Belnap motivates his condition $C_7$\footnote{Recall that Belnap's condition $C_7$ corresponds to Wansing's  cons-regularity for formulas occurring in precedent position. An analogous explanation holds of course for the ant-regularity condition of formulas in succedent position.} saying that ``rules need not be wholly closed under substitution of structures for congruent formulas which are antecedent parts, but they must be closed enough.'' Then he explains that {\em closed enough} refers to the closure under substitution of formulas $A$ for structures $X$ such that a  derivation is available in the system for the sequent $X\vdash A$, in which the occurrence of $A$ in the conclusion is principal.  The crucial observation is that, even if a system is not defined {\em a priori} as multi-type, it can be regarded as a multi-type setting: indeed, the type of $A$ can be defined as consisting of all the structures $X$ such that the shape of derivation alluded to above exists. Then, condition  $C_6/C_7$  can be equivalently reformulated as the requirement that rules should be closed under uniform substitution within each type. Notice that, under the stipulations above, different types must be separated by at least one structural rule. For instance, in the Dynamic Calculus for EAK, the rule \emph{balance} separates \textsf{Fnc} from \textsf{Act}.
In conclusion, our conjecture is that Wansing's and Belnap's  conditions $C_6/C_7$ boil down to a type-uniformity requirement in a context in which types are not given explicitly. The observations above indicate that type-uniformity is  a desirable design requirement for general dynamic calculi, and in particular for the development of an adequate proof theory for dynamic logics, particularly in view of a uniform path  to  Belnap-style cut-elimination.

\paragraph{Non-proliferation.} Our analysis towards Belnap-style cut elimination led us to refine and weaken various aspects of the cut elimination metatheorem. For instance, the requirement of non-proliferation of parameters for quasi-properly displayable multi-type calculi applies only to types the grammar of which is rich enough that allows non-trivial cut applications, that is, applications of cut  the conclusion of which is different from both premises. The case study of EAK allows such a simple grammar on functional actions and agents  that these two types  are not subject to the restriction of non-proliferation. This in turn makes it possible to include the \emph{swap-in} rules in the calculus, in which every occurring parameter of a type which can proliferate  does indeed proliferate. Introducing  some nontrivial grammar on functional actions (e.g.\ sequential composition) would make the restriction of non-proliferation applicable to this type, and hence would make \emph{swap-in} not suitable anymore.

\paragraph{Expanding the signature.}  Notwithstanding the concerns about \emph{swap-in}, the multi-type language provides the opportunity to consider various natural expansions of the language of actions.  Early on, we argued that the connective ${\mbla}_{\! 1}$ which takes formulas in both coordinate as arguments and delivers an action, has the following natural interpretation: for all formulas $A, B$, the term $B {\mbla}_{\! 1}   A$ denotes the weakest epistemic action $\gamma$ such that, if $A$ was true before $\gamma$ was performed, then $B$ is true after any successful execution of $\gamma$. This connective seems particularly suited to explore epistemic capabilities and planning.

\section{Appendix}\label{Appendix}
    
In the following subsection, we collect the reduction steps verifying that the Dynamic Calculus verifies condition C'$_8$; in Section \ref{Completeness}, we collect the derivations which prove the syntactic completeness of the Dynamic Calculus w.r.t.\ IEAK (cf.\ Section \ref{ssec:IEAK}).

\subsection{Cut elimination}
\label{ssec:cut elimination}
Let us recall that C'$_8$ only concerns applications of the cut rules in which both occurrences of the given cut-term are {\em non parametric}.
Notice that non parametric occurrences of atomic terms of type $\mathsf{Fm}$ involve an axiom on at least  one premise, thus we are reduced to the following cases (the case of the constant $\bot$ is symmetric to the case of $\top$ and is omitted):

\begin{center}
{\footnotesize{
\begin{tabular}{@{}lcrclcr@{}}
\bottomAlignProof
\AX$\Phi p \fCenter p$
\AX$p \fCenter \Psi p$
\BI$\Phi p \fCenter \Psi p$
\DisplayProof
 & $\rightsquigarrow$ &
\bottomAlignProof
\AX$\Phi p \fCenter \Psi p$
\DisplayProof
\ \ \ &\ \ \
\bottomAlignProof
\AX$\textrm{I} \fCenter \top$
\AXC{\ \ $\vdots$ \raisebox{1mm}{$\pi$}}
\noLine
\UI$\textrm{I} \fCenter X$
\UI$\top \fCenter X$
\BI$\textrm{I} \fCenter X$
\DisplayProof
 & $\rightsquigarrow$ &
\bottomAlignProof
\AXC{\ \ $\vdots$ \raisebox{1mm}{$\pi$}}
\noLine
\UI$\textrm{I} \fCenter X$
\DisplayProof
 \\
\end{tabular}

}}
\end{center}
\noindent Notice that non parametric occurrences of any given (atomic) operational term $a$ of type $\mathsf{Fnc}$ or $\mathsf{Ag}$ are confined to axioms $a\vdash a$. Hence:
\begin{center}
\footnotesize{
\begin{tabular}{@{}lcr@{}}
\bottomAlignProof
\AX$a \fCenter a$
\AX$a \fCenter a$
\BI$a \fCenter a$
\DisplayProof
 & $\rightsquigarrow$ &
\bottomAlignProof
\AX$a \fCenter a$
\DisplayProof
 \\
\end{tabular}
}
\end{center}
\noindent In each case above, the cut in the original derivation is strongly uniform by assumption, and is eliminated by the transformation.
As to cuts on non atomic terms, let us restrict our attention to those cut-terms the main connective of which is $\mand_i, \mband_{\! i} \ , \mra_i, \mbra_i$  for $0\leq i \leq 3$ (the remaining operational connectives are straightforward and left to the reader). 

\begin{center}
\footnotesize{
\begin{tabular}{@{}rcl@{}}
\bottomAlignProof
\AXC{\ \ \ $\vdots$ \raisebox{1mm}{$\pi_0$}}
\noLine
\UI$x \fCenter a$
\AXC{\ \ \ $\vdots$ \raisebox{1mm}{$\pi_1$}}
\noLine
\UI$y\ \fCenter\ b$
\BI$x \mAND_{\! i} \  y\ \fCenter\  a \mand_i b$
\AXC{\ \ \ $\vdots$ \raisebox{1mm}{$\pi_2$}}
\noLine
\UI$ a\mAND_{\! i} \  b\ \fCenter\ z$
\UI$ a \mand_i b\ \fCenter\ z$
\BI$x\mAND_{\! i} \  y\ \fCenter\ z$
\DisplayProof
 & $\rightsquigarrow$ &
\bottomAlignProof
\AXC{\ \ \ $\vdots$ \raisebox{1mm}{$\pi_1$}}
\noLine
\UI$y\ \fCenter\ b$
\AXC{\ \ \ $\vdots$ \raisebox{1mm}{$\pi_0$}}
\noLine
\UI$x \fCenter a$

\AXC{\ \ \,$\vdots$ \raisebox{1mm}{$\pi_2$}}
\noLine
\UI$ a \mAND_{\! i} \  b\ \fCenter\ z$
\UI$ a \ \fCenter\ z \mSBLA_{\!\! i} \  b$
\BI$ x  \fCenter\ z \mSBLA_{\!\! i} \  b$

\UI$ x \mAND_{\! i} \  b\ \fCenter\ z$
\UI$b\ \fCenter\ x\mBRA_{\!\! i} \  z$
\BI$y\ \fCenter\ x \mBRA_{\!\! i} \  z$
\UI$x \mAND_{\! i} \  y\ \fCenter\ z$
\DisplayProof
 \\
\end{tabular}
}
\end{center}

\begin{center}
\footnotesize{

\begin{tabular}{@{}rcl@{}}
\bottomAlignProof

\AXC{\ \ \ $\vdots$ \raisebox{1mm}{$\pi_1$}}
\noLine
\UI$y\ \fCenter\ a \mRA_{\!\!i} \  b$
\UI$y\ \fCenter\ a \mra_{\!\!i} \ b$

\AXC{\ \ \ $\vdots$ \raisebox{1mm}{$\pi_0$}}
\noLine
\UI$x \fCenter a$
\AXC{\ \ \ $\vdots$ \raisebox{1mm}{$\pi_2$}}
\noLine
\UI$b\ \fCenter\ z$
\BI$a \mra_{\!\!i} \   b\ \fCenter\ x \mRA_{\!\!i} \  z$
\BI$y \fCenter x \mRA_{\!\!i} \   z$
\DisplayProof
 & $\rightsquigarrow$ &
\bottomAlignProof
\AXC{\ \ \ $\vdots$ \raisebox{1mm}{$\pi_0$}}
\noLine
\UI$x \fCenter a$
\AXC{\ \ \ $\vdots$ \raisebox{1mm}{$\pi_1$}}
\noLine
\UI$y\ \fCenter\ a \mRA_{\!\!i} \   b$
\UI$a \mBAND_{\! i} \  y \fCenter b$
\UI$a  \fCenter b \mSLA_{\!\! i} \  y$
\BI$x  \fCenter b \mSLA_{\!\! i} \  y$
\UI$x  \mBAND_{\! i} \  y \fCenter b$

\AXC{\ \ \ $\vdots$ \raisebox{1mm}{$\pi_2$}}
\noLine
\UI$b\ \fCenter\ z$
\BI$x \mBAND_{\! i} \  y\ \fCenter\ z$
\UI$y\ \fCenter\ x \mRA_{\!\!i} \ z$
\DisplayProof
 \\
\end{tabular}
}
\end{center}

\begin{center}
\footnotesize{

\begin{tabular}{@{}rcl@{}}
\bottomAlignProof
\AXC{\ \ \ $\vdots$ \raisebox{1mm}{$\pi_0$}}
\noLine
\UI$x \fCenter a$
\AXC{\ \ \ $\vdots$ \raisebox{1mm}{$\pi_1$}}
\noLine
\UI$y\ \fCenter\ b$
\BI$x \mBAND_{\! i} \  y\ \fCenter\  a \mband_{\! i} \  b$
\AXC{\ \ \ $\vdots$ \raisebox{1mm}{$\pi_2$}}
\noLine
\UI$ a\mBAND_{\! i} \  b\ \fCenter\ z$
\UI$ a \mband_{\! i} \  b\ \fCenter\ z$
\BI$x\mBAND_{\! i} \  y\ \fCenter\ z$
\DisplayProof
 & $\rightsquigarrow$ &
\bottomAlignProof
\AXC{\ \ \ $\vdots$ \raisebox{1mm}{$\pi_1$}}
\noLine
\UI$y\ \fCenter\ b$
\AXC{\ \ \ $\vdots$ \raisebox{1mm}{$\pi_0$}}
\noLine
\UI$x \fCenter a$

\AXC{\ \ \,$\vdots$ \raisebox{1mm}{$\pi_2$}}
\noLine
\UI$ a \mBAND_{\! i} \  b\ \fCenter\ z$
\UI$ a \ \fCenter\ z \mSLA_{\!\! i} \  b$
\BI$ x  \fCenter\ z \mSLA_{\!\! i} \  b$
\UI$ x \mBAND_{\! i} \  b\ \fCenter\ z$
\UI$b\ \fCenter\ x\mRA_{\!\! i} \  z$
\BI$y\ \fCenter\ x \mRA_{\!\! i} \  z$
\UI$x \mBAND_{\! i} \  y\ \fCenter\ z$
\DisplayProof
 \\
\end{tabular}
}
\end{center}

\begin{center}
\footnotesize{

\begin{tabular}{@{}rcl@{}}
\bottomAlignProof

\AXC{\ \ \ $\vdots$ \raisebox{1mm}{$\pi_1$}}
\noLine
\UI$y\ \fCenter\ a \mBRA_{\!\!i} \ b$
\UI$y\ \fCenter\ a \mbra_{\!\!i}  \  b$

\AXC{\ \ \ $\vdots$ \raisebox{1mm}{$\pi_0$}}
\noLine
\UI$x \fCenter a$
\AXC{\ \ \ $\vdots$ \raisebox{1mm}{$\pi_2$}}
\noLine
\UI$b\ \fCenter\ z$
\BI$a \mbra_{\!\!i}  \  b\ \fCenter\ x \mBRA_{\!\!i} \  z$
\BI$y \fCenter x \mBRA_{\!\!i} \   z$
\DisplayProof
 & $\rightsquigarrow$ &
\bottomAlignProof
\AXC{\ \ \ $\vdots$ \raisebox{1mm}{$\pi_0$}}
\noLine
\UI$x \fCenter a$
\AXC{\ \ \ $\vdots$ \raisebox{1mm}{$\pi_1$}}
\noLine
\UI$y\ \fCenter\ a \mBRA_{\!\!i}  \  b$
\UI$a \mAND_{\! i} \  y \fCenter b$
\UI$a  \fCenter b \mSBLA_{\!\! i} \  y$
\BI$x  \fCenter b \mSBLA_{\!\! i} \  y$
\UI$x  \mAND_{\! i} \  y \fCenter b$

\AXC{\ \ \ $\vdots$ \raisebox{1mm}{$\pi_2$}}
\noLine
\UI$b\ \fCenter\ z$
\BI$x \mAND_{\! i} \  y\ \fCenter\ z$
\UI$y\ \fCenter\ x \mBRA_{\!\!i} \ z$
\DisplayProof
 \\
\end{tabular}
}
\end{center}

\noindent In each  case above, the cut in the original derivation is strongly uniform by assumption, and after the transformation, cuts of lower complexity are introduced which can be easily verified to be strongly uniform for each $0\leq i\leq 3$.

\subsection{Completeness}\label{Completeness}


To prove the completeness of the Dynamic Calculus it is enough to show that all the axioms and rules of H.IEAK are theorems and, respectively, derived or admissible rules of Dynamic Calculus. Below we show the derivations of the dynamic axioms.

{\scriptsize{
\noindent
$\rule{126.2mm}{0.5pt}$
\begin{itemize}
\item $\alpha \mand p \dashv\vdash (\alpha \mand \top) \pand p$
\end{itemize}

\begin{center}
\begin{tabular}{@{}lr@{}}

\AX$\alpha \fCenter \alpha$
\AX$\textrm{I} \fCenter \top$
\BI$\alpha \mAND \textrm{I} \fCenter \alpha \mand \top$
\AX$\alpha \mAND p \fCenter p$
\BI$(\alpha \mAND \textrm{I}) \,; (\alpha \mAND p) \fCenter (\alpha \mand \top) \pand p$
\UI$\alpha \mAND (\textrm{I} \,; p) \fCenter (\alpha \mand \top) \pand p$
\UI$\textrm{I} \,; p \fCenter \alpha \mBRA (\alpha \mand \top) \pand p$
\UI$p \fCenter \alpha \mBRA (\alpha \mand \top) \pand p$
\UI$\alpha \mAND p \fCenter (\alpha \mand \top) \pand p$
\UI$\alpha \mand p \fCenter (\alpha \mand \top) \pand p$
\DisplayProof

 &

\AX$\alpha \fCenter \alpha$
\AX$\alpha \mBAND p \fCenter p$
\UI$\textrm{I} \fCenter (\alpha \mBAND p) > p$
\UI$\top \fCenter (\alpha \mBAND p) > p$
\UI$(\alpha \mBAND p) \,; \top \fCenter p$

\BI$\alpha \mAND ((\alpha \mBAND p) \,; \top) \fCenter \alpha \mand p$
\UI$(\alpha \mAND \top) \,; p \fCenter \alpha \mand p$
\UI$\alpha \mAND \top \fCenter \alpha \mand p < p$
\UI$\alpha \mand \top \fCenter p > \alpha \mand p$
\UI$\alpha \mand \top \,; p \fCenter \alpha \mand p$
\UI$(\alpha \mand \top) \pand p \fCenter \alpha \mand p$
\DisplayProof
 \\
\end{tabular}
\end{center}

\noindent
$\rule{126.2mm}{0.5pt}$
\begin{itemize}
\item$\alpha \mra p \dashv\vdash (\alpha \mand \top) \pra p$
\end{itemize}

\begin{center}
\begin{tabular}{@{}lr@{}}
\AX$\alpha \fCenter \alpha$
\AX$p \fCenter \alpha \mBRA p$
\BI$\alpha \mra p \fCenter \alpha \mRA (\alpha \mBRA p)$
\UI$\alpha \mBAND \alpha \mra p \fCenter \alpha \mBRA p$
\UI$\textrm{I} \fCenter (\alpha \mBAND \alpha \mra p) > \alpha \mBRA p$
\UI$\top \fCenter (\alpha \mBAND \alpha \mra p) > \alpha \mBRA p$
\UI$(\alpha \mBAND \alpha \mra p) \,; \top \fCenter \alpha \mBRA p$
\UI$\alpha \mAND ((\alpha \mBAND \alpha \mra p) \,; \top) \fCenter p$
\UI$\alpha \mra p \,; \alpha \mAND \top \fCenter p$
\UI$\alpha \mand \top \fCenter \alpha \mra p > p$
\UI$\alpha \mra p \,; \alpha \mand \top \fCenter p$
\UI$\alpha \mand \top \,; \alpha \mra p \fCenter p$
\UI$\alpha \mra p \fCenter \alpha \mand \top > p$
\UI$\alpha \mra p \fCenter (\alpha \mand \top) \pra p$
\DisplayProof

 &

\AX$\alpha \fCenter \alpha$
\AX$\textrm{I} \fCenter \top$
\BI$\alpha \mAND \textrm{I} \fCenter \alpha \mand \top$
\AX$p \fCenter \alpha \mRA p$
\BI$(\alpha \mand \top) \pra p \fCenter (\alpha \mAND \textrm{I}) > \alpha \mRA p$
\UI$(\alpha \mand \top) \pra p \fCenter \alpha \mRA (\textrm{I} > p)$
\UI$\alpha \mBAND (\alpha \mand \top) \pra p \fCenter \textrm{I} > p$
\UI$\alpha \mBAND (\alpha \mand \top) \pra p \fCenter p$
\UI$(\alpha \mand \top) \pra p \fCenter \alpha \mRA p$
\UI$(\alpha \mand \top) \pra p \fCenter \alpha \mra p$
\DisplayProof
 \\
\end{tabular}
\end{center}

\noindent
$\rule{126.2mm}{0.5pt}$
\begin{itemize}
\item $\lc\alpha\rc \top \dashv\vdash 1_\alpha \ \rightsquigarrow\  \alpha \mand \top \dashv\vdash \alpha \mand \top$
\end{itemize}

\begin{center}
\AX$\alpha \fCenter \alpha$
\AX$\top \fCenter \top$
\BI$\alpha \mAND \top \fCenter \alpha \mand \top$
\UI$\alpha \mand \top \fCenter \alpha \mand \top$
\DisplayProof
\end{center}

\newpage

\noindent
$\rule{126.2mm}{0.5pt}$
\begin{itemize}
\item $\alpha \mra \bot \dashv\vdash  \alpha \mand  \top \pra   \bot $
\end{itemize}

\begin{center}
\begin{tabular}{@{}lr@{}}
\AX$ \alpha \fCenter \alpha$

\AX$\bot \fCenter \textrm{I}$
\UI$\bot \fCenter \alpha \mBRA \textrm{I}$

\BI$\alpha \mra \bot \fCenter \alpha \mRA (\alpha\ \mBRA \textrm{I})$

\UI$ \alpha \mBAND ( \alpha \mra \bot )\fCenter \alpha\ \mBRA \textrm{I}$

\UI$ \textrm{I} \fCenter \alpha \mBAND ( \alpha \mra \bot ) > \alpha\ \mBRA \textrm{I}$

\UI$ \top \fCenter \alpha \mBAND ( \alpha \mra \bot ) > \alpha\ \mBRA \textrm{I}$

\UI$ \alpha \mBAND ( \alpha \mra \bot) \,;  \top \fCenter \alpha\ \mBRA \textrm{I}$

\UI$  \alpha \mAND (\alpha \mBAND ( \alpha \mra \bot) \,; \top)\fCenter \textrm{I}$

\UI$  \alpha \mra \bot \,; (\alpha \mAND \top)\fCenter \textrm{I}$
\UI$  \alpha \mra \bot \,; (\alpha \mAND \top)\fCenter \bot$
\UI$  \alpha \mAND \top \fCenter  \alpha \mra \bot  > \bot$

\UI$   \alpha \mand  \top \fCenter \alpha \mra \bot  > \bot$

\UI$  \alpha \mra \bot \,; \alpha \mand  \top \fCenter    \bot$

\UI$  \alpha \mand \top \,; \alpha \mra  \bot \fCenter    \bot$

\UI$  \alpha \mra \bot  \fCenter \alpha \mand  \top >   \bot$

\UI$  \alpha \mra \bot  \fCenter \alpha \mand  \top \pra   \bot$
\DisplayProof
 &
\AX$\alpha \fCenter \alpha$
\AX$\textrm{I} \fCenter \top$
\BI$\alpha \mAND \textrm{I} \fCenter \alpha \mand \top$
\AX$\bot \fCenter \textrm{I}$
\UI$\bot \fCenter \alpha \mRA \textrm{I}$
\UI$ \alpha \mBAND \bot \fCenter \textrm{I}$
\UI$ \alpha \mBAND \bot \fCenter \bot$
\UI$\bot \fCenter \alpha \mRA \bot$
\BI$\alpha \mand \top \rightarrow \bot \fCenter \alpha \mAND \textrm{I}  > \alpha\ \mRA \bot$
\UI$\alpha \mand \top \rightarrow \bot \fCenter \alpha \mRA( \textrm{I}  >  \bot)$
\UI$\alpha \mBAND( \alpha \mand \top \rightarrow \bot) \fCenter  \textrm{I}  >  \bot$

\UI$\textrm{I} ; \alpha \mBAND( \alpha \mand \top \rightarrow \bot) \fCenter    \bot$

\UI$\textrm{I}  \fCenter    \bot <  \alpha \mBAND( \alpha \mand \top \rightarrow \bot)$

\UI$\alpha \mBAND( \alpha \mand \top \rightarrow \bot) \fCenter    \bot$

\UI$ \alpha \mand \top \rightarrow \bot \fCenter  \alpha \mRA  \bot$

\UI$ \alpha \mand \top \rightarrow \bot \fCenter  \alpha \mra  \bot$

\DisplayProof
 \\
\end{tabular}
\end{center}

\noindent
$\rule{126.2mm}{0.5pt}$
\begin{itemize}
\item $\alpha \mand \bot \dashv\vdash \bot  $
\end{itemize}

\begin{center}
\begin{tabular}{@{}lr@{}}
\AX$\bot \fCenter \textrm{I}$
\UI$\bot \fCenter \alpha \mBRA \textrm{I}$
\UI$ \alpha \mAND \bot \fCenter \textrm{I}$
\UI$\alpha \mand \bot \fCenter \textrm{I}$
\UI$ \alpha \mand \bot \fCenter \bot$

\DisplayProof
 &
\AX$\bot \fCenter \textrm{I}$

\UI$\bot \fCenter \alpha \mand \bot$
\DisplayProof
 \\
\end{tabular}
\end{center}

\noindent
$\rule{126.2mm}{0.5pt}$
\begin{itemize}
\item $\alpha \mra \top \dashv\vdash \top $
\end{itemize}

\begin{center}
\begin{tabular}{@{}lr@{}}
\AX$\textrm{I} \fCenter \top$

\UI$\alpha \mra \top \fCenter \top$
\DisplayProof
 &
\AX$  \textrm{I} \fCenter \top$
\UI$ \alpha \mBAND \textrm{I} \fCenter  \top  $

\UI$  \textrm{I} \fCenter \alpha \mRA  \top  $
\UI$  \top \fCenter \alpha \mRA  \top  $

\UI$    \alpha \mBAND \top  \fCenter \top$
\UI$    \top  \fCenter \alpha \mRA  \top$
\UI$    \top  \fCenter \alpha \mra  \top$

\DisplayProof
 \\
\end{tabular}
\end{center}

\noindent
$\rule{126.2mm}{0.5pt}$
\begin{itemize}
\item $\alpha \mra (A\pand B) \dashv\vdash \alpha \mra A \pand \alpha \mra B$
\end{itemize}

\begin{center}
{\footnotesize{
\begin{tabular}{@{}lr@{}}
\AX$\alpha \fCenter \alpha$
\AX$A \fCenter A$
\UI$A\,; B \fCenter A$
\UI$A\pand B \fCenter A$
\BI$\alpha \mra (A\pand B) \fCenter \alpha \mRA A$
\UI$\alpha \mra (A\pand B) \fCenter \alpha \mra A$
\AX$\alpha \fCenter \alpha$
\AX$B \fCenter B$
\UI$A\,; B \fCenter B$
\UI$A\pand B \fCenter B$

\BI$\alpha \mra (A\pand B) \fCenter \alpha \mRA B$
\UI$\alpha \mra (A\pand B) \fCenter \alpha \mra B$
\BI$\alpha \mra (A\pand B)\,; \alpha \mra (A\pand B) \fCenter \alpha \mra A \pand \alpha \mra B$
\UI$\alpha \mra (A\pand B) \fCenter \alpha \mra A \pand \alpha \mra B$
\DisplayProof
 &
\AX$\alpha \fCenter \alpha$
\AX$A \fCenter A$
\BI$\alpha \mra A \fCenter \alpha \mRA A$
\UI$\alpha \mBAND (\alpha \mra A) \fCenter A$
\AX$\alpha \fCenter \alpha$
\AX$B \fCenter B$
\BI$\alpha \mra B \fCenter \alpha \mRA B$
\UI$\alpha \mBAND (\alpha \mra B) \fCenter B$
\BI$\alpha \mBAND (\alpha \mra A)\ ; \alpha \mBAND (\alpha \mra B) \fCenter A\pand B$
\UI$\alpha \mBAND (\alpha \mra A\,; \alpha \mra B) \fCenter A \pand B$
\UI$\alpha \mra A\,; \alpha \mra B \fCenter \alpha \mRA (A \pand B)$
\UI$\alpha \mra A \pand \alpha \mra B \fCenter \alpha \mra (A\pand B)$
\DisplayProof
 \\
\end{tabular}}}
\end{center}

\newpage

\noindent
$\rule{126.2mm}{0.5pt}$
\begin{itemize}
\item $\alpha \mand (A \pand B) \dashv\vdash \alpha \mand A \pand \alpha \mand B$
\end{itemize}

\begin{center}
{\scriptsize{
\begin{tabular}{@{}lr@{}}
\AX$\alpha \fCenter \alpha$
\AX$A \fCenter A$
\UI$A\,; B \fCenter A$
\UI$A \pand B \fCenter A$
\BI$\alpha \mAND A \pand B \fCenter \alpha \mand A$
\UI$\alpha \mand (A \pand B) \fCenter \alpha \mand A$

\AX$\alpha \fCenter \alpha$
\AX$B \fCenter B$
\UI$A\,; B \fCenter B$
\UI$A \pand B \fCenter B$

\BI$\alpha \mAND A \pand B \fCenter \alpha \mand B$
\UI$\alpha \mand (A \pand B) \fCenter \alpha \mand B$
\BI$\alpha \mand (A \pand B)\,; \alpha \mand (A\wedge B) \fCenter \alpha \mand A \wedge \alpha \mand B$
\UI$\alpha \mand (A \pand B) \fCenter \alpha \mand A \wedge \alpha \mand B$
\DisplayProof

 &
\!\!\!\!\!\!\!\!\!\!\!\!\!\!\!\!\!

\AX$\alpha \fCenter \alpha$
\AX$A \fCenter A$
\LeftLabel{\emph{balance}}
\UI$\alpha \mAND A \fCenter \alpha \mRA A$
\UI$\alpha \mBAND (\alpha \mAND A) \fCenter A$
\AX$B \fCenter B$
\RightLabel{\emph{balance}}
\UI$\alpha \mAND B \fCenter \alpha \mRA B$
\UI$\alpha \mBAND (\alpha \mAND B) \fCenter B$
\BI$\alpha \mBAND (\alpha \mAND A)\,; \alpha \mBAND (\alpha \mAND B) \fCenter A \pand B$
\UI$\alpha \mBAND (\alpha \mAND A\,; \alpha \mAND B) \fCenter A \pand B$

\BI$\alpha \mAND (\alpha \mBAND (\alpha \mAND A\,; \alpha \mAND B)) \fCenter \alpha \mand (A \pand B)$

\UI$\alpha \mBAND (\alpha \mAND A\,; \alpha \mAND B) \fCenter \alpha \mBRA ( \alpha \mand (A \pand B))$
\UI$\textrm{I}  \fCenter \alpha \mBAND (\alpha \mAND A\,; \alpha \mAND B) > \alpha \mBRA ( \alpha \mand (A \pand B)) $

\UI$   \alpha \mBAND (\alpha \mAND A\,; \alpha \mAND B)\ ; \textrm{I}\   \fCenter  \alpha \mBRA ( \alpha \mand (A \pand B)) $

\UI$ \alpha \mAND (  \alpha \mBAND (\alpha \mAND A\,; \alpha \mAND B)) ;  \textrm{I}) \fCenter  \alpha \mand (A \pand B) $
\RightLabel{\scriptsize{$conj$}}
\UI$ (\alpha \mAND A\,; \alpha \mAND B) ; \alpha \mAND  \textrm{I}\ \fCenter  \alpha \mand (A \pand B) $
\UI$ \alpha \mAND A\,; (\alpha \mAND B ; \alpha \mAND  \textrm{I}) \fCenter  \alpha \mand (A \pand B) $

\UI$  \alpha \mAND B ; \alpha \mAND  \textrm{I} \fCenter  \alpha \mAND A\,> \alpha \mand (A \pand B) $

\UI$  \alpha \mAND (B ;   \textrm{I}) \fCenter  \alpha \mAND A\,> \alpha \mand (A \pand B) $
\UI$  B ;   \textrm{I} \fCenter  \alpha \mBRA (\alpha \mAND A\,> \alpha \mand (A \pand B)) $
\UI$   \textrm{I} \fCenter  B > (\alpha \mBRA ( \alpha \mAND A\,> \alpha \mand (A \pand B )))$
\UI$     B \fCenter \alpha \mBRA ( \alpha \mAND A\,> \alpha \mand (A \pand B)) $

\UI$  \alpha \mAND    B \fCenter \alpha \mAND A\,> \alpha \mand (A \pand B) $
\UI$  \alpha \mand    B \fCenter \alpha \mAND A\,> \alpha \mand (A \pand B) $
\UI$  \alpha \mAND A\ ; \alpha \mand    B \fCenter \alpha \mand (A \pand B) $
\UI$  \alpha \mAND A\  \fCenter \alpha \mand (A \pand B) < \alpha \mand    B$
\UI$  \alpha \mand A\  \fCenter \alpha \mand (A \pand B) < \alpha \mand    B$
\UI$  \alpha \mand A\ ; \alpha \mand    B  \fCenter \alpha \mand (A \pand B) $
\UI$  \alpha \mand A\ \land \alpha \mand    B  \fCenter \alpha \mand (A \pand B) $
\DisplayProof
 \\
\end{tabular}}}
\end{center}

\noindent
$\rule{126.2mm}{0.5pt}$
\begin{itemize}
\item $\alpha \mand (A\vee B) \dashv\vdash \alpha \mand A\vee \alpha \mand B$
\end{itemize}

\begin{center}
{\scriptsize{
\begin{tabular}{@{}lr@{}}
\AX$\alpha \fCenter \alpha$
\AX$A \fCenter A$
\BI$\alpha \mAND A \fCenter \alpha \mand A$
\UI$A \fCenter \alpha \mBRA (\alpha \mand A)$

\AX$\alpha \fCenter \alpha$
\AX$B \fCenter B$
\BI$\alpha \mAND B \fCenter \alpha \mand B$
\UI$B \fCenter \alpha \mBRA (\alpha \mand B)$
\BI$A \por B \fCenter \alpha \mBRA (\alpha \mand A); \alpha \mBRA (\alpha \mand B)$
\UI$A \por B \fCenter \alpha \mBRA (\alpha \mand A\,; \alpha \mand B)$
\UI$\alpha \mAND A \por B \fCenter \alpha \mand A\,; \alpha \mand B$
\UI$\alpha \mand (A \por B) \fCenter \alpha \mand A\,; \alpha \mand B$
\UI$\alpha \mand (A \por B) \fCenter \alpha \mand A \por \alpha \mand B$
\DisplayProof

 &

\AX$\alpha \fCenter \alpha$

\AX$A \fCenter A$
\UI$A > A \fCenter B$
\UI$A \fCenter A\,; B$
\UI$A \fCenter A \por B$

\BI$\alpha \mAND A \fCenter \alpha \mand (A \por B)$
\UI$\alpha \mand A \fCenter \alpha \mand (A \por B)$

\AX$\alpha \fCenter \alpha$

\AX$B \fCenter B$
\UI$B < B\fCenter A $
\UI$B \fCenter A\,; B$
\UI$B \fCenter A \por B$

\BI$\alpha \mAND B \fCenter \alpha \mand (A \por B)$
\UI$\alpha \mand B \fCenter \alpha \mand (A \por B)$
\BI$\alpha \mand A \por \alpha \mand B \fCenter \alpha \mand (A \por B)\,; \alpha \mand (A \por B)$
\UI$\alpha \mand A \por \alpha \mand B \fCenter \alpha \mand (A \por B)$
\DisplayProof
 \\
\end{tabular}
}}
\end{center}

\newpage

\noindent
$\rule{126.2mm}{0.5pt}$
\begin{itemize}
\item $\alpha \mra (A\por B) \dashv\vdash (\alpha \mand \top) \pra (\alpha \mand A\por \alpha \mand B)$
\end{itemize}

\begin{center}
{\scriptsize{
\begin{tabular}{@{}lr@{}}
\AX$\alpha \fCenter \alpha$
\AX$\alpha \fCenter \alpha$
\AX$A \fCenter A$
\BI$\alpha \mAND A \fCenter \alpha \mand A$
\UI$A \fCenter \alpha \mBRA (\alpha \mand A)$
\AX$\alpha \fCenter \alpha$
\AX$B \fCenter B$
\BI$\alpha \mAND B \fCenter \alpha \mand B$
\UI$B \fCenter \alpha \mBRA (\alpha \mand B)$
\BI$A\por B \fCenter \alpha \mBRA (\alpha \mand A\,) ; \alpha \mBRA (\alpha \mand B)$
\UI$A\por B \fCenter \alpha \mBRA (\alpha \mand A\,; \alpha \mand B)$
\UI$\alpha \mAND (A\por B) \fCenter  \alpha \mand A\,; \alpha \mand B$
\UI$\alpha \mAND (A\por B) \fCenter  \alpha \mand A\, \por \alpha \mand B$
\UI$ A\por B \fCenter \alpha \mBRA  (\alpha \mand A\, \por \alpha \mand B)$

\BI$\alpha \mra (A\por B) \fCenter \alpha \mRA (\alpha \mBRA (\alpha \mand A\por \alpha \mand B))$

\UIC{$\alpha \mBAND (\alpha \mra (A\por B)) \fCenter \alpha \mBRA (\alpha \mand A\por \alpha \mand B)$}
\UIC{$\textrm{I} \fCenter \alpha \mBAND (\alpha \mra (A\por B)) > \alpha \mBRA (\alpha \mand A\por \alpha \mand B)$}
\UIC{$(\alpha \mBAND (\alpha \mra (A\por B)))\, ; \textrm{I}\fCenter \alpha \mBRA (\alpha \mand A\por \alpha \mand B)$}

\UIC{$ \alpha \mAND ((\alpha \mBAND (\alpha \mra (A\por B)))\, ; \textrm{I})\fCenter\alpha \mand A\por \alpha \mand B$}
\RightLabel{\emph{conj}$_0\!\mand$}

\UIC{$  (\alpha \mra (A\por B))\, ; \alpha \mAND \textrm{I}\fCenter\alpha \mand A\por \alpha \mand B$}
\UIC{$   \alpha \mAND \textrm{I}\fCenter \alpha \mra (A\por B) >  \alpha \mand A\por \alpha \mand B$}
\UIC{$  \textrm{I}\fCenter  \alpha \mRA ( \alpha \mra (A\por B) >  \alpha \mand A\por \alpha \mand B)$}
\UIC{$  \top \fCenter  \alpha \mRA ( \alpha \mra (A\por B) >  \alpha \mand A\por \alpha \mand B)$}
\UIC{$   \alpha \mAND \top \fCenter \alpha \mra (A\por B) >  \alpha \mand A\por \alpha \mand B$}
\UIC{$   \alpha \mand \top \fCenter \alpha \mra (A\por B) >  \alpha \mand A\por \alpha \mand B$}
\UIC{$ (\alpha \mra (A\por B))\, ;  (\alpha \mand \top) \fCenter  \alpha \mand A\por \alpha \mand B$}
\UIC{$ (\alpha \mand \top)\, ; (\alpha \mra (A\por B))  \fCenter  \alpha \mand A\por \alpha \mand B$}
\UIC{$ (\alpha \mra (A\por B))  \fCenter  (\alpha \mand \top) > \alpha \mand A\por \alpha \mand B$}
\UIC{$ (\alpha \mra (A\por B))  \fCenter  (\alpha \mand \top) \pra (\alpha \mand A\por \alpha \mand B)$}

\DisplayProof

 &
\!\!\!\!\!\!\!\!
\AX$\alpha \fCenter \alpha$
\AX$\textrm{I} \fCenter \top$
\BI$\alpha \mAND \textrm{I} \fCenter \alpha \mand \top$
\AX$A \fCenter A$

\UI$\alpha \mAND A \fCenter \alpha \mRA A$
\UI$\alpha \mand A \fCenter \alpha \mRA A$
\AX$B \fCenter B$
\UI$\alpha \mAND B \fCenter \alpha \mRA B$
\UI$\alpha \mand B \fCenter \alpha \mRA B$
\BI$\alpha \mand A \por \alpha \mand B \fCenter \alpha \mRA A\,; \alpha \mRA B$
\UI$\alpha \mand A \por \alpha \mand B \fCenter \alpha \mRA (A\,; B)$
\UI$\alpha \mBAND (\alpha\mand A \por \alpha \mand B) \fCenter A\,; B$
\UI$\alpha \mBAND (\alpha\mand A \por \alpha \mand B) \fCenter A\, \por B$
\UI$\alpha \mand A \por \alpha \mand B \fCenter \alpha \mRA (A \por B)$
\BI$( \alpha \mand \top) \pra (\alpha \mand A \por \alpha \mand B) \fCenter \alpha \mAND \textrm{I} > \alpha \mRA(A \por B)$

\UI$( \alpha \mand \top) \pra (\alpha \mand A \por \alpha \mand B) \fCenter \alpha \mRA( \textrm{I} > ( A \por B))$

\UI$\alpha \mBAND(( \alpha \mand \top) \pra (\alpha \mand A \por \alpha \mand B)) \fCenter  \textrm{I} > ( A \por B)$
\UI$ \textrm{I}\, ; \alpha \mBAND(( \alpha \mand \top) \pra (\alpha \mand A \por \alpha \mand B)) \fCenter  A \por B$
\UIC{$ \textrm{I}\,  \fCenter  (A \por B) < \alpha \mBAND(( \alpha \mand \top) \pra (\alpha \mand A \por \alpha \mand B))$}
\UI$ \alpha \mBAND(( \alpha \mand \top) \pra (\alpha \mand A \por \alpha \mand B)) \fCenter  A \por B$
\UI$( \alpha \mand \top) \pra (\alpha \mand A \por \alpha \mand B) \fCenter   \alpha \mRA(A \por B)$
\UI$( \alpha \mand \top) \pra (\alpha \mand A \por \alpha \mand B) \fCenter   \alpha \mra(A \por B)$

\DisplayProof
 \\
\end{tabular}
}}
\end{center}

\noindent
$\rule{126.2mm}{0.5pt}$
\begin{itemize}
\item $ \alpha\, \mand ( A \pra B) \dashv\vdash (\alpha \mand \top) \pand (\alpha \mand A \pra \alpha \mand B)$ 
\end{itemize}

\begin{center}
{\scriptsize{
\begin{tabular}{@{}lr@{}}
\!\!\!\!\!\!\!\!\!\!\!\!\!\!\!
\AX$\alpha \fCenter \alpha$
\AX$\textrm{I} \fCenter \top$
\BI$\alpha \mAND  \textrm{I} \fCenter \alpha \mand \top$

\AX$A \fCenter A$
\UI$\alpha \mAND A \fCenter \alpha \mRA A$
\UI$\alpha\mand A \fCenter \alpha \mRA A$
\UI$\alpha \mBAND (\alpha \mand A) \fCenter A$
\AX$\alpha \fCenter \alpha$

\AX$B \fCenter B$
\BI$\alpha \mAND B \fCenter \alpha \mand B$
\UI$B \fCenter \alpha \mBRA (\alpha \mand B)$
\BI$A \pra B \fCenter (\alpha \mBAND (\alpha \mand A)) > (\alpha \mBRA (\alpha \mand B))$
\UI$A \pra B \fCenter \alpha \mBRA (\alpha \mand A > \alpha \mand B)$
\UI$\alpha \mAND (A \pra B) \fCenter \alpha \mand A > \alpha \mand B$
\UI$\alpha \mAND (A \pra B) \fCenter \alpha \mand A \pra \alpha \mand B$
\BI$(\alpha\, \mAND \textrm{I})\, ; \alpha \mAND  (A \pra B) \fCenter (\alpha \mand \top) \pand (\alpha\mand A \pra \alpha\mand B)$
\UI$\alpha\, \mAND (\textrm{I}\, ; A \pra B) \fCenter (\alpha \mand \top) \pand (\alpha\mand A \pra \alpha\mand B)$
\UI$\textrm{I}\, ; A \pra B \fCenter \alpha\, \mBRA ( (\alpha \mand \top) \pand (\alpha\mand A \pra \alpha\mand B))$
\UIC{$\textrm{I}\,  \fCenter (\alpha\, \mBRA ( (\alpha \mand \top) \pand (\alpha\mand A \pra \alpha\mand B))) < (A \pra B)$}
\UI$ A \pra B \fCenter \alpha\, \mBRA ( (\alpha \mand \top) \pand (\alpha\mand A \pra \alpha\mand B))$
\UI$ \alpha\, \mAND ( A \pra B) \fCenter (\alpha \mand \top) \pand (\alpha\mand A \pra \alpha \mand B)$
\UI$ \alpha\, \mand ( A \pra B) \fCenter (\alpha \mand \top) \pand (\alpha \mand A \pra \alpha \mand B)$

\DisplayProof

 &

\!\!\!\!\!\!\!\!\!\!\!\!\!\!\!
\AX$\alpha \fCenter \alpha$
\AX$A \fCenter A$
\BI$\alpha \mAND A \fCenter \alpha \mand A$

\AX$B \fCenter B$
\UI$\alpha \mAND B \fCenter \alpha \mRA B$
\UI$\alpha \mand B \fCenter \alpha \mRA B$
\BI$\alpha \mand A \pra \alpha \mand B \fCenter \alpha \mAND A > \alpha \mRA B$

\UIC{$\alpha \mand A \pra \alpha \mand B \fCenter \alpha \mRA (A > B)$}
\UIC{$\alpha \mBAND (\alpha \mand A \pra \alpha \mand B) \fCenter A > B$}
\UIC{$\alpha \mBAND (\alpha \mand A \pra \alpha \mand B) \fCenter A \pra B$}
\UIC{$\alpha \mAND (\alpha \mBAND (\alpha \mand A \pra \alpha \mand B)) \fCenter  \alpha \mand (A \pra B)$}
\UIC{$\alpha \mBAND (\alpha \mand A \pra \alpha \mand B) \fCenter \alpha \mBRA ( \alpha \mand (A \pra B))$}
\UIC{$\textrm{I} \fCenter \alpha \mBAND (\alpha \mand A \pra \alpha \mand B)> \alpha \mBRA ( \alpha \mand (A \pra B))$}
\UIC{$\alpha \mBAND (\alpha \mand A \pra \alpha \mand B)\, ; \textrm{I}  \fCenter \alpha \mBRA ( \alpha \mand (A \pra B))$}
\UIC{$ \alpha \mAND(\alpha \mBAND (\alpha \mand A \pra \alpha \mand B)\, ; \textrm{I}  )\fCenter  \alpha \mand (A \pra B)$}
\UIC{$ \alpha \mand A \pra \alpha \mand B\, ;(\alpha \mAND \textrm{I}  )\fCenter  \alpha \mand (A \pra B)$}
\UIC{$\alpha \mAND \textrm{I}  \fCenter  \alpha \mand A \pra \alpha \mand B\, >  \alpha \mand (A \pra B)$}
\UIC{$ \textrm{I}  \fCenter \alpha \mBRA( \alpha \mand A \pra \alpha \mand B\, >  \alpha \mand (A \pra B))$}
\UIC{$ \top  \fCenter \alpha \mBRA( \alpha \mand A \pra \alpha \mand B\, >  \alpha \mand (A \pra B))$}
\UIC{$\alpha \mAND \top  \fCenter \alpha \mand A \pra \alpha \mand B\, >  \alpha \mand (A \pra B)$}
\UIC{$\alpha \mand \top  \fCenter \alpha \mand A \pra \alpha \mand B\, >  \alpha \mand (A \pra B)$}
\UIC{$\alpha \mand A \pra \alpha \mand B\, ; \alpha \mand \top  \fCenter   \alpha \mand (A \pra B)$}
\UIC{$ \alpha \mand \top\, ; \alpha \mand A \pra \alpha \mand B\,  \fCenter   \alpha \mand (A \pra B)$}
\UIC{$ \alpha \mand \top\, \land (\alpha \mand A \pra \alpha \mand B)  \fCenter   \alpha \mand (A \pra B)$}
\DisplayProof
 \\
\end{tabular}
}}
\end{center}

\noindent
$\rule{126.2mm}{0.5pt}$
\begin{itemize}
\item $   \alpha \mra (A \pra B)  \dashv\vdash  \alpha\mand A\, \pra \alpha\mand B $ 
\end{itemize}
\begin{center}
{\scriptsize{
\begin{tabular}{@{}lr@{}}
\AX$\alpha \fCenter \alpha$
\AX$A \fCenter A$
\UI$\alpha\mAND A \fCenter \alpha\mRA A$
\UI$\alpha \mBAND (\alpha\mAND A )\fCenter A$
\AX$\alpha \fCenter \alpha$
\AX$B \fCenter B$
\BI$\alpha\mAND B \fCenter \alpha\mand B$
\UI$B \fCenter \alpha\mBRA (\alpha\mand B)$
\BI$A \pra B \fCenter \alpha \mBAND (\alpha\mAND A) > \alpha \mBRA( \alpha\mand B)$
\UI$A \pra B \fCenter \alpha \mBRA (\alpha\mAND A > \alpha\mand B)$
\BI$\alpha \mra (A \pra B) \fCenter \alpha\mRA ( \alpha \mBRA (\alpha\mAND A > \alpha\mand B))$
\UI$\alpha\mBAND (\alpha \mra (A \pra B) \fCenter \alpha \mBRA (\alpha\mAND A > \alpha\mand B)$
\UI$\textrm{I} \fCenter \alpha\mBAND (\alpha \mra (A \pra B) > \alpha \mBRA (\alpha\mAND A > \alpha\mand B)$
\UI$\alpha\mBAND (\alpha \mra (A \pra B)\, ;\textrm{I} \fCenter \alpha \mBRA (\alpha\mAND A > \alpha\mand B)$
\UI$ \alpha \mAND (\alpha\mBAND (\alpha \mra (A \pra B)\, ;\textrm{I}) \fCenter\alpha\mAND A > \alpha\mand B$
\RightLabel{$conj$}
\UI$ \alpha \mra (A \pra B)\, ; (\alpha \mAND\textrm{I}) \fCenter\alpha\mAND A > \alpha\mand B$
\UI$ (\alpha \mAND\textrm{I})\, ; \alpha \mra (A \pra B) \fCenter\alpha\mAND A > \alpha\mand B$
\UI$\alpha\mAND A\, ; ((\alpha \mAND\textrm{I})\, ; \alpha \mra (A \pra B)) \fCenter \alpha\mand B$
\UI$(\alpha\mAND A\, ; (\alpha \mAND\textrm{I}))\, ; \alpha \mra (A \pra B) \fCenter \alpha\mand B$
\UI$\alpha\mAND A\, ; (\alpha \mAND\textrm{I})\,  \fCenter \alpha\mand B < \alpha \mra (A \pra B)$
\UI$\alpha\mAND (A\, ; \textrm{I})\,  \fCenter \alpha\mand B < \alpha \mra (A \pra B)$
\UI$A\, ; \textrm{I}\,  \fCenter \alpha\mBRA (\alpha\mand B < \alpha \mra (A \pra B))$
\UI$ \textrm{I}\,  \fCenter A\, > (\alpha\mBRA (\alpha\mand B < \alpha \mra (A \pra B)))$
\UI$A\,   \fCenter \alpha\mBRA (\alpha\mand B < \alpha \mra (A \pra B))$
\UI$ \alpha\mAND A\,   \fCenter \alpha\mand B < \alpha \mra (A \pra B)$
\UI$ \alpha\mand A\,   \fCenter \alpha\mand B < \alpha \mra (A \pra B)$
\UI$ \alpha\mand A\,   \fCenter \alpha\mand B < \alpha \mra (A \pra B)$
\UI$ \alpha\mand A\, ;  \alpha \mra (A \pra B)  \fCenter \alpha\mand B $
\UI$   \alpha \mra (A \pra B)  \fCenter \alpha\mand A\, > \alpha\mand B $
\UI$   \alpha \mra (A \pra B)  \fCenter \alpha\mand A\, \pra \alpha\mand B $

\DisplayProof

 &
\!\!\!\!\!\!\!\!\!\!
\AX$\alpha \fCenter \alpha$
\AX$A \fCenter A$
\BI$\alpha \mAND A \fCenter \alpha \mand A$

\AX$\alpha \fCenter \alpha$
\AX$B \fCenter B$
\BI$\alpha \mAND B \fCenter \alpha \mRA B$
\UI$\alpha \mand B \fCenter \alpha \mRA B$
\BI$\alpha \mand A\rightarrow \alpha \mand B \fCenter \alpha \mAND A > \alpha \mRA B$
\UI$\alpha \mand A\rightarrow \alpha \mand B \fCenter \alpha \mRA (A > B)$
\UI$\alpha \mBAND (\alpha \mand A\rightarrow \alpha \mand B) \fCenter A > B$
\UI$\alpha \mBAND (\alpha \mand A\rightarrow \alpha \mand B) \fCenter A \rightarrow B$
\UI$\alpha \mand A\rightarrow \alpha \mand B \fCenter \alpha \mRA A \rightarrow B$
\UI$\alpha \mand A\rightarrow \alpha \mand B \fCenter \alpha \mra (A \rightarrow B)$
\DisplayProof
 \\
\end{tabular}
}}
\end{center}
\newpage

\noindent
$\rule{126.2mm}{0.5pt}$
\begin{itemize}
\item $1_\alpha \pand \bigvee \{\lc\aga\rc \lc\beta\rc A\,|\,\alpha\aga\beta\} \vdash \lc\alpha\rc \lc \aga \rc A \rightsquigarrow (\alpha \mand \top\,) \land (\aga\mand ((\aga \mband \alpha)\mand A))\fCenter \alpha\mand (\aga \mand A)$
\end{itemize}
\begin{center}
{\footnotesize

\AX$\alpha \fCenter \alpha$
\AX$\aga \fCenter \aga$
\AX$A \fCenter A$
\BI$ \aga \mAND A \fCenter \aga \mand A$
\BI$\alpha \mAND (\aga \mAND A) \fCenter \alpha \mand (\aga \mand A)$
\UI$\aga \mAND A \fCenter \alpha \mBRA (\alpha \mand (\aga \mand A))$
\UI$ A \fCenter \aga \mBRA (\alpha \mBRA (\alpha \mand (\aga \mand A)))$
\RightLabel{swap-in$_R$}
\UI$ A \fCenter (\aga \mBAND \alpha) \mBRA (\aga \mBRA (  (\alpha \mAND \textrm{I}) > (\alpha \mand (\aga \mand A))  )  )$
\UI$  (\aga \mBAND \alpha) \mAND A \fCenter   \aga \mBRA (  (\alpha \mAND \textrm{I}) > (\alpha \mand (\aga \mand A))  )     $
\UI$  \aga \mBAND \alpha \fCenter   (\aga \mBRA (  (\alpha \mAND \textrm{I}) > (\alpha \mand (\aga \mand A))  )  )  \mBLA  A $
\UI$  \aga \mband \alpha \fCenter   (\aga \mBRA (  (\alpha \mAND \textrm{I}) > (\alpha \mand (\aga \mand A))  )  )  \mBLA  A $
\UI$  (\aga \mband \alpha) \mand A \fCenter   \aga \mBRA (  (\alpha \mAND \textrm{I}) > (\alpha \mand (\aga \mand A))  )     $
\UI$\aga \mAND (  (\aga \mband \alpha) \mand A ) \fCenter    (\alpha \mAND \textrm{I}) > (\alpha \mand (\aga \mand A))    $
\UI$\aga \mand (  (\aga \mband \alpha) \mand A ) \fCenter    (\alpha \mAND \textrm{I}) > (\alpha \mand (\aga \mand A))    $
\UI$(\alpha \mAND \textrm{I}) ; (\aga \mand (  (\aga \mband \alpha) \mand A )) \fCenter     \alpha \mand (\aga \mand A)  $
\UI$\alpha \mAND \textrm{I}  \fCenter    ( \alpha \mand (\aga \mand A) ) < (\aga \mand (  (\aga \mband \alpha) \mand A ))$
\UI$  \textrm{I}  \fCenter  \alpha  \mBRA (  ( \alpha \mand (\aga \mand A) ) < (\aga \mand (  (\aga \mband \alpha) \mand A ))  )$
\UI$  \top  \fCenter  \alpha  \mBRA (  ( \alpha \mand (\aga \mand A) ) < (\aga \mand (  (\aga \mband \alpha) \mand A ))  )$
\UI$\alpha \mAND \top  \fCenter   (  \alpha \mand (\aga \mand A) ) < (\aga \mand (  (\aga \mband \alpha) \mand A ))$

\UI$\alpha \mand \top \fCenter    ( \alpha \mand (\aga \mand A) ) < (\aga \mand (  (\aga \mband \alpha) \mand A ))$
\UI$(\alpha \mand \top ) ; (\aga \mand (  (\aga \mband \alpha) \mand A )) \fCenter     \alpha \mand (\aga \mand A)  $
\UI$(\alpha \mand \top ) \land (\aga \mand (  (\aga \mband \alpha) \mand A )) \fCenter     \alpha \mand (\aga \mand A)  $

\DisplayProof
}
\end{center}



$\rule{126.2mm}{0.5pt}$
\begin{itemize}
\item $[\alpha][\aga] A \vdash Pre(\alpha)\rightarrow \bigwedge \{[\aga][\beta] A\,|\,\alpha\aga\beta\} \rightsquigarrow \alpha \mra (\aga \mra A) \fCenter (\alpha \mand \top) \pra (\aga \mra ((\aga \mband \alpha)\mra A))$
\end{itemize}

\begin{center}
{\footnotesize
\AX$\alpha \fCenter \alpha$
\AX$\aga \fCenter \aga$
\AX$A \fCenter A$
\BI$\aga \mra A \fCenter \aga \mRA A$
\BI$\alpha \mra (\aga \mra A) \fCenter \alpha \mRA(\aga \mRA A )$
\UI$\alpha \mBAND(\alpha \mra (\aga \mra A)) \fCenter \aga \mRA A $
\UI$\aga \mBAND (\alpha \mBAND(\alpha \mra (\aga \mra A))) \fCenter  A $
\LeftLabel{\emph{swap-in}$_L$}
\UI$(\aga \mBAND \alpha) \mBAND (\aga \mBAND  (  
( \alpha \mAND \textrm{I} ) ;
( \alpha \mra (\aga \mra A)  )    )) \fCenter  A $
\UI$\aga \mBAND \alpha \fCenter  A \mLA (\aga \mBAND  (  
( \alpha \mAND \textrm{I} ) ;
( \alpha \mra (\aga \mra A)  )    ))  $
\UI$\aga \mband \alpha \fCenter  A \mLA (\aga \mBAND  (  
( \alpha \mAND \textrm{I} ) ;
( \alpha \mra (\aga \mra A)  )    ))  $
\UI$(\aga \mband \alpha) \mBAND (\aga \mBAND  (  
( \alpha \mAND \textrm{I} ) ;
( \alpha \mra (\aga \mra A)  )    )) \fCenter  A $
\UI$ \aga \mBAND  (  
( \alpha \mAND \textrm{I} ) ;
( \alpha \mra (\aga \mra A)  )    ) \fCenter (\aga \mband \alpha) \mRA  A $
\UI$ \aga \mBAND  (  
( \alpha \mAND \textrm{I} ) ;
( \alpha \mra (\aga \mra A)  )    ) \fCenter (\aga \mband \alpha) \mra  A $
\UI$  
( \alpha \mAND \textrm{I} ) ;
( \alpha \mra (\aga \mra A)  )     \fCenter \aga \mRA ( (\aga \mband \alpha) \mra  A ) $
\UI$  
( \alpha \mAND \textrm{I} ) ;
( \alpha \mra (\aga \mra A)  )     \fCenter \aga \mra ( (\aga \mband \alpha) \mra  A ) $
\UI$  
 \alpha \mAND \textrm{I}
    \fCenter (\aga \mra ( (\aga \mband \alpha) \mra  A ) )
    < ( \alpha \mra (\aga \mra A)  ) $
\UI$  
  \textrm{I}
    \fCenter \alpha \mBRA (  (\aga \mra ( (\aga \mband \alpha) \mra  A ) )
    < ( \alpha \mra (\aga \mra A)  ) ) $
\UI$  
  \top
    \fCenter \alpha \mBRA (  (\aga \mra ( (\aga \mband \alpha) \mra  A ) )
    < ( \alpha \mra (\aga \mra A)  ) ) $
\UI$  
 \alpha \mAND \top
    \fCenter (\aga \mra ( (\aga \mband \alpha) \mra  A ) )
    < ( \alpha \mra (\aga \mra A)  ) $
\UI$  
 \alpha \mand \top
    \fCenter (\aga \mra ( (\aga \mband \alpha) \mra  A ) )
    < ( \alpha \mra (\aga \mra A)  ) $
\UI$  
( \alpha \mand \top ) ;
( \alpha \mra (\aga \mra A)  )     \fCenter \aga \mra ( (\aga \mband \alpha) \mra  A ) $
\UI$  
( \alpha \mra (\aga \mra A)  )     \fCenter 
( \alpha \mand \top )  > ( \aga \mra ( (\aga \mband \alpha) \mra  A ) ) $
\UI$  
( \alpha \mra (\aga \mra A)  )     \fCenter 
( \alpha \mand \top )  \rightarrow ( \aga \mra ( (\aga \mband \alpha) \mra  A ) ) $

\DisplayProof
}
\end{center}


$\rule{126.2mm}{0.5pt}$
\begin{itemize}
\item $\lc\alpha\rc [\aga] A \vdash Pre(\alpha) \pand \bigwedge \{[\aga] \ls\beta\rs A\,|\, \alpha\aga\beta\} \rightsquigarrow 
\alpha \mand (\aga \mra A) 
\fCenter ( \alpha \mand \top) \land (    \aga \mra ( (\aga \mband \alpha)\mra A ) ) $
\end{itemize}

\begin{center}
{\scriptsize
\begin{tabular}{@{}c@{}}
\AX$\alpha \fCenter \alpha$
\AX$\textrm{I} \fCenter \top$
\BI$\alpha \mAND \textrm{I} \fCenter \alpha \mand \top $

\AX$\aga \fCenter \aga$
\AX$A \fCenter A$
\BI$\aga \mra A \fCenter \aga \mRA A$
\LeftLabel{\fns{$balance$}}
\UI$\alpha \mAND (\aga \mra A) \fCenter \alpha \mRA (\aga \mRA A) $
\UI$\alpha \mBAND (\alpha \mAND (\aga \mra A) )\fCenter \aga \mRA A $
\UI$ \aga \mBAND(\alpha \mBAND (\alpha \mAND (\aga \mra A) ))\fCenter A $
\LeftLabel{\fns{swap-in$_L$}}
\UI$ (\aga \mBAND \alpha)\mBAND (\aga \mBAND (
( \alpha \mAND \textrm{I} ) ;
( \alpha \mAND (\aga \mra A)   )  ))\fCenter A $
\UI$ \aga \mBAND \alpha \fCenter A 
\mBLA (\aga \mBAND (
( \alpha \mAND \textrm{I} ) ;
( \alpha \mAND (\aga \mra A)   )  ))$
\UI$ \aga \mband \alpha \fCenter A 
\mBLA (\aga \mBAND (
( \alpha \mAND \textrm{I} ) ;
( \alpha \mAND (\aga \mra A)   )  ))$
\UI$ (\aga \mband \alpha)\mBAND (\aga \mBAND (
( \alpha \mAND \textrm{I} ) ;
( \alpha \mAND (\aga \mra A)   )  ))\fCenter A $
\UI$  \aga \mBAND (
( \alpha \mAND \textrm{I} ) ;
( \alpha \mAND (\aga \mra A)   )  )
\fCenter (\aga \mband \alpha)\mRA A $
\UI$  \aga \mBAND (
( \alpha \mAND \textrm{I} ) ;
( \alpha \mAND (\aga \mra A)   )  )
\fCenter (\aga \mband \alpha)\mra A $
\UI$  
( \alpha \mAND \textrm{I} ) ;
( \alpha \mAND (\aga \mra A)   )  
\fCenter \aga \mRA ( (\aga \mband \alpha)\mra A )$
\UI$  
( \alpha \mAND \textrm{I} ) ;
( \alpha \mAND (\aga \mra A)   )  
\fCenter \aga \mra ( (\aga \mband \alpha)\mra A )$
\LeftLabel{$reduce'_L$}
\UI$  
 \alpha \mAND (\aga \mra A)    
\fCenter \aga \mra ( (\aga \mband \alpha)\mra A )$

\BI$( \alpha \mAND \textrm{I} ) ; ( \alpha \mAND (\aga \mra A)   )  
\fCenter ( \alpha \mand \top) \land (    \aga \mra ( (\aga \mband \alpha)\mra A ) ) $
\LeftLabel{$reduce'_L$}
\UI$\alpha \mAND (\aga \mra A) 
\fCenter ( \alpha \mand \top) \land (    \aga \mra ( (\aga \mband \alpha)\mra A ) ) $
\UI$\alpha \mand (\aga \mra A) 
\fCenter ( \alpha \mand \top) \land (    \aga \mra ( (\aga \mband \alpha)\mra A ) ) $
\DisplayProof
\end{tabular}
}
\end{center}


$\rule{126.2mm}{0.5pt}$
\begin{itemize}
\item $Pre(\alpha) \pand \bigwedge \{[\aga] \ls\beta\rs A\,|\, \alpha\aga\beta\} \vdash \lc\alpha\rc [\aga] A \rightsquigarrow (\alpha \mand \top) \pand  \aga \mra ((\aga \mband \alpha)\mra A) \fCenter \alpha\mand (\aga \mra A) $
\end{itemize}
\begin{center}
{\scriptsize
\AX$\alpha \fCenter \alpha$
\AX$\aga \fCenter \aga$
\AX$\aga \fCenter \aga$
\AX$\alpha \fCenter \alpha$
\BI$\aga \mBAND \alpha \fCenter \aga \mband \alpha$

\AX$A \fCenter A$
\BI$(\aga \mband \alpha) \mra A \fCenter (\aga \mBAND \alpha) \mRA A$
\BI$\aga \mra ((\alpha \mband \alpha) \mra A)\fCenter \aga \mRA((\aga \mBAND \alpha) \mRA A)$
\UI$\aga \mBAND(\aga \mra ((\alpha \mband \alpha) \mra A))\fCenter (\aga \mBAND \alpha) \mRA A$
\UI$(\aga \mBAND \alpha) \mBAND (\aga \mBAND(\aga \mra ((\alpha \mband \alpha) \mra A)))\fCenter  A$
\LeftLabel{\fns{swap-out$_L$}}
\UI$\aga \mBAND (\alpha \mBAND(\aga \mra ((\alpha \mband \alpha) \mra A)))\fCenter  A$
\UI$\alpha \mBAND(\aga \mra ((\alpha \mband \alpha) \mra A))\fCenter \aga \mRA  A$
\UI$\alpha \mBAND(\aga \mra ((\alpha \mband \alpha) \mra A))\fCenter \aga \mra  A$
\UI$\textrm{I} \fCenter  \alpha \mBAND(\aga \mra ((\alpha \mband \alpha) \mra A)) > \aga \mra  A $
\UI$\alpha \mBAND(\aga \mra ((\alpha \mband \alpha) \mra A))\, ; \textrm{I} \fCenter \aga \mra  A$
\BI$\alpha \mAND( \alpha \mBAND(\aga \mra ((\alpha \mband \alpha) \mra A))\, ; \textrm{I}) \fCenter \alpha \mand (\aga \mra  A)$
\LeftLabel{\fns{$conj_0 \! \mand$}}
\UI$ \aga \mra ((\alpha \mband \alpha) \mra A)\, ;  (\alpha \mAND \textrm{I}) \fCenter \alpha \mand (\aga \mra  A)$
\UI$ \alpha \mAND \textrm{I} \fCenter  \aga \mra ((\alpha \mband \alpha) \mra A)\, > \alpha \mand (\aga \mra  A)$
\UI$  \textrm{I} \fCenter  \alpha \mBRA (\aga \mra ((\alpha \mband \alpha) \mra A)\, > \alpha \mand (\aga \mra  A))$
\UI$  \top \fCenter  \alpha \mBRA (\aga \mra ((\alpha \mband \alpha) \mra A)\, > \alpha \mand (\aga \mra  A))$
\UI$ \alpha \mAND \top \fCenter  \aga \mra ((\alpha \mband \alpha) \mra A)\, > \alpha \mand (\aga \mra  A)$
\UI$ \alpha \mand \top \fCenter  \aga \mra ((\alpha \mband \alpha) \mra A)\, > \alpha \mand (\aga \mra  A)$
\UI$ \aga \mra ((\alpha \mband \alpha) \mra A)\, ; \alpha \mand \top \fCenter  \alpha \mand (\aga \mra  A)$

\UI$(\alpha \mand \top)\, ;  \aga \mra ((\aga \mband \alpha)\mra A) \fCenter \alpha\mand (\aga \mra A)$
\UI$(\alpha \mand \top) \pand  \aga \mra ((\aga \mband \alpha)\mra A) \fCenter \alpha\mand (\aga \mra A)$
\DisplayProof}
\end{center}
}
}

\newpage

\section{Cut Elimination for the Dynamic Calculus for EAK}
\label{Appendix : cut in MultiType DisplayCalculus for EAK}


Let us recall that C'$_8$ only concerns applications of the cut rules in which both occurrences of the given cut-term are {\em non parametric}.
Notice that non parametric occurrences of atomic terms of type $\mathsf{Fm}$ involve an axiom on at least  one premise, thus we are reduced to the following cases (the case of the constant $\bot$ is symmetric to the case of $\top$ and is omitted):
\begin{center}
\begin{tabular}{@{}lcrclcr@{}}
\bottomAlignProof
\AX$\Phi p \fCenter p$
\AX$p \fCenter \Psi p$
\BI$\Phi p \fCenter \Psi p$
\DisplayProof
 & $\rightsquigarrow$ &
\bottomAlignProof
\AX$\Phi p \fCenter \Psi p$
\DisplayProof
\ \ \ & \quad &\ \ \
\bottomAlignProof
\AX$\textrm{I} \fCenter \top$
\AXC{\ \ $\vdots$ \raisebox{1mm}{$\pi$}}
\noLine
\UI$\textrm{I} \fCenter X$
\UI$\top \fCenter X$
\BI$\textrm{I} \fCenter X$
\DisplayProof
 & $\rightsquigarrow$ &
\bottomAlignProof
\AXC{\ \ $\vdots$ \raisebox{1mm}{$\pi$}}
\noLine
\UI$\textrm{I} \fCenter X$
\DisplayProof
 \\
\end{tabular}
\end{center}

\noindent Notice that non parametric occurrences of any given (atomic) operational term $a$ of type $\mathsf{Fnc}$ or $\mathsf{Ag}$ are confined to axioms $a\vdash a$. Hence:
\begin{center}
\begin{tabular}{@{}lcr@{}}
\bottomAlignProof
\AX$a \fCenter a$
\AX$a \fCenter a$
\BI$a \fCenter a$
\DisplayProof
 & $\rightsquigarrow$ &
\bottomAlignProof
\AX$a \fCenter a$
\DisplayProof
 \\
\end{tabular}
\end{center}
\noindent In each case above, the cut in the original derivation is strongly uniform by assumption, and is eliminated by the transformation.
As to cuts on non atomic terms, let us restrict our attention to those cut-terms the main connective of which is $\mand_i, \mband_{\! i} \ , \mra_i, \mbra_i$  for $0\leq i \leq 3$ (the remaining operational connectives are straightforward and left to the reader). 

The following table disambiguates the various structural connectives occurring in the reduction steps represented below:

\renewcommand{\arraystretch}{1.3}
\begin{center}
\begin{tabular}{|p{1cm}<{\centering}|p{1cm}<{\centering}|p{1cm}<{\centering}|p{1cm}<{\centering}|p{1cm}<{\centering}|}
\cline{2-5}
\mc{1}{c|}{} & $i=0$ & $i=1$ & $i=2$ & $i=3$\\
\hline
$\mRA$ & $\mRA_0$ & $\mRA_1$ & $\mRA_2$ & $\mSRA_3$
\\
\hline
$\mBRA$ & $\mBRA_0$ & $\mBRA_1$ & $\mBRA_2$ & $\mSBRA_3$
\\ 
\hline
$\mLA$ & $\mSLA_0$ & $\mLA_1$ & $\mSLA_2$ & $\mSLA_3$
\\
\hline
$\mBLA$ & $\mSBLA_0$ & $\mBLA_1$ & $\mSBLA_2$ & $\mSBLA_3$
\\
\hline
\end{tabular}
\end{center}
\renewcommand{\arraystretch}{1}

\begin{center}
\begin{tabular}{@{}rcl@{}}
\bottomAlignProof
\AXC{\ \ \ $\vdots$ \raisebox{1mm}{$\pi_0$}}
\noLine
\UI$x \fCenter a$
\AXC{\ \ \ $\vdots$ \raisebox{1mm}{$\pi_1$}}
\noLine
\UI$y\ \fCenter\ b$
\BI$x \mAND_{\! i} \  y\ \fCenter\  a \mand_i b$
\AXC{\ \ \ $\vdots$ \raisebox{1mm}{$\pi_2$}}
\noLine
\UI$ a\mAND_{\! i} \  b\ \fCenter\ z$
\UI$ a \mand_i b\ \fCenter\ z$
\BI$x\mAND_{\! i} \  y\ \fCenter\ z$
\DisplayProof
 & $\rightsquigarrow$ &
\bottomAlignProof
\AXC{\ \ \ $\vdots$ \raisebox{1mm}{$\pi_1$}}
\noLine
\UI$y\ \fCenter\ b$
\AXC{\ \ \ $\vdots$ \raisebox{1mm}{$\pi_0$}}
\noLine
\UI$x \fCenter a$

\AXC{\ \ \,$\vdots$ \raisebox{1mm}{$\pi_2$}}
\noLine
\UI$ a \mAND_{\! i} \  b\ \fCenter\ z$
\UI$ a \ \fCenter\ z \mBLA_{\! i} \  b$
\BI$ x  \fCenter\ z \mBLA_{\! i} \  b$

\UI$ x \mAND_{\! i} \  b\ \fCenter\ z$
\UI$b\ \fCenter\ x\mBRA_{\!\! i} \  z$
\BI$y\ \fCenter\ x \mBRA_{\!\! i} \  z$
\UI$x \mAND_{\! i} \  y\ \fCenter\ z$
\DisplayProof
 \\
\end{tabular}
\end{center}

\begin{center}
%
\begin{tabular}{@{}rcl@{}}
\bottomAlignProof

\AXC{\ \ \ $\vdots$ \raisebox{1mm}{$\pi_1$}}
\noLine
\UI$y\ \fCenter\ a \mRA_{\!\!i} \  b$
\UI$y\ \fCenter\ a \mra_{\!\!i} \ b$

\AXC{\ \ \ $\vdots$ \raisebox{1mm}{$\pi_0$}}
\noLine
\UI$x \fCenter a$
\AXC{\ \ \ $\vdots$ \raisebox{1mm}{$\pi_2$}}
\noLine
\UI$b\ \fCenter\ z$
\BI$a \mra_{\!\!i} \   b\ \fCenter\ x \mRA_{\!\!i} \  z$
\BI$y \fCenter x \mRA_{\!\!i} \   z$
\DisplayProof
 & $\rightsquigarrow$ &
\bottomAlignProof
\AXC{\ \ \ $\vdots$ \raisebox{1mm}{$\pi_0$}}
\noLine
\UI$x \fCenter a$
\AXC{\ \ \ $\vdots$ \raisebox{1mm}{$\pi_1$}}
\noLine
\UI$y\ \fCenter\ a \mRA_{\!\!i} \   b$
\UI$a \mBAND_{\! i} \  y \fCenter b$
\UI$a  \fCenter b \mLA_{\! i} \  y$
\BI$x  \fCenter b \mLA_{\! i} \  y$
\UI$x  \mBAND_{\! i} \  y \fCenter b$

\AXC{\ \ \ $\vdots$ \raisebox{1mm}{$\pi_2$}}
\noLine
\UI$b\ \fCenter\ z$
\BI$x \mBAND_{\! i} \  y\ \fCenter\ z$
\UI$y\ \fCenter\ x \mRA_{\!\!i} \ z$
\DisplayProof
 \\
\end{tabular}
\end{center}

\begin{center}
%
\begin{tabular}{@{}rcl@{}}
\bottomAlignProof
\AXC{\ \ \ $\vdots$ \raisebox{1mm}{$\pi_0$}}
\noLine
\UI$x \fCenter a$
\AXC{\ \ \ $\vdots$ \raisebox{1mm}{$\pi_1$}}
\noLine
\UI$y\ \fCenter\ b$
\BI$x \mBAND_{\! i} \  y\ \fCenter\  a \mband_{\! i} \  b$
\AXC{\ \ \ $\vdots$ \raisebox{1mm}{$\pi_2$}}
\noLine
\UI$ a\mBAND_{\! i} \  b\ \fCenter\ z$
\UI$ a \mband_{\! i} \  b\ \fCenter\ z$
\BI$x\mBAND_{\! i} \  y\ \fCenter\ z$
\DisplayProof
 & $\rightsquigarrow$ &
\bottomAlignProof
\AXC{\ \ \ $\vdots$ \raisebox{1mm}{$\pi_1$}}
\noLine
\UI$y\ \fCenter\ b$
\AXC{\ \ \ $\vdots$ \raisebox{1mm}{$\pi_0$}}
\noLine
\UI$x \fCenter a$

\AXC{\ \ \,$\vdots$ \raisebox{1mm}{$\pi_2$}}
\noLine
\UI$ a \mBAND_{\! i} \  b\ \fCenter\ z$
\UI$ a \ \fCenter\ z \mLA_{\! i} \  b$
\BI$ x  \fCenter\ z \mLA_{\! i} \  b$
\UI$ x \mBAND_{\! i} \  b\ \fCenter\ z$
\UI$b\ \fCenter\ x\mRA_{\!\! i} \  z$
\BI$y\ \fCenter\ x \mRA_{\!\! i} \  z$
\UI$x \mBAND_{\! i} \  y\ \fCenter\ z$
\DisplayProof
 \\
\end{tabular}
\end{center}

\begin{center}
%
\begin{tabular}{@{}rcl@{}}
\bottomAlignProof

\AXC{\ \ \ $\vdots$ \raisebox{1mm}{$\pi_1$}}
\noLine
\UI$y\ \fCenter\ a \mBRA_{\!\!i} \ b$
\UI$y\ \fCenter\ a \mbra_{\!\!i}  \  b$

\AXC{\ \ \ $\vdots$ \raisebox{1mm}{$\pi_0$}}
\noLine
\UI$x \fCenter a$
\AXC{\ \ \ $\vdots$ \raisebox{1mm}{$\pi_2$}}
\noLine
\UI$b\ \fCenter\ z$
\BI$a \mbra_{\!\!i}  \  b\ \fCenter\ x \mBRA_{\!\!i} \  z$
\BI$y \fCenter x \mBRA_{\!\!i} \   z$
\DisplayProof
 & $\rightsquigarrow$ &
\bottomAlignProof
\AXC{\ \ \ $\vdots$ \raisebox{1mm}{$\pi_0$}}
\noLine
\UI$x \fCenter a$
\AXC{\ \ \ $\vdots$ \raisebox{1mm}{$\pi_1$}}
\noLine
\UI$y\ \fCenter\ a \mBRA_{\!\!i}  \  b$
\UI$a \mAND_{\! i} \  y \fCenter b$
\UI$a  \fCenter b \mBLA_{\! i} \  y$
\BI$x  \fCenter b \mBLA_{\! i} \  y$
\UI$x  \mAND_{\! i} \  y \fCenter b$

\AXC{\ \ \ $\vdots$ \raisebox{1mm}{$\pi_2$}}
\noLine
\UI$b\ \fCenter\ z$
\BI$x \mAND_{\! i} \  y\ \fCenter\ z$
\UI$y\ \fCenter\ x \mBRA_{\!\!i} \ z$
\DisplayProof
 \\
\end{tabular}
\end{center}

\noindent In each  case above, the cut in the original derivation is strongly uniform by assumption, and after the transformation, cuts of lower complexity are introduced which can be easily verified to be strongly uniform for each $0\leq i\leq 3$.

\bibliography{BIB}
\bibliographystyle{plain}

\end{document}